\documentclass[reqno]{amsart}
\usepackage[dvipsnames]{xcolor}
\usepackage{amsmath}
\usepackage{graphicx, color}
\usepackage{chngcntr}
\usepackage{apptools}
\usepackage{tikz}
\usetikzlibrary{backgrounds}
\usepackage{amsmath,amsfonts, amssymb}
\usepackage[normalem]{ulem}
\usepackage{wrapfig}
\usepackage{chngcntr}
\usepackage{paralist}
\usepackage{graphics} 
  \usepackage{epsfig} 
 \usepackage[colorlinks=true]{hyperref}
 \usepackage{graphicx}  \usepackage{epstopdf}
 \usepackage{xcolor}
 \usepackage{textcomp}
\usepackage{subcaption}
\usepackage{epstopdf}
\usepackage[UKenglish]{babel}
\usepackage{syntonly}
\usepackage{fancyhdr}
\usepackage{esint}
\usepackage[title]{appendix}
\usepackage{algorithm}
\usepackage[noend]{algpseudocode}
\usepackage{tabularx}

 \usepackage{tikz-3dplot}
 \usepackage{tikz}
 \usepackage{pgfplots}
 \usepackage{bm}
 \usepgfplotslibrary{fillbetween}
 \usetikzlibrary{backgrounds}
 \usetikzlibrary{patterns}
 \usetikzlibrary{scopes}
 \usetikzlibrary{shapes,arrows}
 \usetikzlibrary{plotmarks,intersections}
 \usetikzlibrary{calc}
 \usetikzlibrary{positioning}
 
 \usepackage{float}

 \usepackage{multirow}
 \usepackage{array}
 
\usetikzlibrary{calc}
\usepackage{pgfplots}
 
 \usepackage{bbm}
\hypersetup{urlcolor=blue, citecolor=red}

\algrenewcommand\algorithmicrequire{\textbf{Precondition:}}

\makeatletter
\newcommand{\multiline}[1]{%
	\begin{tabularx}{\dimexpr\linewidth-\ALG@thistlm}[t]{@{}X@{}}
		#1
	\end{tabularx}
}
\makeatother

\usepackage{array}
\newcolumntype{N}{@{}m{0pt}@{}}
\makeatletter
\newcommand*{\centerfloat}{%
  \parindent \z@
  \leftskip \z@ \@plus 1fil \@minus \textwidth
  \rightskip\leftskip
  \parfillskip \z@skip}
\makeatother

\def\mathrlap{\mathpalette\mathrlapinternal}
\def\mathrlapinternal#1#2{%
	\rlap{$\mathsurround=0pt#1{#2}$}}

\newtheorem{theorem}{Theorem}[section]
\newtheorem{corollary}{Corollary}

\newtheorem{lemma}[theorem]{Lemma}

\newtheorem{prop}[theorem]{Proposition}

\newtheorem{rem}{Remark}[section]

\numberwithin{equation}{section}

\algrenewcommand\algorithmicrequire{\textbf{Precondition:}}

\usepackage{array}
\newcolumntype{N}{@{}m{0pt}@{}}
\def\mathrlap{\mathpalette\mathrlapinternal}
\def\mathrlapinternal#1#2{%
	\rlap{$\mathsurround=0pt#1{#2}$}}

\newcommand{\dv}{\textrm{d}\boldsymbol{v}}
\newcommand{\Lv}{L_{\boldsymbol{v}}}
\newcommand{\OmegaLv}{\Omega_{\Lv}}
\newcommand{\Rd}{\mathbb{R}^d}
\newcommand{\Nd}{\mathbb{N}^d}
\newcommand{\Chi}{\mbox{\Large$\chi$}}
\newcommand{\gchi}{\mbox{\Large$\chi$} g}
\newcommand{\hchi}{\mbox{\Large$\chi$} h}
\newcommand{\Meq}{\mathcal{M}_{eq}}
\newcommand{\delchi}{\delta_{\chi}}

\title[]{Convergence and Error Estimates for the Conservative Spectral Method for Fokker-Planck-Landau Equations}

\author[Clark A. Pennie  and Irene M. Gamba]{}

\begin{document}
\maketitle		
	
	\centerline{\scshape  Clark A. Pennie  and Irene M. Gamba}
\medskip
{\footnotesize
	\centerline{Department of Mathematics and Oden Institute }
	\centerline{University of Texas at Austin}
	\centerline{2515 Speedway Stop C1200
		Austin, Texas, 78712-1202, USA}

}

\date{}
\setcounter{tocdepth}{2}
\setcounter{secnumdepth}{2}

	
	
	
	
	\begin{abstract}	
	Error estimates are rigorously derived for a semi-discrete version of a conservative spectral method for approximating the space-homogeneous Fokker-Planck-Landau (FPL) equation associated to hard potentials.  The analysis included shows that the semi-discrete problem has a unique solution with bounded moments.  In addition, the derivatives of such a solution up to any order also remain bounded in $L^2$ spaces globally time, under certain conditions.  These estimates, combined with control of the spectral projection, are enough to obtain error estimates to the analytical solution and convergence to equilibrium states.  It should be noted that this is the first time that an error estimate has been produced for any numerical method which approximates FPL equations associated to any range of potentials.

	\medskip
	

	\end{abstract}

	


\section{Introduction}
An important model for plasmas is the Landau equation, which results from the grazing collision limit of the Boltzmann equation.  This limit, first derived by Landau \cite{LandauEq}, assumes that colliding particles are travelling almost parallel to each other due to repulsive Coulomb forces.  

A more mathematical description of the limit was detailed by Degond and Lucquin-Desreux \cite{Degond&LD}, Desvillettes \cite{Desvillettes92, Desvillettes15}, Villani \cite{VillaniLandau} and Desvillettes and Villani \cite{DesvillettesVillani}, even for extended potential rates higher than Coulomb interactions and up to hard spheres.  When rates different to Coulomb interactions are used, the equation is referred to as being of Fokker-Planck-Landau type.  Computationally, the limiting problem has been studied by Bobylev and Potapenko \cite{BobylevPotapenko}, using Monte Carlo methods, and in Fourier space by Haack and Gamba \cite{GambaHaack1, GambaHaack2}.

The Landau equation is rather difficult to model, either analytically or numerically, due to the high dimensionality, non-linearity and non-locality.  For numerical simulations, a deterministic scheme can be used, such as the conservative spectral method, developed by Zhang and Gamba \cite{Chenglong}, which is the model of choice for the current work.  The method described in \cite{Chenglong} is in fact a solver for the space-inhomogeneous Landau equation, coupled to Poisson's equation, where the advection is modeled by a discontinuous Galerkin scheme, but here only the space-homogeneous version is considered.  Qualitative evidence of the strength of this scheme has been provided by the present authors \cite{RGDPaper, EntropyDecayPaper} by showing that it can capture the correct decay rate to equilibrium for a range of potentials, from Coulomb to hard sphere interactions.  This was exemplified by simulations of entropy decay for both the space-homogeneous equation and by coupling to Poisson's equation in the mean-field limit.

Even less is known analytically about Fokker-Planck-Landau equations than the Boltzmann equation.  For hard potentials, many important results regarding existence and uniqueness, as well as moment and entropy estimates, can be found in the work of Desvillettes and Villani \cite{D&V_1, D&V_2}.  They use entropy methods to prove convergence to equilibrium at a rate faster than any polynomial in time  but this was improved to an exponential rate by Carrapatoso \cite{Carrapatoso_hard}.  

In \cite{D&V_2} there is an explicit lower bound for the spectral gap of the linearised collision operator associated to hard potentials, which improves on the original results by Degond and Lemou \cite{Degond&Lemou}.  Long-time behaviour of linearised Fokker-Planck-Landau equations has also been considered by several authors, including Guo \cite{Guo} and Baranger and Mouhot \cite{Baranger&Mouhot}.  For soft potentials, Carrapatoso \cite{Carrapatoso_soft} used entropy methods to prove convergence to equilibrium at a rate faster than any polynomial for the fully nonlinear equation and exponentially in the linear setting.

Spectral methods as an approximating  model for the space-homogeneous Landau equation were first considered by Pareschi et al.\ \cite{Pareschi_Spectral}, and later by Filbet and Pareschi\cite{Filbet_InhomSpectral} and Crouseilles and Filbet \cite{Crouseilles_Spectral}, but did not preserve the conservation properties of the Landau equation.  This limited the ability of these schemes to compute accurate dynamics for long time approximations to the Maxwell-Boltzmann equilibrium. 

The version of the spectral method in this work exploits the weak form of the Landau equation in order to calculate the Fourier transform of the collision operator.  It does so in just $\mathcal{O}(N^3 \log{N})$ operations, where the number of Fourier modes $N$ can be rather small, thanks to the conservation enforcement.  No periodising is required either, once again thanks to conservation.  Instead, a cut-off domain in velocity space is used, within which the majority of the solution's mass should be supported, based on a result by Gamba et al.\ \cite{BoltzmannSupport} for the Boltzmann equation.  This general construction was first applied to the Boltzmann equation by Gamba and Tharkabhushaman \cite{Harsha} and the details for the derivation of the Landau equation scheme can be found in the paper by Zhang and Gamba \cite{Chenglong}.

It is important to note that, despite the importance of deterministic methods for simulating Fokker-Planck-Landau type equations, there has been no analysis into the convergence of the approximations of these schemes to the true solutions.  Error estimates were developed, however, by Alonso et.\ al \cite{BoltzmannConvergence} for the spectral method used to solve the Boltzmann equation.  These will be closely followed when deriving error estimates for Fokker-Planck-Landau type equations, and coupled with the analytical theory developed by Desvillettes and Villani \cite{D&V_1}.

The focus of the present work is to obtain an error estimate on the conservative spectral method approximation to the space-homogeneous Landau equation associated to hard potentials.  The problem being approximated is described in Section \ref{Landau_Introdcution}.  All work presented here is actually carried out on a semi-discrete version, which is introduced in Section \ref{semi_discrete_section}.  The general idea to produce an error estimate on the solution of that semi-discrete problem begins with proving that the distance between the conserved and unconserved collision operators is small in $L_2$-norm, which is achieved in Section \ref{Conservation_Estimate}.  This fact is essential for every result that follows.  The next step is to prove that the moments and $L_2$-norm of the semi-discrete problem propagate in time, which is carried out in Section \ref{Propagation} after obtaining an estimate on the derivatives of these quantities.  In that section it is also shown that the negative part of the numerical approximation can be controlled.

It is then shown in Section \ref{ExistenceAndRegularity} that the semi-discrete problem has a unique solution which also satisfies the propagation of moments and $L^2$-norm and in Section \ref{Regularity} the regularity of that solution is considered by bounding its $H^s$ norm, for any $s > 0$.  This work is concluded by deriving the $L^2$ error estimate in Section \ref{L2_Estimate} and proving that the approximation will always converge to the appropriate equilibrium Maxwellian, provided it is close enough, in Section \ref{Long_Time_Behaviour}.  Before starting this analysis, however, some notation must be introduced in Section \ref{Notation} and useful results from previous authors on Fokker-Planck-Landau type equations stated in Section \ref{Previous_Results}.

\section{The Landau Equation} \label{Landau_Introdcution}
The initial value problem associated to a space-homogeneous Fokker-Planck-Landau type equation is to find $f(t, \boldsymbol{v})$, where $(t, \boldsymbol{v}) \in (\mathbb{R}^+, \Rd)$, for dimension $d \in \mathbb{N}^+$, such that
\begin{equation}
f_t(t, \boldsymbol{v}) = Q(f,f)(t, \boldsymbol{v}), \label{Landau_homo_Kn1}
\end{equation}
where $f(0, \boldsymbol{v}) = f_0(\boldsymbol{v})$ and $Q(f,f)$ is the collision operator defined as
\begin{flalign}
	&& Q(f,f) := \nabla_{\boldsymbol{v}}\cdot \int_{\mathbb{R}^3} & S(\boldsymbol{v} - \boldsymbol{v}_*) (f_*\nabla_{\boldsymbol{v}}f - f \nabla_{\boldsymbol{v}_*}f_*)~\textrm{d} \boldsymbol{v}_*, && \label{Q_FPL} \\
	\textrm{for} && S(\boldsymbol{u}) &:= |\boldsymbol{u}|^{\lambda + 2}\left(\textrm{I} - \frac{\boldsymbol{uu}^T}{|\boldsymbol{u}|^2}\right), && \nonumber
\end{flalign}
with $I \in \mathbb{R}^{d \times d}$ the identity matrix and the subscript notation $f_*$ meaning evaluation at $\boldsymbol{v_*}$ (the velocity of a colliding particle).  In general, $\lambda > 0$  corresponds to hard potentials and $\lambda < 0$ to soft potentials.  Here, $\lambda = 1$ models hard sphere interactions; $\lambda = 0$ is known as a Maxwell type interaction; and $\lambda = -3$ models Coulomb interactions between particles.  For the current work, the attention is on $0 \leq \lambda \leq 1$.

\subsection{Properties of Fokker-Planck-Landau Type Equations}
An important identity associated to the Fokker-Planck-Landau operator (\ref{Q_FPL}) is the weak form.  After multiplying the operator (\ref{Q_FPL}) by a sufficiently smooth function $\phi$ and integrating over $\Rd$, along with some variable changes and an application of the divergence theorem, the weak form is given by
\begin{equation}
	\int_{\mathbb{R}^d} Q(f,f) \phi ~\textrm{d}\boldsymbol{v} = \int_{\mathbb{R}^{2d}} (\nabla_{\boldsymbol{v}_*}\phi_* - \nabla_{\boldsymbol{v}}\phi)^T S(\boldsymbol{v} - \boldsymbol{v}_*) f_* \nabla_{\boldsymbol{v}}f ~\textrm{d}\boldsymbol{v}_* \textrm{d}\boldsymbol{v}. \label{Q_FPL_weak1}
\end{equation}
Furthermore, since Fokker-Planck-Landau type equations are a limit of the Boltzmann equation, they enjoy the same conservation laws.  In particular, for the set of collision invariants $\left\{\phi_k(\boldsymbol{v})\right\}_{k=0}^{d+1} := \left\{1, v_1, \ldots, v_d, |\boldsymbol{v}|^2\right\}$,
\begin{equation}
	\int_{\mathbb{R}^d} Q(f, f)(\boldsymbol{v}) \phi_k (\boldsymbol{v}) ~\textrm{d}\boldsymbol{v} = 0, ~~~~~\textrm{for } k = 0, 1, \ldots, d. \label{conservation_L}
\end{equation}
This is important because it leads to the conservation of mass $\rho$, average velocity $\boldsymbol{V}$ and kinetic energy $T$, where each of these quantities are found via
\begin{equation*}
\rho := \int_{\mathbb{R}^d} f(t, \boldsymbol{v}) ~\textrm{d}\boldsymbol{v}, ~~~\boldsymbol{V} := \frac{1}{\rho} \int_{\mathbb{R}^d} f(t, \boldsymbol{v}) \boldsymbol{v} ~\textrm{d}\boldsymbol{v}
\end{equation*}
\begin{flalign*}
\textrm{and} && T := \frac{1}{d\rho} \int_{\Rd} f(t, \boldsymbol{v}) |\boldsymbol{v}|^2 ~\textrm{d}\boldsymbol{v}. && 
\end{flalign*}

These moments will always be conserved for the single-species space-homo-geneous Fokker-Planck-Landau type equation (\ref{Landau_homo_Kn1}).  They will also be conserved for the space-inhomogeneous version when solved with appropriate boundary conditions, such as reflective or periodic conditions.  If the initial mass, average velocity and kinetic energy are denoted by $\rho_0$, $\boldsymbol{V}_0$ and $T_0$, respectively, the equilibrium solution of the Landau equation is a Gaussian distribution with the same moments, namely
\begin{equation}
	\mathcal{M}_{eq}(\boldsymbol{v}) := \frac{\rho_0}{(2\pi T_{0})^{\frac{d}{2}}} e^{-\frac{|\boldsymbol{v} - \boldsymbol{V}_0|^2}{2 T_{0}}}. \label{M_eq_hom}
\end{equation}

Initially $f(0, \boldsymbol{v}) = f_0(\boldsymbol{v})$ and it is assumed that $\textrm{supp}f \Subset \Omega_{\boldsymbol{v}}$, since $f$ should have sufficient decay in velocity-space \cite{BoltzmannSupport} and $\Omega_{\boldsymbol{v}} \subset \Rd$ is chosen depending on the initial data.  It should still be that $\boldsymbol{v} \in \mathbb{R}^d$ but values of $f$ are negligible outside a sufficiently large ball.  The initial data is then extended by zero outside the computational domain, which means it can be controlled by $e^{-c|\boldsymbol{v}|^2}$, for $c > 0 $ depending on the moments of $f_0$.  Under such conditions, it is expected that the computational solution will remain supported on $\Omega_{\boldsymbol{v}}$ up to a fixed small error that depends on the initial data.  More details on this choice of domain can be seen in Subsection \ref{Domain_Choice}.

\subsection{The Spectral Method for Space-homogeneous Fokker-Planck-Landau Type Equations}
The method used for solving the space-homogeneous Fokker-Planck-Landau type equation (\ref{Landau_homo_Kn1}) is referred to as spectral because it exploits the weighted convolution structure of the collision operator (\ref{Q_FPL}) to compute its values using the fast Fourier transform (FFT).  To see this, first note that by choosing the Fourier transform kernel $\phi(\boldsymbol{v}) = (2\pi)^{-\frac{d}{2}}e^{-i \boldsymbol{\xi}\cdot\boldsymbol{v}}$ in the weak form identity (\ref{Q_FPL_weak1}), the Fourier transform of $Q$ at a given Fourier mode $\boldsymbol{\xi} \in \mathbb{R}^d$ is
\begin{align*}
	\hat{Q}(\boldsymbol{\xi}) &= (2\pi)^{-\frac{d}{2}} \int_{\mathbb{R}^d} Q(f,f) e^{-i \boldsymbol{\xi}\cdot\boldsymbol{v}} ~\textrm{d}\boldsymbol{v} \\
	&= (2\pi)^{-\frac{d}{2}} \int_{\mathbb{R}^{2d}} \Bigl(\nabla_{\boldsymbol{v}_*}\left(e^{-i \boldsymbol{\xi}\cdot\boldsymbol{v}_*} \right) - \nabla_{\boldsymbol{v}}\left(e^{-i \boldsymbol{\xi}\cdot\boldsymbol{v}} \right)\Bigr)^T S(\boldsymbol{v} - \boldsymbol{v}_*) f_* \nabla_{\boldsymbol{v}}f ~\textrm{d}\boldsymbol{v}_* \textrm{d}\boldsymbol{v}.
\end{align*}
By evaluating the derivatives of the exponential functions and carrying out many manipulations which take advantage of properties of the Fourier transform (the details of which can be found in \cite{Chenglong}), this expression for the Fourier transform $\hat{Q}(\boldsymbol{\xi})$ reduces to
\begin{equation}
	\hat{Q}\left(\boldsymbol{\xi}\right) = \int_{\mathbb{R}^d}\hat{f}\left(\boldsymbol{\xi} - \boldsymbol{\omega}\right)\hat{f} (\boldsymbol{\omega})\Bigl({\boldsymbol{\omega}}^T\hat{S}~\left(\boldsymbol{\omega}\right)\boldsymbol{\omega} ~-~ {\left(\boldsymbol{\xi} - \boldsymbol{\omega}\right)}^T\hat{S}~\left(\boldsymbol{\omega}\right)\left(\boldsymbol{\xi} - \boldsymbol{\omega}\right)\Bigr)~\textrm{d}\boldsymbol{\omega}, \label{qHat0}
\end{equation}
where $\hat{S}$ is the Fourier transform of $S$ given in (\ref{Q_FPL}).

It should be noted that the Fourier transformed collision operator $\hat{Q}$ retains the convolution structure and can be written as
\begin{equation}
	\hat{Q}\left(\boldsymbol{\xi}\right) = \int_{\mathbb{R}^d}\hat{f}\left(\boldsymbol{\xi} - \boldsymbol{\omega}\right)\hat{f} (\boldsymbol{\omega}) \hat{G}_L(\boldsymbol{\xi}, \boldsymbol{\omega}) ~\textrm{d}\boldsymbol{\omega}, \label{qHat_G}
\end{equation}
for some weight $G_L$.  When applying similar manipulations to find the Fourier transform of the Boltzmann collision operator, it can be written in the same form as (\ref{qHat_G}) but with a different weight, say, $G_B$.  This is the basis behind the proof of convergence to the Boltzmann collision operator to the Landau one in Fourier space described by Gamba and Haack \cite{GambaHaack1, GambaHaack2}, as the problem reduces to proving that the weight $G_B$ converges to $G_L$.

To save time computationally, the weight $\hat{G}_L$ can be precomputed.  The evaluation of $\hat{Q}$ then only involves an application of the FFT to find $\hat{f}$ and a weighted convolution which can be calculated by some quadrature method.  Due to the speed of the FFT, this can therefore be completed in $O(N^d \log(N))$ operations for the number of Fourier modes $N$ in each velocity dimension.  The final step is to take the inverse fast Fourier transform (IFFT) to recover $N^d$ many discrete evaluations of $Q$.

\subsection{Choosing a Computational Domain} \label{Domain_Choice}
For computational purposes, the evaluation of the integral in expression (\ref{qHat0}) for $\hat{Q}$ cannot be carried out over all of $\mathbb{R}^d$ and a cut-off domain must be defined, say $\Omega_{\boldsymbol{\xi}} \subset \mathbb{R}^d$, which is paired with an associated velocity domain $\Omega_{\boldsymbol{v}} \subset \mathbb{R}^d$.  This also means that the Landau collision operator must be approximated on the computational domain $\Omega_{\boldsymbol{v}}$.  In doing so, the conservative property of the Landau equation is lost because the integrals (\ref{conservation_L}) which enforce the conservation no longer use the collision operator $Q$ defined on the entire space $\Rd$.  This point is addressed by the conservation method but the effects of introducing a bounded domain can be minimised by choosing it large enough.  

In particular, as was proven by Gamba et al.\ \cite{BoltzmannSupport} for the Boltzmann equation, if the initial data $f_0(\boldsymbol{v})$ is bounded by a Gaussian distribution then there is a slightly larger Gaussian distribution for which the solution $f(t, \boldsymbol{v})$ is supported underneath for any time $t > 0$.  This means that if the domain is chosen large enough then, not only will the approximation remain bounded, it also decays with Gaussian tails and so it can be guaranteed that the majority of the mass and energy remain contained in $\Omega_{\boldsymbol{v}}$ during any simulation.

More precisely, in order to choose the domain, assume that $f_0$ satisfies
\begin{equation}
	f_0(\boldsymbol{v}) \leq \frac{C_0 \rho_0}{(2 \pi T_0)^{\frac{d}{2}}} e^{-\frac{r_0 |\boldsymbol{v}|^2}{2 T_0}}, \label{Gaussian_bound_init}
\end{equation}
for constants $r_0 \in (0, 1]$ and $C_0 \geq 1$ which stretch and dilate the Gaussian, as well as the initial mass $\rho_0 = \int_{\Rd} g_0(\boldsymbol{v}) ~\dv$ and temperature $T_0 = \frac{1}{d\rho_0} \int_{\Rd} g_0(\boldsymbol{v}) |\boldsymbol{v}|^2 ~\dv$ (where $\boldsymbol{V} = \boldsymbol{0}$ is chosen without loss of generality).  Then, based on the result of \cite{BoltzmannSupport}, there are new constants $r \in (0, r_0]$ and $C \geq C_0$ which depend on the initial data and the potential $\lambda$ (as well as the cross-section $b$ for the result on the Boltzmann equation) such that, for all $t > 0$,
\begin{equation}
	f(\boldsymbol{v}, t) \leq \frac{C \rho_0}{(2 \pi T_0)^{\frac{d}{2}}} e^{-\frac{r |\boldsymbol{v}|^2}{2 T_0}}. \label{Gaussian_bound}
\end{equation}

Then, since the tails of the Gaussian (\ref{Gaussian_bound}) have negligible mass and energy, a positive value $\delta_M \ll 1$ can be chosen such that
\begin{equation*}
	\int_{\Omega_{\boldsymbol{v}}^c} f(t, \boldsymbol{v}) \langle \boldsymbol{v} \rangle^{2} ~\dv \leq \int_{\Omega_{\boldsymbol{v}}^c} \frac{C \rho_0}{(2 \pi T_0)^{\frac{d}{2}}} e^{-\frac{r |\boldsymbol{v}|^2}{2 T_0}} \langle \boldsymbol{v} \rangle^{2} ~\dv \leq \delta_M \int_{\Omega_{\boldsymbol{v}}} |f_0(\boldsymbol{v})| \langle \boldsymbol{v} \rangle^{2} ~\dv.
\end{equation*}
This means that $\delta_M$ can be considered as a tolerance for what portion of the initial mass and energy is lost by using the bounded domain $\Omega_{\boldsymbol{v}}$.

\subsection{The Conservation Routine}
Even if a larger domain $\Omega_{\boldsymbol{v}}$ will give a more accurate approximation to $\hat{Q}$, and therefore $Q$, some amount of error is unavoidable whenever truncating the velocity domain.  Conservation can be enforced, however, by considering a constrained minimisation problem.  Given a collection of discrete values of the collision operator $Q$ resulting from the spectral method, say $\{\tilde{Q}_n\}_{n=1}^{N^d}$, a new set of values $\{Q_n\}_{n=1}^{N^d}$ must be found which are as close as possible to the original values in $\ell^2$-norm but satisfy the discrete form of (\ref{conservation_L}).  This discrete form replaces the integrals in (\ref{conservation_L}) with quadrature sums and can be written as
\begin{flalign*}
	&& && \sum_{n=1}^{N^d} Q_n (\phi_k)_n \omega_n = 0, && \textrm{for } k = 0, 1, \ldots, d+1,
\end{flalign*}
where $\{(\phi_k)_n\}_{k=0}^{d+1}$ are evaluations of the $d+2$ collision invariants at the same discrete point where $Q_n$ is evaluated and $\omega_n$ is the corresponding quadrature weight for that point.  If the discrete values of $Q$ are stored in the vector $\boldsymbol{Q}$ of length $N^d$, this discrete conservation can be written as 
\begin{flalign}
	A \boldsymbol{Q} = \boldsymbol{0}, && \textrm{where } A_{k,n} := (\phi_{k-1})_n \omega_n, ~~~~\textrm{for } k = 1,2\ldots,d+2, ~n = 1,2,\ldots N^d. \label{A_entries}
\end{flalign}

Then, given $\tilde{\boldsymbol{Q}} = (\tilde{Q}_1, \tilde{Q}_2, \ldots, \tilde{Q}_{N^d})$, the least squares problem is to find the vector $\boldsymbol{Q} = (Q_1, Q_2, \ldots, Q_{N^d})$ of conserved evaluations of the collision operator which solves
\begin{flalign}
	&& \min_{\boldsymbol{Q} \in \mathbb{R}^{N^d}} \bigl|\bigl|\tilde{\boldsymbol{Q}} - \boldsymbol{Q} \bigr|\bigr|^2_{\ell^2}, && \textrm{such that } A \boldsymbol{Q} = \boldsymbol{0}. && \label{least_squares_v}
\end{flalign}
This can then be solved as a $(d+2)$-dimensional Lagrange multiplier problem by defining the operator
\begin{equation*}
	\mathbb{L}(\boldsymbol{Q},\boldsymbol{\gamma}) := \sum_{n=1}^{N^d} \bigl(\tilde{Q}_n - Q_n \bigr)^2 - \boldsymbol{\gamma}^T A \boldsymbol{Q}.
\end{equation*}
By solving $\nabla_{\boldsymbol{Q}} \mathbb{L} = \boldsymbol{0}$ for the Lagrange multiplier $\boldsymbol{\gamma}$, the discrete values of the conserved collision operator are found to be
\begin{flalign}
	&& \boldsymbol{Q} = \Lambda(A) \tilde{\boldsymbol{Q}} && \textrm{where } \Lambda(A) := I - A^T(A A^T)^{-1} A. && \label{conserve_velocity}
\end{flalign}

This means that the conservation is simply matrix-vector multiplication.  The details of the derivation of $\Lambda(A)$ can be found in \cite{Harsha} and \cite{Chenglong} for the Boltzmann and Landau equations, respectively, but it should be noted that $\Lambda(A)$ is identical for both equations.  The full algorithm of the conservative spectral method for solving the space-homogeneous Fokker-Planck-Landau type equation (\ref{Landau_homo_Kn1}) when conserving in velocity space is then given in Algorithm \ref{conservation_algorithm_velocity}.

\begin{algorithm}
	\caption{The conservative spectral method for solving the space-homogeneous Fokker-Planck-Landau type equation (\ref{Landau_homo_Kn1}) when conserving in velocity space
		\label{conservation_algorithm_velocity}}
	\begin{algorithmic}[1]
		\Require{$\boldsymbol{F}$ contains evaluations of $f$ on the uniform velocity grid at a given time-step $t_n$}
		\Statex
		\For{each step in Runge-Kutta}
		\State Calculate the FFT of $\boldsymbol{F}$ and store the values in $\hat{\boldsymbol{F}}$ \Comment{$\mathcal{O}(N^d \log{N})$}
		\State \multiline{Calculate $\hat{Q}(\hat{\boldsymbol{F}})$ at each point in the uniform Fourier space grid using identity (\ref{qHat0}) and store the values in $\hat{\boldsymbol{Q}}$ \Comment{$\mathcal{O}(N^d)$}}
		\State Calculate the IFFT of $\hat{\boldsymbol{Q}}$ and store the values in $\tilde{\boldsymbol{Q}}$ \Comment{$\mathcal{O}(N^d \log{N})$}
		\State Set $\boldsymbol{Q} = \Lambda(A) \hat{\boldsymbol{Q}}$ as in (\ref{conserve_velocity}), with $A$ given in (\ref{A_entries}) \Comment{$\mathcal{O}(N^d)$}
		\State Perform the iteration of Runge-Kutta to update $\boldsymbol{F}$ \Comment{$\mathcal{O}(N^d)$}
		\EndFor
	\end{algorithmic}
\end{algorithm}	  

\section{Notation} \label{Notation}
\subsection{Spaces}
The error estimate will be obtained by reproducing the error analysis performed by Alonso et al.\ \cite{BoltzmannConvergence} on the space-homogeneous Boltzmann equation.  For this reason, a lot of the notation from their work will be introduced here.  First, it is convenient to use the angle brackets $\langle \boldsymbol{v} \rangle := \sqrt{1 + |\boldsymbol{v}|^2}$, for $\boldsymbol{v} \in \Rd$.  Then, for a measurable set $\Omega \subseteq \Rd$, $p \geq 1$ and $k \in \mathbb{R}$, the Lebesgue spaces $L^p_k(\Omega)$ are defined by
\begin{equation*}
L^p_k(\Omega) := \left\{ f: \|f\|_{L^p_k(\Omega)} < \infty \right\},
\end{equation*}
where
\begin{equation}
\|f\|_{L^p_k(\Omega)} := \begin{cases}
\left(\int_{\Omega} |f(\boldsymbol{v})\langle \boldsymbol{v} \rangle^{k}|^p ~\dv \right)^{\frac{1}{p}}, &\textrm{when } 1 \leq p < \infty, \\
\textrm{esssup } |f(\boldsymbol{v})\langle \boldsymbol{v} \rangle^{k}|, &\textrm{when } p = \infty.
\end{cases} \label{L2k_def}
\end{equation}

In addition, for $s \in \mathbb{N}$, the Hilbert spaces $H^s_k(\Omega)$ are then defined as 
\begin{equation}
H^s_k(\Omega) := \left\{ f: \|f\|_{H^s_k(\Omega)} < \infty \right\}, ~~~~~\textrm{where } \|f\|_{H^s_k(\Omega)} := \left(\sum_{\substack{\alpha \in \Nd: \\ |\alpha| \leq s}} \|D^{\alpha} f\|_{L^2_k(\Omega)}^2 \right)^{\frac{1}{2}}. \label{Hsk_def}
\end{equation}
Note that the Hilbert space $H^{s}$ is the completion of the infinitely smooth functions $C^{\infty}(\omega)$ by the norm $\|f\|_{H^s(\Omega)} = \|\langle \boldsymbol{\xi} \rangle^s \hat{f}\|_{L^2(\Omega)}$ and is sometimes referred to as an \textit{inhomogeneous Sobolev space}.  The corresponding \textit{homogeneous Sobolev space}, denoted $\dot{H}^{s}$, will sometimes be used as well and is the completion of $C^{\infty}(\Omega)$ by the norm $\|f\|_{\dot{H}^s(\Omega)} = \|\,|\boldsymbol{\xi}|^s \hat{f}\|_{L^2(\Omega)}$.

Most of the time, $\Omega$ is simply $\OmegaLv := (-\Lv, \Lv)^d$ or $\Rd$ and $d = 3$.  Note also that the classical $L^p$ and $H^{\alpha}$ norms satisfy $\|f\|_{L^p(\Omega)} = \|f\|_{L^p_0(\Omega)}$ and $\|f\|_{H^{\alpha}(\Omega)} = \|f\|_{H^{\alpha}_0(\Omega)}$.  The special case of the $L^1_k(\Rd)$ norm is referred to as the $k$th \textit{moment} of a function $f$, for $k \in \mathbb{R}$, denoted $m_k(f)$, which means
\begin{equation}
m_k(f) := \int_{\Rd} |f(\boldsymbol{v})| \langle \boldsymbol{v} \rangle^{k} ~\dv. \label{mk_def}
\end{equation}

\subsection{Fourier Series}
In order to discuss the convergence of the approximation to the true solution, it is necessary to introduce the truncated Fourier series.  To begin with, for a function $f \in L^1(\mathcal{U})$ defined on an open set $\mathcal{U} \subset \Rd$, the formulation used for the Fourier transform of $f$ here is
\begin{equation*}
\hat{f}(\boldsymbol{\xi}) := \frac{1}{(2\pi)^{\frac{d}{2}}} \int_{U} f(\boldsymbol{v}) e^{-i \boldsymbol{\xi} \cdot \boldsymbol{v}} \dv.
\end{equation*}

Then, for $\mathcal{U} = \OmegaLv = (-\Lv, \Lv)^d$, if $f \in L^2(\OmegaLv)$, the Fourier series of $f$ is given by
\begin{equation*}
f(\boldsymbol{v}) = \frac{1}{(2\Lv)^{d}} \sum_{k \in \mathbb{Z}^d} \hat{f}(\boldsymbol{\xi}_k) e^{i \boldsymbol{\xi}_k \cdot \boldsymbol{v}},
\end{equation*}
where $\boldsymbol{\xi}_k = \frac{2\pi}{\Lv} k$ for the multi-index $k = (k_1, k_2, \ldots, k_d)$.  Furthermore, the orthogonal projection onto the first $(2N+1)^d$ basis elements is found through the \textit{mode projection operator} $\Pi^N_{\Lv}: L^2(\OmegaLv) \to L^2(\OmegaLv)$, where $\Pi^N_{\Lv} f$ is defined by
\begin{equation*}
\Pi^N_{\Lv} f (\boldsymbol{v}) := \frac{1}{(2\Lv)^{d}} \sum_{\substack{k \in \mathbb{Z}^d: \\ |k_i| \leq N}} \hat{f}(\boldsymbol{\xi}_k) e^{i \boldsymbol{\xi}_k \cdot \boldsymbol{v}}.
\end{equation*}

At this point it is convenient to notice that, for a multi-index $\alpha \in \Nd$, if $f \in H^{|\alpha|}_0(\OmegaLv)$, by using the properties of derivatives of Fourier transforms,
\begin{align}
D^{\alpha} \Pi^N_{\Lv} f (\boldsymbol{v}) &= \frac{1}{(2\Lv)^{d}} \sum_{\substack{k \in \mathbb{Z}^d: \\ |k_i| \leq N}} (i \boldsymbol{\xi})^{\alpha} \hat{f}(\boldsymbol{\xi}_k) e^{i \boldsymbol{\xi}_k \cdot \boldsymbol{v}} \nonumber \\
&= \frac{1}{(2\Lv)^{d}} \sum_{\substack{k \in \mathbb{Z}^d: \\ |k_i| \leq N}} \widehat{D^{\alpha}f}(\boldsymbol{\xi}_k) e^{i \boldsymbol{\xi}_k \cdot \boldsymbol{v}} = \Pi^N_{\Lv} D^{\alpha} f (\boldsymbol{v}). \label{derivative_commuting}
\end{align}
This means that $D^{\alpha} \Pi^N_{\Lv} = \Pi^N_{\Lv} D^{\alpha}$ and so the mode projection operator commutes with differentiation.  In addition, by  Parseval's theorem, 
\begin{equation*}
\| \Pi^N_{\Lv} f\|_{L^2(\OmegaLv)} \leq \|f\|_{L^2(\OmegaLv)}, ~~~~~\textrm{for any } N \in \mathbb{N}, 
\end{equation*}
with equality when $N = \infty$, and 
\begin{equation*}
\|f - \Pi^N_{\Lv} f\|_{L^2(\OmegaLv)} \to 0, ~~~\textrm{as } N \to \infty.
\end{equation*}

\subsection{Extension Operators}
Now, restricting the original velocity domain of $\Rd$ to $\OmegaLv$ introduces some complications at the boundary as a result of the solution being truncated.  To handle these issues, a scaled cut-off function \mbox{\Large$\chi$} is introduced, given by
\begin{equation}
\Chi (\boldsymbol{v}) := \varphi \left(\frac{\boldsymbol{v}}{\Lv}\right), \label{chi_def}
\end{equation}
where $\varphi: \Rd \to [0,1]$ is any smooth non-negative function such that supp($\varphi$) $\subset$ \linebreak $(1 - \frac{1}{5}\delchi) [-1, 1]^d$ and $\varphi(\boldsymbol{v}) = 1$ when $\boldsymbol{v} \in (1 - \delchi) [-1, 1]^d$, for some small $0 < \delchi \ll 1$.  Whereas the introduction of this cut-off function is not necessary for $L^2$ estimates, it is for the higher order Sobolev regularity results since it smooths the boundary in such a way that does not introduce a significant amount of error.  In particular, for any function $g \in H^{s}(\OmegaLv)$, by the product rule,
\begin{equation}
\|\gchi\|_{H^{s}(\OmegaLv)} \leq C_{\chi} \|g\|_{H^{s}(\OmegaLv)}, \label{extension_bound}
\end{equation}
where $C_{\chi} \geq \|\Chi \|_{\mathcal{C}^{s}}$ can be taken as independent of $\Lv \geq 1$.  

Furthermore, for any function $g \in H^{s}(\OmegaLv)$, the restriction of $g$ by \mbox{\Large$\chi$} can be viewed as a function defined on all of $\Rd$ with values extended by zero outside of $\OmegaLv$, which gives
\begin{equation*}
\|\gchi\|_{H^{s}(\Rd)} = \|\gchi\|_{H^{s}(\OmegaLv)}.
\end{equation*}
As a result, the cut-off function \mbox{\Large$\chi$} can actually be viewed as an extension operator from $H^{s}(\OmegaLv)$ to $H^{s}(\Rd)$.  This will be useful when needing to compare the semi-discrete approximation (defined on $\OmegaLv$) with the true continuous solution (defined on all of $\Rd$).  Note that for $L^2$ estimates, where the boundary does not present any issues, \mbox{\Large$\chi$} may be chosen as \mbox{\Large$\chi$}$(\boldsymbol{v}) = 1$ for any $\boldsymbol{v} \in \Rd$.

\section{The Semi-discrete Problem} \label{semi_discrete_section}
First, applying the mode projection operator $\Pi^N_{\Lv}$ to both sides of the space-homogeneous Landau equation gives
\begin{equation*}
\frac{\partial}{\partial t}(\Pi^N_{\Lv} f)(t, \boldsymbol{v}) = \Pi^N_{\Lv} Q(f,f)(t, \boldsymbol{v}),
\end{equation*}
\begin{flalign*}
\textrm{with } && f(0, \boldsymbol{v}) = f_0(\boldsymbol{v}), &&
\end{flalign*}
for the pdf $f(t, \boldsymbol{v})$, with $(t,\boldsymbol{v}) \in (\mathbb{R}^+, \OmegaLv)$.  It should then be expected that, as in \cite{BoltzmannConvergence}, when $N$ is chosen to be sufficiently large,
\begin{equation*}
\Pi^N_{\Lv} Q(f,f)(t, \boldsymbol{v}) \sim \Pi^N_{\Lv} Q(\Pi^N_{\Lv} f,\Pi^N_{\Lv} f)(t, \boldsymbol{v}), ~~~~~\textrm{for } (t,\boldsymbol{v}) \in (\mathbb{R}^+, \OmegaLv).
\end{equation*}

When this analysis was applied to the Boltzmann equation, the collision operator was split as $Q = Q^+ - Q^-$ at this point, where $Q^+$ and $Q^-$ are the positive and negative parts of the collision operator, respectively.  For the Boltzmann equation, this is essential because any negative values in the collision operator come from $Q^-$ and so this part needs to be isolated.  For the Landau equation, however, the negative contribution is more entangled due to the derivatives but can be extracted by manipulations carried out by Desvillettes and Villani \cite{D&V_1}.  In order to use their calculations, it is useful to introduce the notation
\begin{align}
a_{i,j}(\boldsymbol{z}) &:= S_{i,j}(\boldsymbol{z}) = |\boldsymbol{z}|^{\lambda + 2} \left(\delta_{i,j} - \frac{z_i z_j}{|\boldsymbol{z}|^2}\right), &\textrm{for } i,j = 1, 2, \ldots, d, \label{a_def}\\
b_i(\boldsymbol{z}) &:= \frac{\partial}{\partial v_j}(a_{i,j}) = -2 |\boldsymbol{z}|^{\lambda} z_i, &\textrm{for } i = 1, 2, \ldots, d, \label{b_def}\\
\textrm{and } ~~~~~  c(\boldsymbol{z}) &:= \mathrlap{\frac{\partial^2}{\partial z_i \partial z_j}(\bar{a}_{i,j}) = \nabla \cdot \boldsymbol{b} = -2 (\lambda + 3) |\boldsymbol{z}|^{\lambda}.} \label{c_def}
\end{align}

Then, the collision operator $Q$ can be written as
\begin{equation*}
Q(f,g) = \bar{a}_{i,j} \frac{\partial^2 f}{\partial v_i \partial v_j} - \bar{c} f, 
\end{equation*}
where the bar notation means convolution with the second argument of $Q(f,g)$ (e.g.\ $\bar{a} = a * g$).

Also with this notation, assuming that $f(\boldsymbol{v}) \to 0$ as $|\boldsymbol{v}| \to 0$, the weak form of $Q$ can be written as
\begin{equation}
\int_{\Rd} Q(f,g) \phi ~\dv = \int_{\Rd} f \left(\bar{a}_{i,j} \frac{\partial^2 \phi}{\partial v_i \partial v_j} + 2 \bar{b_i} \frac{\partial \phi}{\partial v_i} \right) ~\dv, \label{Q^a_weak} 
\end{equation}
which is derived in Lemma \ref{Weak_Landau_Lemma} of the appendix.

One final thing to note is that, since $Q$ involves a convolution operator, it has support in $\Omega_{\sqrt{2}\Lv}$.  So, for simplicity, it will be assumed that $Q$ is defined on $\Omega_{2\Lv}$.  The problem then reduces to finding the pdf $g(t, \boldsymbol{v})$, with $(t,\boldsymbol{v}) \in (\mathbb{R}^+, \OmegaLv)$, such that
\begin{equation}
\frac{\partial g}{\partial t}(t, \boldsymbol{v}) = \Pi^N_{2\Lv} Q(\gchi, \gchi)(t, \boldsymbol{v}), \label{semi-discrete1}
\end{equation}
\begin{flalign*}
\textrm{with } && g(0, \boldsymbol{v}) = g_0^N(\boldsymbol{v}) = \Pi^N_{2\Lv} f_0(\boldsymbol{v}). &&
\end{flalign*}

Here, it should be that $g \approx \Pi^N_{2\Lv} f$.  The classical theory of spectral accuracy already ensures that $g$ is indeed a good approximation to the solution of the Landau equation in the cut-off domain $\Omega_{2\Lv}$, for sufficiently large $N$.  The problem here, however, is that it is more desirable to fix a smaller number of Fourier modes for computational purposes, but this removes the conservative properties of the collision operator.  For this reason, the right-hand side of equation (\ref{semi-discrete1}) is referred to as the \textit{unconserved collision operator}, denoted $Q_u$, namely
\begin{equation}
Q_u(g,g)(t, \boldsymbol{v}) := \Pi^N_{2\Lv} Q(\gchi, \gchi)(t, \boldsymbol{v}). \label{Q_u}
\end{equation}

If the operator $Q_u$ was used by itself in the numerical method, an error would accumulate due to the lack of conservation.  This has already been seen in the work of the present authors \cite{RGDPaper, EntropyDecayPaper} when the simulation was left to run with this operator alone.  The issue is fixed, however, by correcting any values of $Q_u$ computed to the solution of a Lagrange minimisation problem which enforces conservation.

In order to describe the minimisation problem more precisely, first define the Banach space $\mathcal{B}$ as
\begin{equation}
\mathcal{B} := \left\{X \in L^2(\Omega_{\Lv}): \int_{\OmegaLv} X ~\dv = \int_{\OmegaLv} X \boldsymbol{v} ~\dv = \int_{\OmegaLv} X |\boldsymbol{v}|^2 ~\dv = 0 \right\},
\end{equation}
which is simply all functions $X \in L^2(\Omega_{\Lv})$ with appropriate moments conserved.  This means that the solution $X^{\dag}$ to the minimisation problem of finding the closest $L^2(\Omega_{\Lv})$ function to $Q_u$ with the correct collision invariants can be written as
\begin{equation}
X^{\dag} := \min_{X \in \mathcal{B}} \int_{\OmegaLv} (Q_u(f,f) - X)^2 ~\dv. \label{Minimisation}
\end{equation}
Since this corrected version of $Q_u$ satisfies the correct conservation properties, it is referred to as the \textit{conserved collision operator} and denoted $Q_c(f,f) := X^{\dag}$.  It should be noted that, when $f$ is discretised onto $N^d$ many Fourier modes, the conserved collision operator is also a discrete vector $Q_c \in \mathbb{R}^{N^d}$.

After correcting $Q_u$ to $Q_c$ via the above minimisation method, the actual semi-discrete problem that will be analysed here is to find $g(t, \boldsymbol{v})$, with $(t,\boldsymbol{v}) \in (\mathbb{R}^+, \OmegaLv)$, such that
\begin{equation}
\frac{\partial g}{\partial t}(t, \boldsymbol{v}) = Q_c(g,g)(t, \boldsymbol{v}), \label{semi-discrete2}
\end{equation}
\begin{flalign*}
\textrm{with } && g(0, \boldsymbol{v}) = g_0^N(\boldsymbol{v}) = \Pi^N_{2\Lv} f_0(\boldsymbol{v}). &&
\end{flalign*}

The only assumption on the initial data required for the analysis here is that, for some $\varepsilon$ with $0 < \varepsilon \leq \frac{1}{4}$, $g_0$ satisfies
\begin{equation}
\int_{\{g_0 < 0\}} |g_0(\boldsymbol{v})|\langle\boldsymbol{v}\rangle^2 ~\dv \leq \varepsilon \int_{\{g_0 \geq 0\}} g_0(\boldsymbol{v})\langle\boldsymbol{v}\rangle^2 ~\dv. \label{stability_init}
\end{equation}
Assumption (\ref{stability_init}) can be seen as a sort of stability condition on the initial data.

\section{Useful Analytical Results} \label{Previous_Results}
Here certain useful results from the previous work of other authors will be collected and restated for use throughout the analysis.  First, perhaps the most frequently used identity is the ellipticity of $\bar{a}$, written as Proposition 4 in \cite{D&V_1}.

\begin{prop} \label{ellipticity_prop}
	If $f \in L^1_2 \bigcap L \log L (\Rd)$ then there is a constant $K_{\lambda}(f_0)$, depending on the potential $\lambda$ and the initial condition $f_0$, such that 
	\begin{equation*}
	\bar{a}_{i,j} \xi_i \xi_j \geq K_{\lambda}(f_0) (1 + |\boldsymbol{v}|^{\lambda}) |\boldsymbol{\xi}|^2, ~~~~~\textrm{for any } \boldsymbol{\xi} = (\xi_1, \xi_2, \ldots, \xi_d) \in \Rd,
	\end{equation*}
	where $K_{\lambda}(f_0)$ is defined by a combination of expressions (52) and (55) in \cite{D&V_1} and can be written as $K_{\lambda}(f_0) = \widetilde{K_{\lambda}}m_0(f_0)$, for a new constant $\widetilde{K_{\lambda}}$ which no longer depends on $g_0$, but only $\lambda$.
\end{prop}

Next, an important pair of theorems from \cite{D&V_1} state that the moments and weighted Sobolev $H^2_k$-norm of the true solution $f$ to the Fokker-Planck-Landau equation (\ref{Landau_homo_Kn1}) remain bounded in time, especially because the semi-discrete analogues of these theorem will be proven on the solution $g$ of equation (\ref{semi-discrete2}).  The relevant theorem on propagation of moments is as follows.
\begin{prop} \label{moment_estimate_prop}
	If $f$ is a weak solution to the Fokker-Planck-Landau type equation (\ref{Landau_homo_Kn1}) associated to hard potentials, with initial condition $f_0 \in L^1_{2}$ then, for any $s > 0$ and time $t_0 > 0$, there is a constant $C^f(f_0) > 0$ depending on $t_0$ and $f_0$ such that
	\begin{flalign*}
	&&~~~~~~~~~~~~~~ m_k(f)(t) \leq C^f(f_0), && \textrm{for } t > t_0.	
	\end{flalign*}
\end{prop}
Then, the theorem from \cite{D&V_1} related to the propagation of $H^2_k$-norm of the true solution $f$ is stated.  Some identities contained in the proof of this theorem are also important so they will be quoted here as well.
\begin{prop} \label{Halpha_estimate_prop}
	(a) If $f_0 \in L^1_{2+\delta} \bigcap L \log L (\Rd)$, for some $\delta > 0$, then there is a weak solution $f$ to the Fokker-Planck-Landau type equation (\ref{Landau_homo_Kn1}) associated to hard potentials such that, for any $k > 0$,
	
	(i) if $\|f_0\|_{L^2_k(\Rd)} < \infty$ and $\|f_0\|_{L^2_{\frac{5}{4}(k + \lambda)}(\Rd)} < \infty$ then
	\begin{equation*}
	\sup_{t>t_0} \|f(t, \cdot)\|_{L^2_k}(\Rd) < \infty;
	\end{equation*}
	
	(ii) for any time $t_0 > 0$ and $s \in \mathbb{N}_0$, there is a constant $C^f_{\lambda,s,k}(f_0)$ which depends on $\lambda$, $s$, $k$ and $t_0$ as much $f_0$ such that
	\begin{equation*}
	\sup_{t>t_0} \|f(t, \cdot)\|_{H^s_k}(\Rd) \leq C^f_{\lambda,s,k}(f_0);
	\end{equation*}
	
	(iii) for any time $t > 0$, $f \in C^{\infty}\bigl([t_0, \infty)); \mathcal{S}(\Rd)\bigr)$.
	
	\noindent(b) Moreover, for any $k \geq 0$, if $K_{\lambda}(f_0)$ is the constant from Proposition \ref{ellipticity_prop},
	
	(i) there is a constant $C_{\lambda}$, depending on the potential $\lambda$, such that
	\begin{equation*}
	\int_{\Rd} Q(f, f) f(\boldsymbol{v}) \langle \boldsymbol{v} \rangle^{2 k} ~\dv \leq -K_{\lambda}(f_0) \|f\|^2_{\dot{H}^1_{k + \frac{\lambda}{2}}(\Rd)} + C_{\lambda} \|f\|^{2}_{L^1_{\frac{5}{2}\left(k + \frac{\lambda}{2}\right)}(\Rd)};
	\end{equation*}
	
	(ii) for $s \in \mathbb{N}$, there is some $\nu > 1$ and constants $K_{\lambda, s}(f_0) > 0$ (related to $K_{\lambda}(f_0)$) and $C_{\lambda, s} > 0$, both depending on $\lambda$ and $s$ such that
	\begin{multline*}
	\sum_{\substack{\alpha \in \Nd: \\ |\alpha| = s}} \int_{\Rd} D^{\alpha} Q(f, f) D^{\alpha} f \langle \boldsymbol{v} \rangle^{2 k} ~\dv \leq - K_{\lambda, s}(f_0) \left(\left\| f \right\|^2_{\dot{H}^{s}_k(\Rd)} \right)^{\nu} \\
	+ C_{\lambda, s} \left\| f \right\|_{\dot{H}^{s}} \left\| f \right\|^2_{H^{s-1}_{k + \frac{\lambda}{2} + 1}(\Rd)}.
	\end{multline*}
\end{prop}

Note that parts (b)$(i)$ and $(ii)$ of Proposition \ref{Halpha_estimate_prop} appear in the proofs of (a)$(i)$ and $(ii)$, respectively, in Section 6 of \cite{D&V_1}.  In the proof of (a)$(ii)$, the $C_{\lambda, s} \left\| f \right\|_{\dot{H}^{s}}$ term in the identity from (b)$(ii)$ is merely replaced by a generic $C$ but, by following the previous details of the proof, it must be that $C$ includes the factor $\left\| f \right\|_{\dot{H}^{s}}$.

A more recent estimate on the Fokker-Planck-Landau operator given by He will also be useful, in particular for obtaining a bound on the $L^2$-norm of the operator.  The following result is a part of Theorem 1.5 in \cite{LBH}.
\begin{prop} \label{LBH_prop}
	When $\lambda > -2$, given constants $w_1, w_2, a, b \in \mathbb{R}$ with $w_1 + w_2 = \lambda + 2$ and $a + b = 2$ and any smooth functions $F$, $G$ and $H$, there is some uniform constant $C_H > 0$ such that the Landau collision operator $Q$ satisfies
	\begin{multline*}
	\left|\int_{\OmegaLv} Q(F, G) H ~\dv \right| \leq C_H \Bigl(\|F\|_{L^1_{\lambda + 2 + (-w_1)^+ + (-w_2)^+}(\OmegaLv)} \\
	+ \|F\|_{L^2(\OmegaLv)}\Bigr) \|G\|_{H^a_{w_1}(\OmegaLv)} \|H\|_{H^b_{w_2}(\OmegaLv)},
	\end{multline*}
	where the superscript ${}^+$ notation means to take the positive part (i.e.\ \linebreak $(w)^+ := \max(w,0)$).  
	
	For the specific case of $b = w_2 = 0$, so that $a = 2$ and $w_1 = \lambda + 2$,
	\begin{multline}
	\left|\int_{\OmegaLv} Q(F, G) H ~\dv \right| \leq C_H \Bigl(\|F\|_{L^1_{\lambda + 2}(\OmegaLv)} \\
	+ \|F\|_{L^2(\OmegaLv)}\Bigr) \|G\|_{H^2_{\lambda + 2}(\OmegaLv)} \|H\|_{L^2(\OmegaLv)}. \label{LBH_bound}
	\end{multline}
\end{prop}

To see how this leads to an estimate on the $L^2$-norm of the collision operator, first consider the dual space definition of the $L^2$-norm, namely
\begin{equation*}
\|Q(F, G)\|_{L^2(\OmegaLv)} = \sup_{\substack{H \in L^2(\OmegaLv): \\ \|H\|_{L^2(\OmegaLv)} \leq 1}} \left|\int_{\OmegaLv} Q(F, G) H ~\dv \right|.
\end{equation*}
By using the estimate (\ref{LBH_bound}), this means
\begin{equation}
\|Q(F, G)\|_{L^2(\OmegaLv)} \leq C_H \left(\|F\|_{L^1_{\lambda + 2}(\OmegaLv)} + \|F\|_{L^2(\OmegaLv)}\right) \|G\|_{H^2_{\lambda + 2}(\OmegaLv)}. \label{Q_L2_estimate}
\end{equation}

\section{An Extended Isoperimetric Problem for Conservation} \label{Conservation_Estimate}
To begin the analysis required toward deriving an error estimate, it will first be shown that the minimisation problem introduced in Section \ref{semi_discrete_section} to obtain the conserved collision operator $Q_c$ corresponding to its unconserved counterpart $Q_c$ has a unique solution.  This is given in the following lemma.
\begin{lemma} \label{Isoperimetric_Lemma}
	The unique minimiser to the conservation problem (\ref{Minimisation}) is given by
	\begin{equation}
	Q_c(f,f)(\boldsymbol{v}) = Q_u(f,f)(\boldsymbol{v}) - \frac{1}{2}\left(\gamma_1 + \sum_{j = 1}^{d} \gamma_{j+1} v_j + \gamma_{d+2} |\boldsymbol{v}|^2 \right), \label{Conserved_Q}
	\end{equation}
	where $Q_u$ is defined by expression (\ref{Q_u}) and $\{\gamma_j\}_{j=1}^{d+2}$ are the Lagrange multipliers associated with the minimisation problem and are given by
	\begin{align}
	\gamma_1 &= O_d M_1 + O_{d+2} M_{|\boldsymbol{v}|^2}, \nonumber \\
	\gamma_{j+1} &= O_{d+2} M_{v_j}, ~~~~~\textrm{for } j = 1, \ldots, d, \label{Lagrange_multipliers}\\ 
	\gamma_{d+2} &= O_{d+2} M_1 + O_{d+4} M_{|\boldsymbol{v}|^2}, \nonumber
	\end{align}
	where $O_r$ denotes a constant that is $\mathcal{O}({\Lv}^{-r})$ and $M_{\phi(\boldsymbol{v})} := \int_{\OmegaLv} Q_u(f,f) \phi(\boldsymbol{v}) ~\dv$, for $\phi(\boldsymbol{v}) = 1, v_1, \ldots, v_d, |\boldsymbol{v}|^2$, is a moment of $Q_u$ and is the residual error that should be zero if $Q_u$ were conservative.  Furthermore, the $L^2(\OmegaLv)$ function being minimised in (\ref{Minimisation}) can be bounded as
	\begin{align*}
	\|Q_u(f,f) - Q_c(f,f)\|_{L^2(\OmegaLv)} &\leq c_d\left(2\gamma_1^2 {\Lv}^d + \sum_{j = 1}^{d} \gamma_{j+1}^2 {\Lv}^{d+2} + \gamma_{d+2}^2 {\Lv}^{d+4} \right) \\
	&\leq \frac{C_d}{{\Lv}^d}\left(M_1^2 + \sum_{j=1}^{d} M_{v_j}^2 + \frac{M_{|\boldsymbol{v}|^2}}{L^{d+1}} \right),
	\end{align*}
	where $c_d$ is a constant that depends on the dimension $d$.  In particular, for the physically realisable case of $d = 3$, 
	\begin{align*}
	\|Q_u(f,f) - Q_c(f,f)\|_{L^2(\OmegaLv)} &\leq 2\gamma_1^2 {\Lv}^3 + \frac{2}{3} \sum_{j = 1}^{3} \gamma_{j+1}^2 {\Lv}^{5} + 4 \gamma_1 \gamma_d {\Lv}^5 + \frac{38}{15} \gamma_{5}^2 {\Lv}^{7} \\
	&\leq \frac{c_3}{{\Lv}^3}\left(M_1^2 + \sum_{j=1}^{3} M_{v_j}^2 + \frac{M_{|\boldsymbol{v}|^2}}{{\Lv}^{4}} \right).
	\end{align*}
\end{lemma}
The proof of this is mainly an exercise in calculus of variations and is no different for the current Fokker-Planck-Landau setting than it is for the Boltzmann equation, as the explicit form of $Q_u$ is never required.  The full proof can be found under lemma 3.2 in \cite{BoltzmannConvergence}.  An important point to notice, however, is that the difference in $Q_u$ and $Q_c$ only depends on squares of the moments $M_1$, $M_{v_1}$, \ldots,$M_{v_d}$, $M_{|\boldsymbol{v}|^2}$, which are going to be small due to being errors in quantities that should be zero.  

The next step is to find a bound on the moments of this correction, as in the following theorem.  
\begin{theorem} \label{Q_moment_theorem}
	For $f \in L^2(\OmegaLv)$, the $L^2$-error in the moments of the difference between the conserved and unconserved collision operators is proportional to the spectral error plus a negligible term inversely proportional to the size of the domain.  More precisely, if $Q_c$ is the conserved operator given by (\ref{Conserved_Q}) and $Q_u$ is the unconserved operator given by (\ref{Q_u}), for any $k \geq 0$ and $k' \geq 2$,
	\begin{align}
	&\left\|(Q_u(f,f) - Q_c(f,f)) \langle \boldsymbol{v} \rangle^{k}\right\|_{L^2(\OmegaLv)} \nonumber \\
	\leq&~ \frac{C_d}{\sqrt{2k + 1}} \biggl(O_{-k} \bigl| \bigl| \left(\Pi^N_{2\Lv} - 1 \right) Q(\Chi f, \Chi f)  \bigr| \bigr|_{L^2(\OmegaLv)} \nonumber \\
	&~~~~~~~~~~~~~~ + O_{\frac{d}{2} + k' - k} \Bigl(m_{0}(f) m_{k' + \lambda}(f) + m_{\lambda}(f) m_{k'}(f) \nonumber \\
	&~~~~~~~~~~~~~~~~~~~~~~~~~~~~~ + m_{2}(f) m_{k' + \lambda - 2}(f) + m_{\lambda + 2}(f) m_{k' - 2}(f)\Bigr)\biggr), \label{Q_moment_error_k}
	\end{align}
	where $C_d$ is a constant depending only on the dimension $d$ and $O_r = \mathcal{O}({\Lv}^{-r})$.  In particular, for $k = 0$ and $k' = 2$, this reduces to
	\begin{align}
	&\left\|Q_u(f,f) - Q_c(f,f) \right\|_{L^2(\OmegaLv)} \nonumber \\
	\leq &~ C_d \biggl(\bigl| \bigl| \left(\Pi^N_{2\Lv} - 1 \right) Q(\Chi f, \Chi f) \bigr| \bigr|_{L^2(\OmegaLv)} \nonumber \\
	&~~~~~~~~~~~~~~~~~~~~~~~~~~~~~~ + O_{\frac{d}{2} + 2} \Bigl(m_{0}(f) m_{2 + \lambda}(f) + m_{\lambda}(f) m_{2}(f) \Bigr)\biggr), \label{Q_moment_error_0}
	\end{align}
	for a potentially different constant $C_d$.
\end{theorem}
\begin{proof}
	First, by using the expression for the conserved collision operator (\ref{Conserved_Q}) in place of $Q_c$ and then the triangle inequality,
	\begin{align}
	&\left\|(Q_u(f,f) - Q_c(f,f)) \langle \boldsymbol{v} \rangle^{k}\right\|_{L^2(\OmegaLv)} \nonumber \\
	= ~&\left\| \frac{1}{2}\left(\gamma_1 + \sum_{j = 1}^{d} \gamma_{j+1} v_j + \gamma_{d+2} |\boldsymbol{v}|^2 \right) \langle \boldsymbol{v} \rangle^{k}\right\|_{L^2(\OmegaLv)}  \nonumber \\
	\leq ~&\frac{1}{2}\biggl(|\gamma_1| \left\| \langle \boldsymbol{v} \rangle^{k}\right\|_{L^2(\OmegaLv)} + \sum_{j = 1}^{d} |\gamma_{j+1}| \left\|v_j \langle \boldsymbol{v} \rangle^{k}\right\|_{L^2(\OmegaLv)} \nonumber \\
	&~~~~~~~~~~~~~~~~~~~~~~~~~~~~~~~~~~~~~~~+ |\gamma_{d+2}| \left\| |\boldsymbol{v}|^2 \langle \boldsymbol{v} \rangle^{k} \right\|_{L^2(\OmegaLv)} \biggr). \label{correction_bound1}
	\end{align}
	Here, by using spherical polar co-ordinates and noticing that $\OmegaLv \subset B_{\sqrt{d}\Lv}$, it can be shown that
	\begin{flalign}
	&&\left\| \langle \boldsymbol{v} \rangle^{k}\right\|_{L^2(\OmegaLv)} &\leq \left(\frac{\omega_{d-1} d^{\frac{d}{2}}}{2k + 1} \sum_{j = 0}^{2k} {2k + 1 \choose j + 1} {\Lv}^{j}\right)^{\frac{1}{2}} {\Lv}^{\frac{d}{2}}, && \label{v^k_L2_bound}\\
	&&\left\| v_j \langle \boldsymbol{v} \rangle^{k}\right\|_{L^2(\OmegaLv)} &\leq \left(\frac{\omega_{d-1} d^{\frac{d}{2} + 1}}{2k + 1} \sum_{j = 0}^{2k} {2k + 1 \choose j + 1} {\Lv}^{j}\right)^{\frac{1}{2}} {\Lv}^{\frac{d}{2} + 1} && \\
	\textrm{and } && \left\| |\boldsymbol{v}|^2 \langle \boldsymbol{v} \rangle^{k}\right\|_{L^2(\OmegaLv)} & \leq \left(\frac{\omega_{d-1} d^{\frac{d}{2} + 2}}{2k + 1} \sum_{j = 0}^{2k} {2k + 1 \choose j + 1} {\Lv}^{j}\right)^{\frac{1}{2}} {\Lv}^{\frac{d}{2} + 2}, &&
	\end{flalign}
	where it was assumed that $k$ is an integer.  This can be done without loss of generality by noticing that $\langle \boldsymbol{v} \rangle^{k} \leq \langle \boldsymbol{v} \rangle^{\lceil k \rceil}$, so the next highest integer can replace $k$ on the right-hand side if necessary.   
	
	Using these estimates in (\ref{correction_bound1}) and further bounding by the largest power of $d$ gives
	\begin{align}
	&\left\|(Q_u(f,f) - Q_c(f,f)) \langle \boldsymbol{v} \rangle^{k}\right\|_{L^2(\OmegaLv)} \nonumber \\
	\leq ~&\frac{1}{2} \left(\frac{\omega_{d-1} d^{\frac{d}{2} + 2}}{2k + 1} \sum_{j = 0}^{2k} {2k + 1 \choose j + 1} {\Lv}^{j}\right)^{\frac{1}{2}} \biggl(|\gamma_1| {\Lv}^{\frac{d}{2}} + \sum_{j = 1}^{d} |\gamma_{j+1}| {\Lv}^{\frac{d}{2} + 1} \nonumber \\
	&~~~~~~~~~~~~~~~~~~~~~~~~~~~~~~~~~~~~~~~~~~~~~~~~~~~~~~~~~~~~~~~~+ |\gamma_{d+2}| {\Lv}^{\frac{d}{2} + 2} \biggr). \label{correction_bound2}
	\end{align}
	This reduces the problem to finding bounds on the Lagrange multipliers $\gamma_1,$ $\gamma_2, \ldots, \gamma_{d+2}$.  As in expressions (\ref{Lagrange_multipliers}), however, these all depend on the moments of the unconserved collision operator $M_{\phi(\boldsymbol{v})}$, for the collision invariants $\phi(\boldsymbol{v}) = 1, v_1, \ldots, v_d, |\boldsymbol{v}|^2$, so these are the quantities which must be estimated.  Here, integrating the true collision operator $Q$ against the collision invariants over all $\Rd$ gives zero and so
	\begin{align*}
	\left|M_{\phi(\boldsymbol{v})} \right| &= \left|\int_{\OmegaLv} Q_u(f,f) \phi(\boldsymbol{v}) ~\dv \right| \nonumber \\
	&= \left|\int_{\OmegaLv} Q_u(f,f) \phi(\boldsymbol{v}) ~\dv - \int_{\Rd} Q(\Chi f,\Chi f) \phi(\boldsymbol{v}) ~\dv \right| \nonumber \\
	&= \biggl|\int_{\OmegaLv} \left(Q_u(f,f) - Q(\Chi f,\Chi f) \right) \phi(\boldsymbol{v}) ~\dv \nonumber\\
	&~~~~~~~~~~~~~~~~~~~~~~~~~~~~ - \int_{\Rd \backslash \OmegaLv} Q(\Chi f,\Chi f) \phi(\boldsymbol{v}) ~\dv \biggr| \nonumber \\
	&= \biggl|\int_{\OmegaLv} \left(\Pi^N_{2\Lv} - 1 \right) Q(\Chi f, \Chi f) \phi(\boldsymbol{v}) ~\dv \nonumber\\
	&~~~~~~~~~~~~~~~~~~~~~~~~~~~~ - \int_{\Rd \backslash \OmegaLv} Q(\Chi f,\Chi f) \phi(\boldsymbol{v}) ~\dv \biggr|,
	\end{align*}
	by using the definition of $Q_u$ in (\ref{Q_u}).  So, by the triangle inequality and Cauchy-Schwarz,
	\begin{equation*}
	\left|M_{\phi(\boldsymbol{v})} \right| \leq \bigl| \bigl| \left(\Pi^N_{2\Lv} - 1 \right) Q(\Chi f, \Chi f)  \bigr| \bigr|_{L^2(\OmegaLv)} \bigl| \bigl| \phi(\boldsymbol{v}) \bigr| \bigr|_{L^2(\OmegaLv)} + I_{\phi(\boldsymbol{v})},
	\end{equation*}
	\begin{flalign*}
	\textrm{for } && I_{\phi(\boldsymbol{v})} := \left|\int_{\Rd \backslash \OmegaLv} Q(\Chi f,\Chi f) \phi(\boldsymbol{v}) ~\dv \right|. &&
	\end{flalign*}
	
	Now, the $L^2$-norms of the collision invariants are given by
	\begin{flalign}
	&& \bigl| \bigl| 1 \bigr| \bigr|_{L^2(\OmegaLv)} = (2\Lv)^{\frac{d}{2}} = O_{-\frac{d}{2}}, ~~&~~~\bigl| \bigl| v_j \bigr| \bigr|_{L^2(\OmegaLv)} = \frac{1}{2 \sqrt{3}}(2\Lv)^{\frac{d}{2}+1} = O_{-(\frac{d}{2} + 1)} && \nonumber \\
	\textrm{and } && \bigl| \bigl\| |\boldsymbol{v}|^2 \bigr| \bigr|_{L^2(\OmegaLv)} = &\frac{5d^2 + 5d - 9}{180}(2\Lv)^{\frac{d}{2} + 2} = O_{-(\frac{d}{2} + 2)}, && \label{collision_invariant_L2}
	\end{flalign}
	which means, by using the expressions for the Lagrange multipliers in (\ref{Lagrange_multipliers}),
	\begin{align*}
	\gamma_1 =&~ O_d \left(O_{-\frac{d}{2}} \bigl| \bigl| \left(\Pi^N_{2\Lv} - 1 \right) Q(\Chi f, \Chi f)  \bigr| \bigr|_{L^2(\OmegaLv)} + I_{1} \right) \\
	& + O_{d+2} \left(O_{-(\frac{d}{2} + 2)} \bigl| \bigl| \left(\Pi^N_{2\Lv} - 1 \right) Q(\Chi f, \Chi f)  \bigr|  \bigr|_{L^2(\OmegaLv)} + I_{|\boldsymbol{v}|^2} \right) \\
	=&~ O_{\frac{d}{2}} \bigl| \bigl| \left(\Pi^N_{2\Lv} - 1 \right) Q(\Chi f, \Chi f)  \bigr| \bigr|_{L^2(\OmegaLv)} + O_d I_{1} + O_{d+2} I_{|\boldsymbol{v}|^2}, \\
	\gamma_{j+1} =&~ O_{d+2} \left(O_{-(\frac{d}{2} + 1)} \bigl| \bigl| \left(\Pi^N_{2\Lv} - 1 \right) Q(\Chi f, \Chi f)  \bigr| \bigr|_{L^2(\OmegaLv)} + I_{\boldsymbol{v}_j} \right)\\
	=&~ O_{\frac{d}{2} + 1} \bigl| \bigl| \left(\Pi^N_{2\Lv} - 1 \right) Q(\Chi f, \Chi f)  \bigr| \bigr|_{L^2(\OmegaLv)} + O_{d+2} I_{\boldsymbol{v}_j}, ~~~~~\textrm{for } j = 1, \ldots, d, \\ 
	\gamma_{d+2} =&~ O_{d+2} \left(O_{-\frac{d}{2}} \bigl| \bigl| \left(\Pi^N_{2\Lv} - 1 \right) Q(\Chi f, \Chi f)  \bigr| \bigr|_{L^2(\OmegaLv)} + I_{1} \right) \\
	& + O_{d+4} \left(O_{-(\frac{d}{2} + 2)} \bigl| \bigl| \left(\Pi^N_{2\Lv} - 1 \right) Q(\Chi f, \Chi f)  \bigr|  \bigr|_{L^2(\OmegaLv)} + I_{|\boldsymbol{v}|^2} \right) \\
	=&~ O_{\frac{d}{2} + 2} \bigl| \bigl| \left(\Pi^N_{2\Lv} - 1 \right) Q(\Chi f, \Chi f)  \bigr| \bigr|_{L^2(\OmegaLv)} + O_{d+2} I_{1} + O_{d+4} I_{|\boldsymbol{v}|^2}.
	\end{align*}
	
	Then, using these expressions in the bounds on the moments in (\ref{correction_bound2}) and noting that the first $O_{\frac{d}{2}+l}$ term cancels perfectly with the power of $\Lv$ in the product with the Lagrange multiplier, for the constant $C_d := \frac{1}{2} \sqrt{\omega_{d-1} d^{\frac{d}{2} + 2}}$ depending only on the dimension,
	\begin{align}
	&\left\|(Q_u(f,f) - Q_c(f,f)) \langle \boldsymbol{v} \rangle^{k}\right\|_{L^2(\OmegaLv)} \nonumber \\
	\leq ~& \frac{C_d}{\sqrt{2k + 1}} \left(\sum_{j = 0}^{2k} {2k + 1 \choose j + 1} {\Lv}^{j}\right)^{\frac{1}{2}} \biggl(\bigl| \bigl| \left(\Pi^N_{2\Lv} - 1 \right) Q(\Chi f, \Chi f)  \bigr| \bigr|_{L^2(\OmegaLv)} \nonumber \\
	&~~~~~~~~~~~~~~~~~~~~~~~~~ + O_{\frac{d}{2}} I_{1} + \sum_{j = 1}^{d} O_{\frac{d}{2}+1} I_{\boldsymbol{v}_j} + O_{\frac{d}{2}+2} I_{|\boldsymbol{v}|^2} \biggr). \label{correction_bound3} 
	\end{align}
	
	Finally, for the bounds on $I_{1}$, $I_{\boldsymbol{v}_j}$ and $I_{|\boldsymbol{v}|^2}$, since the first argument in $Q$ disappears on the interior boundary, Lemma \ref{Q_decay} from the appendix can be used.  The bound on $I_{1}$ is merely an application of the statement of the lemma.  For $I_{|\boldsymbol{v}|^2}$, however, note that
	\begin{align*}
	\frac{1}{{\Lv}^{2}} \left| \int_{\Rd \backslash \OmegaLv} Q(\Chi f,\Chi f) |\boldsymbol{v}|^2 ~\dv \right| &= \frac{1}{{\Lv}^{2}} \left| \int_{\Rd \backslash \OmegaLv} Q(\Chi f,\Chi f) |\boldsymbol{v}|^{k'} |\boldsymbol{v}|^{2 - k'} ~\dv \right| \\
	&\leq O_{k'} \left| \int_{\Rd \backslash \OmegaLv} Q(\Chi f,\Chi f) |\boldsymbol{v}|^{k'} ~\dv \right|
	\end{align*}
	and the remainder of the proof of the lemma follows in the same way.  Then, for $I_{\boldsymbol{v}_j}$,
	\begin{align*}
	\frac{1}{\Lv} \left| \int_{\Rd \backslash \OmegaLv} Q(\Chi f,\Chi f) \boldsymbol{v}_j ~\dv \right| &\leq \frac{1}{\Lv} \int_{\Rd \backslash \OmegaLv} |Q(\Chi f,\Chi f)\|\boldsymbol{v}\|\boldsymbol{v}|^{k' - 1} |\boldsymbol{v}|^{-k' + 1} ~\dv \\
	&\leq O_{k'} \int_{\Rd \backslash \OmegaLv} |Q(\Chi f,\Chi f)\|\boldsymbol{v}|^{k'} ~\dv
	\end{align*}
	and again the proof of the lemma follows.  Strictly speaking, the weak form identity should be applied to $Q$ before the absolute values are moved inside the integral but the result remains the same.  In the end, for any $k' > 2$,
	\begin{align*}
	\max(I_{1}, ~O_{1} I_{\boldsymbol{v}_j}, ~O_{2} I_{|\boldsymbol{v}|^2}) &= O_{k'} \Bigl(m_{0}(\Chi f) m_{k' + \lambda}(\Chi f) + m_{\lambda}(\Chi f) m_{k'}(\Chi f) \nonumber \\
	&~~~~~~ + m_{2}(\Chi f) m_{k' + \lambda - 2}(\Chi f) + m_{\lambda + 2}(\Chi f) m_{k' - 2}(\Chi f)\Bigr).
	\end{align*}
	
	Therefore, using this in (\ref{correction_bound3}); bounding the extension operator \mbox{\Large$\chi$} by 1; and considering the expression in the sum as an $O_{-k}$ term gives
	\begin{align}
	&\left\|(Q_u(f,f) - Q_c(f,f)) \langle \boldsymbol{v} \rangle^{k}\right\|_{L^2(\OmegaLv)} \nonumber \\
	\leq &~ \frac{C_d}{\sqrt{2k + 1}} \biggl(O_{-k} \bigl| \bigl| \left(\Pi^N_{2\Lv} - 1 \right) Q(\Chi f, \Chi f) \bigr| \bigr|_{L^2(\OmegaLv)} \nonumber \\
	&~~~~~~~~~~~~~~~ + O_{\frac{d}{2} + k' - k} \Bigl(m_{0}(f) m_{k' + \lambda}(f) + m_{\lambda}(f) m_{k'}(f) \nonumber \\
	&~~~~~~~~~~~~~~~~~~~~~~~~~~~~~~~ + m_{2}(f) m_{k' + \lambda - 2}(f) + m_{\lambda + 2}(f) m_{k' - 2}(f)\Bigr)\biggr), \label{UnconservedQ_moment_bound}
	\end{align}
	which is the required result (\ref{Q_moment_error_k}) in Theorem \ref{Q_moment_theorem}.
	
	In addition, when $k = 0$, all of the initial calculations in this proof follow through and the bound (\ref{correction_bound3}) becomes
	\begin{align}
	&\left\|Q_u(f,f) - Q_c(f,f) \right\|_{L^2(\OmegaLv)} \nonumber \\
	\leq ~& C_d \biggl(\bigl| \bigl| \left(\Pi^N_{2\Lv} - 1 \right) Q(\Chi f, \Chi f)  \bigr| \bigr|_{L^2(\OmegaLv)} + O_{\frac{d}{2}} I_{1} + \sum_{j = 1}^{d} O_{\frac{d}{2}+1} I_{\boldsymbol{v}_j} + O_{\frac{d}{2}+2} I_{|\boldsymbol{v}|^2} \biggr). \label{correction_bound4} 
	\end{align}
	Then Lemma \ref{Q_decay} is applied again to bound the $I_{\phi(\boldsymbol{v})}$ terms, but this time with $k' = 2$ it should be noted that direct application of the lemma gives
	\begin{align*}
	&\left| \int_{\Rd \backslash \OmegaLv} Q(\Chi f,\Chi f) ~\dv \right| \\
	\leq~ & 2 \Bigl(m_{0}(\Chi f) m_{2 + \lambda}(\Chi f) + m_{\lambda}(\Chi f) m_{2}(\Chi f) \Bigr) \\
	&~~~~~~ + 2 \Bigl(m_{2}(\Chi f) m_{\lambda}(\Chi f) + m_{\lambda + 2}(\Chi f) m_{0}(\Chi f)\Bigr) \\
	=~ & \frac{4}{d {\Lv}^2} \Bigl(m_{0}(\Chi f) m_{2 + \lambda}(\Chi f) + m_{\lambda}(\Chi f) m_{2}(\Chi f) \Bigr),
	\end{align*}
	which gives the result for $I_1$.  
	
	So, in a similar way to the more general case above and introducing a $|\boldsymbol{v}|^2$ term in the integrand first then continuing through the proof of the lemma, for $I_{\boldsymbol{v}_j}$,
	\begin{align*}
	\frac{1}{\Lv} \left| \int_{\Rd \backslash \OmegaLv} Q(\Chi f,\Chi f) v_j ~\dv \right| &\leq \frac{1}{\Lv} \left| \int_{\Rd \backslash \OmegaLv} Q(\Chi f,\Chi f) |\boldsymbol{v}|^{2} |\boldsymbol{v}|^{-1} ~\dv \right| \\
	&\leq \frac{4}{\sqrt{d} {\Lv}^2} \Bigl(m_{0}(\Chi f) m_{2 + \lambda}(\Chi f) \\
	&~~~~~~~~~~~~~~~~~~~~~~~~+ m_{\lambda}(\Chi f) m_{2}(\Chi f) \Bigr)
	\end{align*}
	and for $I_{|\boldsymbol{v}|^2}$,
	\begin{multline*}
	\frac{1}{{\Lv}^2} \left| \int_{\Rd \backslash \OmegaLv} Q(\Chi f,\Chi f) |\boldsymbol{v}|^2 ~\dv \right| \leq \frac{4}{{\Lv}^2} \Bigl(m_{0}(\Chi f) m_{2 + \lambda}(\Chi f) \\
	+ m_{\lambda}(\Chi f) m_{2}(\Chi f) \Bigr).
	\end{multline*}
	This means
	\begin{equation*}
	\max(I_{1}, ~O_{1} I_{\boldsymbol{v}_j}, ~O_{2} I_{|\boldsymbol{v}|^2}) = \frac{4}{{\Lv}^2} \Bigl(m_{0}(\Chi f) m_{2 + \lambda}(\Chi f) + m_{\lambda}(\Chi f) m_{2}(\Chi f) \Bigr).
	\end{equation*}
	Using this in (\ref{correction_bound4}), after bounding the extension operator \mbox{\Large$\chi$} by 1, gives
	\begin{align*}
	&\left\|(Q_u(f,f) - Q_c(f,f)) \langle \boldsymbol{v} \rangle^{k}\right\|_{L^2(\OmegaLv)} \nonumber \\
	\leq &~ C_d \biggl(\bigl| \bigl| \left(\Pi^N_{2\Lv} - 1 \right) Q(\Chi f, \Chi f) \bigr| \bigr|_{L^2(\OmegaLv)} \\
	&~~~~~~~ + O_{\frac{d}{2} + 2} \Bigl(m_{0}(f) m_{2 + \lambda}(f) + m_{\lambda}(f) m_{2}(f) \Bigr)\biggr),
	\end{align*}
	which shows that $k = 0$ and $k' = 2$ can indeed be plugged into the previous bound (\ref{UnconservedQ_moment_bound}), to give the required result (\ref{Q_moment_error_0}) in Theorem \ref{Q_moment_theorem}.
\end{proof}

\section{Numerical Moment and $L^2$-norm Estimates} \label{Propagation}
\subsection{Estimates on the Time Derivatives of the Moments}
Here it will be assumed the solution $g$ to the semi-discrete problem (\ref{semi-discrete2}) is such that $g \in \mathcal{C}([0,T]; L^2(\OmegaLv))$ and the initial condition $g_0 \in L^2(\OmegaLv)$  satisfies the stability condition (\ref{stability_init}).  It will also be assumed that, for the $\varepsilon$ from this condition, there is some $T_{\varepsilon} > 0$ such that for all $t \in [0, T_{\varepsilon}]$,
\begin{equation}
\int_{\{g(t,\boldsymbol{v}) < 0\}} |g(t, \boldsymbol{v})|\langle\boldsymbol{v}\rangle^2 ~\dv \leq \varepsilon \int_{\{g(t,\boldsymbol{v}) \geq 0\}} g(t, \boldsymbol{v})\langle\boldsymbol{v}\rangle^2 ~\dv \label{stability}
\end{equation}
\begin{flalign*}
\textrm{and } && \sup_{t \in [0, T_{\varepsilon}]} \|g(t, \cdot)\|_{L^2(\OmegaLv)} < \infty. &&
\end{flalign*}
\begin{rem}
	The condition (\ref{stability}) is important here because, even if starting with a non-negative function $g_0$, imposing conservation by use of the solution (\ref{Conserved_Q}) to the Lagrange minimisation problem may force the solution $g$ to become negative in some parts.  It will be shown later in the proof of Theorem \ref{existence_theorem} that this is in fact true for all $t > 0$ and not just as an assumption.
\end{rem}  
Then, by writing $g$ as $g = g^+ - g^-$ where $g^+ \geq 0$ and  $g^- \geq 0$ are the positive and negative parts of $g$, respectively, $|g| = g^+ + g^- = g + 2g^-$ and so
\begin{align*}
\int_{\OmegaLv} |g(t, \boldsymbol{v})|\langle\boldsymbol{v}\rangle^2 ~\dv &= \int_{\OmegaLv} g(t, \boldsymbol{v})\langle\boldsymbol{v}\rangle^2 ~\dv + 2\int_{\OmegaLv} g^-(t, \boldsymbol{v})\langle\boldsymbol{v}\rangle^2 ~\dv \\
&= \int_{\OmegaLv} g_0(\boldsymbol{v})\langle\boldsymbol{v}\rangle^2 ~\dv + 2\int_{\OmegaLv} g^-(t, \boldsymbol{v})\langle\boldsymbol{v}\rangle^2 ~\dv,
\end{align*}
where mass and energy conservation have been used for the first term, because $\langle\boldsymbol{v}\rangle^2 = 1 + |\boldsymbol{v}|^2$.  So, by using the assumption (\ref{stability}),
\begin{align*}
\int_{\OmegaLv} |g(t, \boldsymbol{v})|\langle\boldsymbol{v}\rangle^2 ~\dv &\leq \int_{\OmegaLv} g_0(\boldsymbol{v})\langle\boldsymbol{v}\rangle^2 ~\dv + 2\varepsilon \int_{\OmegaLv} g^+(t, \boldsymbol{v})\langle\boldsymbol{v}\rangle^2 ~\dv \\
&\leq \int_{\OmegaLv} g_0(\boldsymbol{v})\langle\boldsymbol{v}\rangle^2 ~\dv + 2\varepsilon \int_{\OmegaLv} |g(t, \boldsymbol{v})|\langle\boldsymbol{v}\rangle^2 ~\dv
\end{align*}
and, if $\varepsilon \leq \frac{1}{4}$,
\begin{equation}
\int_{\OmegaLv} |g(t, \boldsymbol{v})|\langle\boldsymbol{v}\rangle^2 ~\dv \leq \frac{1}{1 - 2\varepsilon} \int_{\OmegaLv} g_0(\boldsymbol{v})\langle\boldsymbol{v}\rangle^2 ~\dv \leq 2 \int_{\OmegaLv} g_0(\boldsymbol{v})\langle\boldsymbol{v}\rangle^2 ~\dv. \label{stability_consequence}
\end{equation}

Also, before obtaining an estimate on the numerical moments in the following lemma, first note that by subtracting and adding $Q_u(g,g) - Q(\gchi, \gchi)$, as well as using the definition of $Q_u$ in (\ref{Q_u}),
\begin{align}
\frac{\partial g}{\partial t} = Q_c(g,g) &= Q_c(g,g) - Q_u(g,g) + Q(\gchi, \gchi) + Q_u(g,g) - Q(\gchi, \gchi) \nonumber \\
&= Q_c(g,g) - Q_u(g,g) + Q(\gchi, \gchi) \nonumber \\
&~~~~~ + \Pi^N_{2\Lv} Q(\gchi, \gchi)(t, \boldsymbol{v}) - Q(\gchi, \gchi) \nonumber \\
&=  Q_c(g,g) - Q_u(g,g) + Q(\gchi, \gchi) - \left(1 - \Pi^N_{2\Lv}\right) Q(\gchi, \gchi). \label{Q_expansion}
\end{align}

\begin{lemma} \label{moment_derivative_lemma}
	For a solution $g$ of the semi-discrete problem (\ref{semi-discrete2}) which satisfies the stability condition (\ref{stability}) with $\varepsilon \leq \min \left(\frac{1}{4}, \varepsilon_0 \right)$, and also has bounded gradient $\nabla g$, given any $k \geq \max (3, k_0)$, 
	\begin{multline}
	\frac{\textrm{d}}{\textrm{d} t} \bigl(m_k(g) \bigr) \leq - \frac{1}{2} {\varepsilon_{\chi}}^2 K_{\lambda, k} m_{0}(g_0) m_{k + \lambda}(g) + C^1_{d,k} \bigl(m_0(g) + m_k(g) \bigr) \\
	+ C^2_{d,k} \frac{\mathcal{O} \left({\Lv}^{k + \frac{d}{2}} \right)}{N^{\frac{d - 1}{2}}} \bigl| \bigl| \gchi \bigr| \bigr|^2_{L^2_{\lambda + 1}(\OmegaLv)}, \label{moment_derivative_bound} 
	\end{multline}
	where the moment operator $m_k$ and $L^2_k$-norm are defined by (\ref{mk_def}) and (\ref{L2k_def}), respectively; $g_0$ is any initial condition satisfying the stability condition (\ref{stability_init}); $\varepsilon_0$ and $k_0$ are constants that will be defined in the proof by expressions (\ref{epsilon_0}) and (\ref{k_0}), respectively; $\varepsilon_{\chi} \in (0,1)$ can be chosen arbitrarily in Lemma \ref{moment_cutoff_lemma} in the appendix; and $K_{\lambda, k}, C^1_{d,k}, C^2_{d,k} > 0$ are constants with $K_{\lambda, k}$ depending on the potential $\lambda$ and $k$ and $C^1_{d,k}$ and $C^2_{d,k}$ depending on $k$ and the dimension $d$. 
\end{lemma}

\begin{proof}
	First, after multiplying both sides of (\ref{Q_expansion}) by $\textrm{sgn}(g)(\boldsymbol{v}) \langle \boldsymbol{v} \rangle^{k}$ and integrating with respect to $\boldsymbol{v}$ over $\OmegaLv$,
	\begin{align}
	&\int_{\OmegaLv} \frac{\textrm{d}g}{\textrm{d} t} \textrm{sgn}(g)(\boldsymbol{v}) \langle \boldsymbol{v} \rangle^{k} ~\dv \nonumber \\
	=& \int_{\OmegaLv} Q(\gchi, \gchi) \textrm{sgn}(g)(\boldsymbol{v}) \langle \boldsymbol{v} \rangle^{k} ~\dv \nonumber \\
	&+ \int_{\OmegaLv} \Bigl((Q_c(g,g) - Q_u(g,g)) - \left(1 - \Pi^N_{2\Lv}\right) Q(\gchi, \gchi) \Bigr) \textrm{sgn}(g)(\boldsymbol{v}) \langle \boldsymbol{v} \rangle^{k} ~\dv. \label{moment_sgn_bound}
	\end{align}
	Then, since $\frac{\textrm{d}}{\textrm{d}t}\left(|g| \right) = \frac{\textrm{d}g}{\textrm{d}t} \textrm{sgn}(g)$ by the chain rule (at least when $g(t, \boldsymbol{v}) \neq 0$ but the integral is blind to this point),
	\begin{equation}
	\int_{\OmegaLv} \frac{\textrm{d}g}{\textrm{d} t} \textrm{sgn}(g)(\boldsymbol{v}) \langle \boldsymbol{v} \rangle^{k} ~\dv = \frac{\textrm{d}}{\textrm{d} t} \left(\int_{\OmegaLv}|g(t, \boldsymbol{v})| \langle \boldsymbol{v} \rangle^{k} ~\dv \right) = \frac{\textrm{d}}{\textrm{d} t} \bigl(m_k(g) \bigr), \label{moment_derivative_identity}
	\end{equation}
	because $g$ is only defined inside $\OmegaLv$.  This means $g$ can be assumed to be zero in $\Rd \backslash \OmegaLv$ and the integral over $\OmegaLv$ can be replaced with one over $\Rd$ as in the definition of $m_k$ in (\ref{mk_def}).  Also, bounding the last integral in (\ref{moment_sgn_bound}) by its absolute value, then using the triangle inequality, allows $\frac{\textrm{d}}{\textrm{d} t} \bigl(m_k(g) \bigr)$ to be estimated by
	\begin{align}
	\frac{\textrm{d}}{\textrm{d} t} \bigl(m_k(g) \bigr) \leq& \int_{\OmegaLv} Q(\gchi, \gchi) \textrm{sgn}(g)(\boldsymbol{v}) \langle \boldsymbol{v} \rangle^{k} ~\dv \nonumber \\
	&~+ \left\|(Q_c(g,g) - Q_u(g,g)) \langle \boldsymbol{v} \rangle^{k} \right\|_{L^1(\OmegaLv)} \nonumber \\
	&~~+ \left\|\left(1 - \Pi^N_{2\Lv}\right) Q(\gchi, \gchi) \langle \boldsymbol{v} \rangle^{k} \right\|_{L^1(\OmegaLv)}. \label{moment_derivative_bound1}
	\end{align}
	
	Now, for the remaining integral involving $\textrm{sgn}(g)$, the weak form identity (\ref{Q^a_weak}) gives
	\begin{align*}
	& \int_{\OmegaLv} Q(\gchi, \gchi) \textrm{sgn}(g)(\boldsymbol{v}) \langle \boldsymbol{v} \rangle^{k} ~\dv \nonumber \\
	= &\int_{\OmegaLv} \gchi \left(\bar{a}_{i,j} \frac{\partial^2}{\partial v_i \partial v_j} \left(\textrm{sgn}(g)(\boldsymbol{v}) \langle \boldsymbol{v} \rangle^{k} \right) + 2 \bar{b}_{i} \frac{\partial}{\partial v_i} \left(\textrm{sgn}(g)(\boldsymbol{v}) \langle \boldsymbol{v} \rangle^{k} \right) \right) ~\dv. 
	\end{align*}
	In order to handle the derivative of the sgn function, it will be approximated by a second order regularisation of the jump.  In particular, for any $\delta_s> 0$, define the monotone increasing function $H_{\delta}: \mathbb{R} \to [-1, 1]$ by
	\begin{equation}
	H_{\delta}(y) := \begin{cases}
	-1, &\textrm{when }~ y \leq -\delta_s, \\
	\frac{1}{{\delta_s}^2} y^2 - \frac{2}{\delta_s} y, &\textrm{when }~ -\delta_s < y \leq 0, \\
	-\frac{1}{{\delta_s}^2} y^2 + \frac{2}{\delta_s} y, &\textrm{when }~ 0 < y < \delta_s, \\
	1, &\textrm{when }~ y \geq \delta_s.
	\end{cases} \label{H_delta}
	\end{equation}
	
	Then, after using $\textrm{sgn}(g) = \lim_{\delta_s \to 0} H_{\delta}$, differentiating with the product rule and noting that $\bar{a}_{i,j}$ is symmetric so that the cross-terms in the second order derivative double up,
	\begin{align*}
	&\bar{a}_{i,j} \frac{\partial^2}{\partial v_i \partial v_j} \left(\textrm{sgn}(g)(\boldsymbol{v}) \langle \boldsymbol{v} \rangle^{k} \right) + 2 \bar{b}_{i} \frac{\partial}{\partial v_i} \left(\textrm{sgn}(g)(\boldsymbol{v}) \langle \boldsymbol{v} \rangle^{k} \right) \\
	=& \lim_{\delta_s \to 0} \Biggl(\bar{a}_{i,j} \left(\frac{\partial^2 H_{\delta}}{\partial v_i \partial v_j} \langle \boldsymbol{v} \rangle^{k} + 2\frac{\partial H_{\delta}}{\partial v_i} \frac{\partial }{\partial v_j} \left(\langle \boldsymbol{v} \rangle^{k} \right) + H_{\delta} \frac{\partial^2}{\partial v_i \partial v_j} \left(\langle \boldsymbol{v} \rangle^{k} \right) \right) \\
	&~~~~~~~~ + 2 \bar{b}_{i} \left(\frac{\partial H_{\delta}}{\partial v_i} \langle \boldsymbol{v} \rangle^{k} + H_{\delta} \frac{\partial}{\partial v_i} \left(\langle \boldsymbol{v} \rangle^{k} \right) \right) \Biggr).
	\end{align*}
	So, assuming that Lebesgue's dominated convergence theorem can be used,
	\begin{align}
	& \int_{\OmegaLv} Q(\gchi, \gchi) \textrm{sgn}(g)(\boldsymbol{v}) \langle \boldsymbol{v} \rangle^{k} ~\dv \nonumber \\
	= & \lim_{\delta_s \to 0} \int_{\OmegaLv} \gchi \left(\bar{a}_{i,j} \frac{\partial^2 H_{\delta}}{\partial v_i \partial v_j} \langle \boldsymbol{v} \rangle^{k} + 2\bar{a}_{i,j} \frac{\partial H_{\delta}}{\partial v_i} \frac{\partial }{\partial v_j} \left(\langle \boldsymbol{v} \rangle^{k} \right) + 2 \bar{b}_{i} \frac{\partial H_{\delta}}{\partial v_i} \langle \boldsymbol{v} \rangle^{k} \right) \dv \nonumber \\
	&~ + \int_{\OmegaLv} \gchi \left(\bar{a}_{i,j} \frac{\partial^2}{\partial v_i \partial v_j} \left(\langle \boldsymbol{v} \rangle^{k} \right) + 2 \bar{b}_{i} \frac{\partial}{\partial v_i} \left(\langle \boldsymbol{v} \rangle^{k} \right) \right) \lim_{\delta_s \to 0} H_{\delta} ~\dv. \label{moment_sgn_limit}
	\end{align}
	\begin{rem}
		The parameter $\delta_s$ will be chosen sufficiently small to obtain estimates on the derivatives of the moments $m_k$.  In particular, taking the limit in the second integral in (\ref{moment_sgn_limit}) recovers $\textrm{sgn}(g)$ and estimates from \cite{D&V_1} can be used.  This will be seen after demonstrating that the first integral in (\ref{moment_sgn_limit}) converges to zero as $\delta_s \to 0$.
	\end{rem}
	
	To analyse first integral in (\ref{moment_sgn_limit}), first reduce the maximum order of the derivatives on $H_{\delta}$ by the divergence theorem.  In particular, if the vector $\bar{\boldsymbol{a}}_i$ is defined as the $i$th row of the matrix $\bar{a}$ so that $\bar{\boldsymbol{a}}_i = (\bar{a}_{i,1}, \ldots, \bar{a}_{i,d})$, for $i = 1, \ldots, d$,
	\begin{align*}
	\int_{\OmegaLv} \gchi \bar{a}_{i,j} \frac{\partial^2 H_{\delta}}{\partial v_i \partial v_j} \langle \boldsymbol{v} \rangle^{k} \dv = &\int_{\OmegaLv} \left(\gchi \langle \boldsymbol{v} \rangle^{k} \bar{\boldsymbol{a}}_{i} \right) \cdot \nabla \left(\frac{\partial  H_{\delta}}{\partial v_i} \right) \dv \\
	= &-\int_{\OmegaLv} \nabla \cdot \left(\gchi \langle \boldsymbol{v} \rangle^{k} \bar{\boldsymbol{a}}_{i} \right) \frac{\partial  H_{\delta}}{\partial v_i} \dv \\
	= &-\int_{\OmegaLv} \biggl(\frac{\partial}{\partial v_j} \left(\gchi \right) \langle \boldsymbol{v} \rangle^{k} \bar{a}_{i,j} \\
	&~~~~~~~~ + \gchi \frac{\partial}{\partial v_j} \left(\langle \boldsymbol{v} \rangle^{k} \right) \bar{a}_{i,j} + \gchi \langle \boldsymbol{v} \rangle^{k} \bar{b}_{i} \biggr) \frac{\partial  H_{\delta}}{\partial v_i} \dv.
	\end{align*}
	
	Then, by noting that some of the terms which appeared by the product rule in the last line here are the same as those in the first integral in expression (\ref{moment_sgn_limit}), it can then be written as
	\begin{align}
	&\int_{\OmegaLv} \gchi \left(\bar{a}_{i,j} \frac{\partial^2 H_{\delta}}{\partial v_i \partial v_j} \langle \boldsymbol{v} \rangle^{k} + 2\bar{a}_{i,j} \frac{\partial H_{\delta}}{\partial v_i} \frac{\partial }{\partial v_j} \left(\langle \boldsymbol{v} \rangle^{k} \right) + 2 \bar{b}_{i} \frac{\partial H_{\delta}}{\partial v_i} \langle \boldsymbol{v} \rangle^{k} \right) \dv \nonumber \\
	=& -\int_{\OmegaLv} \frac{\partial}{\partial v_j} \left(\gchi \right) \langle \boldsymbol{v} \rangle^{k} \bar{a}_{i,j} \frac{\partial  H_{\delta}}{\partial v_i} \dv \nonumber \\
	&~~ + \int_{\OmegaLv} \gchi \left(\bar{a}_{i,j} \frac{\partial H_{\delta}}{\partial v_i} \frac{\partial }{\partial v_j} \left(\langle \boldsymbol{v} \rangle^{k} \right) + \bar{b}_{i} \frac{\partial H_{\delta}}{\partial v_i} \langle \boldsymbol{v} \rangle^{k} \right) \dv \nonumber \\
	=& -\int_{\OmegaLv} \frac{\partial}{\partial v_j} \left(\gchi \right) \langle \boldsymbol{v} \rangle^{k} \bar{a}_{i,j} H'_{\delta}(g) \frac{\partial g}{\partial v_i} \dv \nonumber \\
	&~~ + \int_{\OmegaLv} \gchi \left(\bar{a}_{i,j} H'_{\delta}(g) \frac{\partial g}{\partial v_i} \frac{\partial }{\partial v_j} \left(\langle \boldsymbol{v} \rangle^{k} \right) + \bar{b}_{i} H'_{\delta}(g) \frac{\partial g}{\partial v_i} \langle \boldsymbol{v} \rangle^{k} \right) \dv. \label{moment_sgn_integral1}
	\end{align}
	Here, by using the product rule one more time,
	\begin{align*}
	-\int_{\OmegaLv} \frac{\partial}{\partial v_j} \left(\gchi \right) \langle \boldsymbol{v} \rangle^{k} \bar{a}_{i,j} H'_{\delta}(g) \frac{\partial g}{\partial v_i} \dv &= - \int_{\OmegaLv} \frac{\partial \Chi }{\partial v_j} g \langle \boldsymbol{v} \rangle^{k} \bar{a}_{i,j} H'_{\delta}(g) \frac{\partial g}{\partial v_i} \dv \\
	&~~~~~ - \int_{\OmegaLv} \Chi \frac{\partial g}{\partial v_j} \langle \boldsymbol{v} \rangle^{k} \bar{a}_{i,j} H'_{\delta}(g) \frac{\partial g}{\partial v_i} \dv \\
	&\leq -\int_{\Omega^C_{(1-\delchi)\Lv}} \frac{\partial \Chi }{\partial v_j} g \langle \boldsymbol{v} \rangle^{k} \bar{a}_{i,j} H'_{\delta}(g) \frac{\partial g}{\partial v_i} \dv,
	\end{align*}
	because $\frac{\partial \Chi }{\partial v_j} = 0$ when $\boldsymbol{v} \in \Omega_{(1-\delchi)\Lv}$ and the ellipticity of $\bar{a}$ in Proposition \ref{ellipticity_prop}, along with \mbox{\Large$\chi$}$ \geq 0$ and $H'_{\delta}(g) \geq 0$, gives
	\begin{equation*}
	-\int_{\OmegaLv} \Chi \frac{\partial g}{\partial v_j} \langle \boldsymbol{v} \rangle^{k} \bar{a}_{i,j} H'_{\delta}(g) \frac{\partial g}{\partial v_i} \dv \leq 0.\\
	\end{equation*}
	
	Now, it is hoped that the integral (\ref{moment_sgn_integral1}) goes to zero as $\delta_s \to 0$, so it can be coarsely bounded by its absolute value.  This means
	\begin{align}
	&\int_{\OmegaLv} \gchi \left(\bar{a}_{i,j} \frac{\partial^2 H_{\delta}}{\partial v_i \partial v_j} \langle \boldsymbol{v} \rangle^{k} + 2\bar{a}_{i,j} \frac{\partial H_{\delta}}{\partial v_i} \frac{\partial }{\partial v_j} \left(\langle \boldsymbol{v} \rangle^{k} \right) + 2 \bar{b}_{i} \frac{\partial H_{\delta}}{\partial v_i} \langle \boldsymbol{v} \rangle^{k} \right) \dv \nonumber \\
	\leq& \int_{\Omega^C_{(1-\delchi)\Lv}} \left|\frac{\partial \Chi }{\partial v_j} \right\|g| \langle \boldsymbol{v} \rangle^{k} \left|\bar{a}_{i,j} \right| \left|H'_{\delta}(g) \right| \left|\frac{\partial g}{\partial v_i} \right| \dv \nonumber \\
	&~~ + \int_{\OmegaLv} \left|\gchi \right| \left(\left|\bar{a}_{i,j} \right| \left|H'_{\delta}(g) \right| \left|\frac{\partial g}{\partial v_i} \right| \left|\frac{\partial }{\partial v_j} \left(\langle \boldsymbol{v} \rangle^{k} \right) \right| + \left|\bar{b}_{i} \right| \left|H'_{\delta}(g) \right| \left|\frac{\partial g}{\partial v_i} \right| \langle \boldsymbol{v} \rangle^{k} \right) \dv \nonumber \\
	\leq& 2 \Biggl(\int_{\{|g| \leq \delta_s\} \bigcap \Omega^C_{(1-\delchi)\Lv}} \left|\frac{\partial \Chi }{\partial v_j} \right|\langle \boldsymbol{v} \rangle^{k} \left|\bar{a}_{i,j} \right| \left|\frac{\partial g}{\partial v_i} \right| \dv \nonumber \\
	&~~~~~ + \int_{\{|g| \leq \delta_s\}} \left(\left|\bar{a}_{i,j} \right| \left|\frac{\partial g}{\partial v_i} \right| \left|\frac{\partial }{\partial v_j} \left(\langle \boldsymbol{v} \rangle^{k} \right) \right| + \left|\bar{b}_{i} \right| \left|\frac{\partial g}{\partial v_i} \right| \langle \boldsymbol{v} \rangle^{k} \right) \dv \Biggr), \label{moment_sgn_integral2}
	\end{align}
	where it's been used that the derivative $H'_{\delta}(g) = 0$ when $|g| > \delta_s$ to reduce the domain of integration.  Also, on this new domain, $|\gchi|, |g| \leq \delta_s$ and $|H'_{\delta}(g)| \leq \frac{2}{\delta_s}$, by definition of the assumed form of $H_{\delta}$ in (\ref{H_delta}).
	
	Here, by noting that $\left|a_{i,j} (\boldsymbol{v} - \boldsymbol{v}_*) \right| = |\boldsymbol{v} - \boldsymbol{v}_*|^{\lambda}\bigl\|\boldsymbol{v} - \boldsymbol{v}_*|^2 - (v - v_*)_i (v - v_*)_j \bigr|$ $\leq 2|\boldsymbol{v} - \boldsymbol{v}_*|^{\lambda + 2} \leq 2 \left(\langle \boldsymbol{v} \rangle^{\lambda + 2} + \langle \boldsymbol{v}_* \rangle^{\lambda + 2} \right)$, for any arbitrary set $\widetilde{\Omega} \subset \Rd$, 
	\begin{align}
	&\int_{\widetilde{\Omega}} \left|\frac{\partial \Chi }{\partial v_j} \right|\langle \boldsymbol{v} \rangle^{k} \left|\bar{a}_{i,j} \right| \left|\frac{\partial g}{\partial v_i} \right| \dv \nonumber \\
	\leq& \int_{\widetilde{\Omega}} \int_{\OmegaLv} \left|\frac{\partial \Chi }{\partial v_j} \right|\langle \boldsymbol{v} \rangle^{k} \left|\bar{a}_{i,j} (\boldsymbol{v} - \boldsymbol{v}_*) \right\|g(\boldsymbol{v}_*)| \left|\frac{\partial g}{\partial v_i} \right| ~\dv_* \dv \nonumber \\
	\leq&~ 2 \sum_{i, j = 1}^{d} \Biggl(\biggl(\int_{\widetilde{\Omega}} \left|\frac{\partial \Chi }{\partial v_j} \right|\left|\frac{\partial g}{\partial v_i} \right| \langle \boldsymbol{v} \rangle^{k + \lambda + 2} \dv \biggr) \biggl(\int_{\OmegaLv} |g(\boldsymbol{v}_*)| \dv_* \biggr) \nonumber \\
	&~~~~~~~~~~~ + \biggl(\int_{\widetilde{\Omega}} \left|\frac{\partial \Chi }{\partial v_j} \right|\left|\frac{\partial g}{\partial v_i} \right| \langle \boldsymbol{v} \rangle^{k} \dv \biggr) \biggl(\int_{\OmegaLv} |g(\boldsymbol{v}_*)| \langle \boldsymbol{v}_* \rangle^{\lambda + 2} \dv_* \biggr) \Biggr) \nonumber \\
	\leq&~ 2 d^2 \mathcal{O} \left(\frac{1}{\delchi} \right) \left(\|\left|\nabla g \right\||_{L^1_{k + \lambda + 2}(\widetilde{\Omega})} m_0(g) + \|\left|\nabla g \right\||_{L^1_{k}(\widetilde{\Omega})} m_{\lambda + 2}(g) \right), \label{moment_sgn_integral2a}
	\end{align}
	where the $\mathcal{O} \left(\frac{1}{\delchi} \right)$ is a result of the fact that \mbox{\Large$\chi$} changes smoothly from a value of $1$ to $0$ in a space of $\mathcal{O} \left(\delchi \right)$.
	
	Similarly, since $\left|\frac{\partial }{\partial v_j} \left(\langle \boldsymbol{v} \rangle^{k} \right) \right| = \left|k v_j \langle \boldsymbol{v} \rangle^{k-2} \right| \leq k \langle \boldsymbol{v} \rangle^{k-1}$,
	\begin{align}
	&\int_{\widetilde{\Omega}} \left|\bar{a}_{i,j} \right| \left|\frac{\partial g}{\partial v_i} \right| \left|\frac{\partial }{\partial v_j} \left(\langle \boldsymbol{v} \rangle^{k} \right) \right| \dv \nonumber \\
	\leq&~ 2k d^2 \left(\|\left|\nabla g \right\||_{L^1_{k + \lambda + 1}(\widetilde{\Omega})} m_0(g) + \|\left|\nabla g \right\||_{L^1_{k - 1}(\widetilde{\Omega})} m_{\lambda + 2}(g) \right). \label{moment_sgn_integral2b}
	\end{align}
	
	Finally, by noting that $\left|b_i (\boldsymbol{v} - \boldsymbol{v}_*) \right| = |\boldsymbol{v} - \boldsymbol{v}_*|^{\lambda}\left|(v - v_*)_i \right| \leq |\boldsymbol{v} - \boldsymbol{v}_*|^{\lambda + 1}$ $\leq \langle \boldsymbol{v} \rangle^{\lambda + 1} + \langle \boldsymbol{v}_* \rangle^{\lambda + 1}$,
	\begin{align}
	&\int_{\widetilde{\Omega}} \left|\bar{b}_{i} \right| \left|\frac{\partial g}{\partial v_i} \right| \langle \boldsymbol{v} \rangle^{k} \dv \nonumber \\
	\leq&~ d^2 \left(\|\left|\nabla g \right\||_{L^1_{k + \lambda + 1}(\widetilde{\Omega})} m_0(g) + \|\left|\nabla g \right\||_{L^1_{k}(\widetilde{\Omega})} m_{\lambda + 1}(g) \right), \label{moment_sgn_integral2c}
	\end{align}
	
	So, by using the bounds (\ref{moment_sgn_integral2a}-\ref{moment_sgn_integral2c}) with the appropriate domains replacing $\widetilde{\Omega}$ in (\ref{moment_sgn_integral2}),
	\begin{align}
	&\int_{\OmegaLv} \gchi \left(\bar{a}_{i,j} \frac{\partial^2 H_{\delta}}{\partial v_i \partial v_j} \langle \boldsymbol{v} \rangle^{k} + 2\bar{a}_{i,j} \frac{\partial H_{\delta}}{\partial v_i} \frac{\partial }{\partial v_j} \left(\langle \boldsymbol{v} \rangle^{k} \right) + 2 \bar{b}_{i} \frac{\partial H_{\delta}}{\partial v_i} \langle \boldsymbol{v} \rangle^{k} \right) \dv \nonumber \\
	\leq& 2 \Biggl(2 d^2 \mathcal{O} \left(\frac{1}{\delchi} \right) \Bigl(\|\left|\nabla g \right\||_{L^1_{k + \lambda + 2}\left(\{|g| \leq \delta_s\} \bigcap \Omega^C_{(1-\delchi)\Lv} \right)} m_0(g) \nonumber \\
	&~~~~~~~~~~~~~~~~~~~~~ + \|\left|\nabla g \right\||_{L^1_{k}\left(\{|g| \leq \delta_s\} \bigcap \Omega^C_{(1-\delchi)\Lv} \right)} m_{\lambda + 2}(g) \Bigr) \nonumber \\
	& + (2k + 1) d^2 \left(\|\left|\nabla g \right\||_{L^1_{k + \lambda + 1}(\{|g| \leq \delta_s\})} m_0(g) + \|\left|\nabla g \right\||_{L^1_{k - 1}(\{|g| \leq \delta_s\})} m_{\lambda + 2}(g) \right) \Biggr). \label{moment_sgn_integral3}
	\end{align}

	This means that, since it is also assumed that the gradient of $g$ is bounded inside the domain, the integral in (\ref{moment_sgn_integral3}) is an $\mathcal{O} (\delta_s)$ term, because $\delchi$ will remain fixed throughout and the domain of the $L^1$-norms of the gradients will shrink to sets of zero of measure.  
	\begin{rem}
		This assumption on the gradient was not required in \cite{BoltzmannConvergence} when devising error estimates for the Boltzmann equation scheme since there are no derivatives involved there.  It seems that for the Landau equation method, however, this is an additional assumption to be added to the previously stated upper Maxwellian bound (\ref{Gaussian_bound}).
	\end{rem}  
	
	Now, taking the limit as $\delta_s \to 0$ in (\ref{moment_sgn_limit}) and using an identity from the proof of Theorem 3 in \cite{D&V_1} on the non-negligible integral gives 
	\begin{align*}
	& \int_{\OmegaLv} Q(\gchi, \gchi) \textrm{sgn}(g)(\boldsymbol{v}) \langle \boldsymbol{v} \rangle^{k} ~\dv \nonumber \\
	= &~k \int_{\OmegaLv} \int_{\Rd} \gchi(\boldsymbol{v}) \gchi(\boldsymbol{v}_*) |\boldsymbol{v} - \boldsymbol{v}_*|^{\lambda} \langle \boldsymbol{v} \rangle^{k - 2} \biggl(-2|\boldsymbol{v}|^{2} + 2|\boldsymbol{v}_*|^{2} \nonumber \\
	&~~~~~~~~~~~~~~~~~~~~ + (k - 2) \frac{|\boldsymbol{v}|^{2} |\boldsymbol{v}_*|^{2} - (\boldsymbol{v} \cdot \boldsymbol{v}_*)^2}{1 + |\boldsymbol{v}|^{2}} \biggr) \textrm{sgn}(g)(\boldsymbol{v}) ~\dv_* \dv. 
	\end{align*}
	Here, in order to start bounding the quantity in brackets, some care must be taken to ensure all other terms are positive.  This is done by substituting $\gchi(\boldsymbol{v})\textrm{sgn}(g)(\boldsymbol{v}) = |\gchi(\boldsymbol{v})|$ and $\gchi(\boldsymbol{v}_*) = |\gchi(\boldsymbol{v}_*)| - 2 \gchi^-(\boldsymbol{v}_*)$ to give
	\begin{align}
	& \int_{\OmegaLv} Q(\gchi, \gchi) \textrm{sgn}(g)(\boldsymbol{v}) \langle \boldsymbol{v} \rangle^{k} ~\dv \nonumber \\
	=&~k \Biggl(\int_{\OmegaLv} \int_{\Rd} |\gchi(\boldsymbol{v})\|\gchi(\boldsymbol{v}_*)\|\boldsymbol{v} - \boldsymbol{v}_*|^{\lambda} \langle \boldsymbol{v} \rangle^{k - 2} \biggl(-2|\boldsymbol{v}|^{2} + 2|\boldsymbol{v}_*|^{2} \nonumber \\
	&~~~~~~~~~~~~~~~~~~~~ + (k - 2) \frac{|\boldsymbol{v}|^{2} |\boldsymbol{v}_*|^{2} - (\boldsymbol{v} \cdot \boldsymbol{v}_*)^2}{1 + |\boldsymbol{v}|^{2}} \biggr)  ~\dv_* \dv \nonumber \\
	&~~~~+ 2 \int_{\OmegaLv} \int_{\Rd} |\gchi(\boldsymbol{v})| \gchi^-(\boldsymbol{v}_*) |\boldsymbol{v} - \boldsymbol{v}_*|^{\lambda} \langle \boldsymbol{v} \rangle^{k - 2} \biggl(2|\boldsymbol{v}|^{2} - 2|\boldsymbol{v}_*|^{2} \nonumber \\
	&~~~~~~~~~~~~~~~~~~~~ - (k - 2) \frac{|\boldsymbol{v}|^{2} |\boldsymbol{v}_*|^{2} - (\boldsymbol{v} \cdot \boldsymbol{v}_*)^2}{1 + |\boldsymbol{v}|^{2}} \biggr) ~\dv_* \dv \Biggr). \label{Q_sgn_bound1}
	\end{align}
	
	Now, by discarding the negative terms in the second integral in (\ref{Q_sgn_bound1}),
	\begin{align}
	& 2 \int_{\OmegaLv} \int_{\Rd} |\gchi(\boldsymbol{v})| \gchi^-(\boldsymbol{v}_*) |\boldsymbol{v} - \boldsymbol{v}_*|^{\lambda} \langle \boldsymbol{v} \rangle^{k - 2} \biggl(2|\boldsymbol{v}|^{2} - 2|\boldsymbol{v}_*|^{2} \nonumber \\
	&~~~~~~~~~~~~~~~~~~~~ - (k - 2) \frac{|\boldsymbol{v}|^{2} |\boldsymbol{v}_*|^{2} - (\boldsymbol{v} \cdot \boldsymbol{v}_*)^2}{1 + |\boldsymbol{v}|^{2}} \biggr) ~\dv_* \dv \nonumber \\
	\leq & ~2 \int_{\OmegaLv} \int_{\Rd} |\gchi(\boldsymbol{v})| \gchi^-(\boldsymbol{v}_*) |\boldsymbol{v} - \boldsymbol{v}_*|^{\lambda} \langle \boldsymbol{v} \rangle^{k - 2} \left(2|\boldsymbol{v}|^{2} \right) ~\dv_* \dv \nonumber \\
	\leq & ~4 \int_{\OmegaLv} \int_{\Rd} |\gchi(\boldsymbol{v})\|\gchi(\boldsymbol{v}_*)\|\boldsymbol{v} - \boldsymbol{v}_*|^{\lambda} \langle \boldsymbol{v} \rangle^{k} ~\dv_* \dv \nonumber \\
	\leq& ~4 \Biggl(\biggl(\int_{\OmegaLv} |\gchi(\boldsymbol{v})| \langle \boldsymbol{v} \rangle^{k + \lambda} ~\dv \biggr) \biggl(\int_{\Rd} \left| \gchi(\boldsymbol{v}_*) \right| ~\dv_* \biggr) \nonumber \\
	&~~~~~~~~~~~~~~ + \biggl(\int_{\OmegaLv} |\gchi(\boldsymbol{v})| \langle \boldsymbol{v} \rangle^{k} ~\dv \biggr) \biggl(\int_{\Rd} \left| \gchi(\boldsymbol{v}_*) \right| \langle \boldsymbol{v}_* \rangle^{\lambda} ~\dv_* \biggr) \Biggr) \nonumber \\
	\leq& ~4 \Biggl(m_{k + \lambda}(g) \int_{\Rd} \left| \gchi(\boldsymbol{v}_*) \right| ~\dv_* + m_{k}(g) \int_{\Rd} \left| \gchi(\boldsymbol{v}_*) \right| \langle \boldsymbol{v}_* \rangle^{\lambda} ~\dv_* \Biggr), \label{Q_neg_bound}
	\end{align}
	where the identity $|\boldsymbol{v} - \boldsymbol{v}_*|^{\lambda} \leq |\boldsymbol{v}|^{\lambda} + |\boldsymbol{v}_*|^{\lambda} \leq \langle \boldsymbol{v} \rangle^{\lambda} + \langle \boldsymbol{v}_* \rangle^{\lambda}$ has been used.  In order to obtain the definitions of $m_k$ in (\ref{mk_def}), it was also used that $\left| \gchi(\boldsymbol{v}) \right| \leq \left| g(\boldsymbol{v}) \right|$ when $\boldsymbol{v} \in \OmegaLv$ and the domain of integration expanded to all of $\Rd$.
	
	Then, since $\left| \gchi(\boldsymbol{v}_*) \right| = 0$ when $\boldsymbol{v}_* \notin \OmegaLv$; $\left| \gchi(\boldsymbol{v}_*) \right| \leq \left| g(\boldsymbol{v}_*) \right|$ when $\boldsymbol{v}_* \in \OmegaLv$; and noting that $\lambda \leq 1$ implies $\langle \boldsymbol{v} \rangle^{\lambda} \leq \langle \boldsymbol{v} \rangle^{2}$,
	\begin{align*}
	\int_{\Rd} \left| \gchi(\boldsymbol{v}_*) \right| \langle \boldsymbol{v}_* \rangle^{\lambda} ~\dv_* \leq \int_{\OmegaLv} \left| g(\boldsymbol{v}_*) \right| \langle \boldsymbol{v}_* \rangle^{2} ~\dv_* &\leq \frac{1}{1 - 2\varepsilon} \int_{\OmegaLv} g_0(\boldsymbol{v})\langle\boldsymbol{v}\rangle^2 ~\dv \\
	&\leq \frac{1}{1 - 2\varepsilon} \|g_0\|_{L^1_2(\OmegaLv)},
	\end{align*}
	by using identity (\ref{stability_consequence}) resulting from the stability condition (\ref{stability}).  So, since $1 \leq \langle \boldsymbol{v} \rangle^{2}$, the same argument bounds the first integral in expression (\ref{Q_neg_bound}) and this means
	\begin{align}
	& 2 \int_{\OmegaLv} \int_{\Rd} |\gchi(\boldsymbol{v})| \gchi^-(\boldsymbol{v}_*) |\boldsymbol{v} - \boldsymbol{v}_*|^{\lambda} \langle \boldsymbol{v} \rangle^{k - 2} \biggl(2|\boldsymbol{v}|^{2} - 2|\boldsymbol{v}_*|^{2} \nonumber \\
	&~~~~~~~~~~~~~~~~~~~~ - (k - 2) \frac{|\boldsymbol{v}|^{2} |\boldsymbol{v}_*|^{2} - (\boldsymbol{v} \cdot \boldsymbol{v}_*)^2}{1 + |\boldsymbol{v}|^{2}} \biggr) ~\dv_* \dv \nonumber \\
	\leq & ~\frac{4}{1 - 2\varepsilon} \left(m_{k + \lambda}(g) + m_{k}(g)\right) \|g_0\|_{L^1_2(\OmegaLv)}. \label{Q_sgn_integral2_bound}
	\end{align}
	
	Also, by using further details of the proof of Theorem 3 in \cite{D&V_1}, since $k > 2$, there exist constants $K_k, C^1_k > 0$, each depending on $k$, such that the first integral in (\ref{Q_sgn_bound1}) can be bounded as
	\begin{align}
	& \int_{\OmegaLv} \int_{\Rd} |\gchi(\boldsymbol{v})\|\gchi(\boldsymbol{v}_*)\|\boldsymbol{v} - \boldsymbol{v}_*|^{\lambda} \langle \boldsymbol{v} \rangle^{k - 2} \biggl(-2|\boldsymbol{v}|^{2} + 2|\boldsymbol{v}_*|^{2} \nonumber \\
	&~~~~~~~~~~~~~~~~~~~~ + (k - 2) \frac{|\boldsymbol{v}|^{2} |\boldsymbol{v}_*|^{2} - (\boldsymbol{v} \cdot \boldsymbol{v}_*)^2}{1 + |\boldsymbol{v}|^{2}} \biggr)  ~\dv_* \dv \nonumber \\
	\leq & - K_k \int_{\OmegaLv} \int_{\Rd} |\gchi(\boldsymbol{v})\|\gchi(\boldsymbol{v}_*)\|\boldsymbol{v} - \boldsymbol{v}_*|^{\lambda} \langle \boldsymbol{v} \rangle^{k} ~\dv_* \dv \nonumber \\
	& + C^1_k \int_{\OmegaLv} \int_{\Rd} |\gchi(\boldsymbol{v})\|\gchi(\boldsymbol{v}_*)\|\boldsymbol{v} - \boldsymbol{v}_*|^{\lambda} \left(\langle \boldsymbol{v}_* \rangle \langle \boldsymbol{v} \rangle^{k - 1} + \langle \boldsymbol{v} \rangle \langle \boldsymbol{v}_* \rangle ^{k - 1} \right) ~\dv_* \dv. \label{Q_sgn_integral1_bound}
	\end{align}
	\begin{rem}
		The results used from \cite{D&V_1} here include the use of their Lemma 1, which they state replaces the Povzner lemma associated to the Boltzmann equation.
	\end{rem}
	Here, by again using $|\boldsymbol{v} - \boldsymbol{v}_*|^{\lambda} \leq \langle \boldsymbol{v} \rangle^{\lambda} + \langle \boldsymbol{v}_* \rangle^{\lambda}$ in the second integral in (\ref{Q_sgn_integral1_bound}),
	\begin{align}
	& \int_{\OmegaLv} \int_{\Rd} |\gchi(\boldsymbol{v})\|\gchi(\boldsymbol{v}_*)\|\boldsymbol{v} - \boldsymbol{v}_*|^{\lambda} \left(\langle \boldsymbol{v}_* \rangle \langle \boldsymbol{v} \rangle^{k - 1} + \langle \boldsymbol{v} \rangle \langle \boldsymbol{v}_* \rangle ^{k - 1} \right) ~\dv_* \dv \nonumber \\
	\leq & \int_{\OmegaLv} \int_{\Rd} |\gchi(\boldsymbol{v})\|\gchi(\boldsymbol{v}_*)| (\langle \boldsymbol{v} \rangle^{k + \lambda - 1} \langle \boldsymbol{v}_* \rangle + \langle \boldsymbol{v} \rangle^{\lambda + 1} \langle \boldsymbol{v}_* \rangle^{k - 1} \nonumber\\
	&~~~~~~~~~~~~~~~~~~~~~~~~~~~~~~~~~~~ + \langle \boldsymbol{v} \rangle^{k - 1} \langle \boldsymbol{v}_* \rangle^{\lambda + 1} + \langle \boldsymbol{v} \rangle \langle \boldsymbol{v}_* \rangle^{k + \lambda - 1}) ~\dv_* \dv \nonumber \\
	\leq & ~m_{k + \lambda - 1}(g)m_{1}(g) + m_{\lambda + 1}(g)m_{k - 1}(g) + m_{k - 1}(g)m_{\lambda + 1}(g) + m_{1}(g)m_{k + \lambda - 1}(g) \nonumber \\
	= & ~2 \bigl(m_{k + \lambda - 1}(g)m_{1}(g) + m_{k - 1}(g)m_{\lambda + 1}(g) \bigr), \label{Q_sgn_integral1b_bound}
	\end{align}
	where the last inequality comes from increasing the domain in the outer integral to all of $\Rd$ to match the definition of $m_k$ in (\ref{mk_def}) and also bounding the extension operator \mbox{\Large$\chi$} by 1.
	
	Finally, for the first integral in (\ref{Q_sgn_integral1_bound}) note that, for the constant $c_{\lambda} := \min(1, 2^{\lambda - 1}) = 2^{\lambda - 1} > 0$ (since $\lambda \leq 1$, so that $2^{\lambda - 1} \leq 1$),
	\begin{equation*}
	|\boldsymbol{v} - \boldsymbol{v}_*|^{\lambda} \geq c_{\lambda}(1 + |\boldsymbol{v}|^2)^{\frac{\lambda}{2}} - (1 + |\boldsymbol{v}_*|^2)^{\frac{\lambda}{2}} \geq c_{\lambda} \langle \boldsymbol{v} \rangle^{\lambda} - \langle \boldsymbol{v}_* \rangle^2,
	\end{equation*}
	because $\lambda \leq 1$ implies $(1 + |\boldsymbol{v}_*|^2)^{\frac{\lambda}{2}} \leq 1 + |\boldsymbol{v}_*|^2$.
	
	So, the first integral in (\ref{Q_sgn_integral1_bound}) can be bounded by
	\begin{align*}
	& - \int_{\OmegaLv} \int_{\Rd} |\gchi(\boldsymbol{v})\|\gchi(\boldsymbol{v}_*)\|\boldsymbol{v} - \boldsymbol{v}_*|^{\lambda} \langle \boldsymbol{v} \rangle^{k} ~\dv_* \dv \nonumber \\
	\leq & \int_{\OmegaLv} \int_{\Rd} |\gchi(\boldsymbol{v})\|\gchi(\boldsymbol{v}_*)| \langle \boldsymbol{v}_* \rangle^{2} \langle \boldsymbol{v} \rangle^{k} ~\dv_* \dv \nonumber \\
	&~~~~~~~~~~~~~~~~ - c_{\lambda} \int_{\OmegaLv} \int_{\Rd} |\gchi(\boldsymbol{v})\|\gchi(\boldsymbol{v}_*)| \langle \boldsymbol{v} \rangle^{\lambda} \langle \boldsymbol{v} \rangle^{k} ~\dv_* \dv \nonumber \\
	= &  \left(\int_{\Rd} |\gchi(\boldsymbol{v})| \langle \boldsymbol{v} \rangle^{k} ~\dv\right) \left(\int_{\OmegaLv}|\gchi(\boldsymbol{v}_*)| \langle \boldsymbol{v_*} \rangle^{2} ~\dv_*\right) \nonumber \\
	&~~~~~~~~~~~~~~~~ - c_{\lambda} \left(\int_{\Rd} |\gchi(\boldsymbol{v})| \langle \boldsymbol{v} \rangle^{k + \lambda} ~\dv\right) \left(\int_{\Rd} |\gchi(\boldsymbol{v}_*)| ~\dv_*\right), \nonumber 
	\end{align*}
	by taking advantage of the fact that $\gchi(\boldsymbol{v}) = 0$ when $\boldsymbol{v} \notin \OmegaLv$ and so the integral domains can be switched between $\OmegaLv$ and $\Rd$.  Then, by using identity (\ref{stability_consequence}) from the stability condition again on the first integral with respect to $\boldsymbol{v}_*$ here, and also by bounding the extension operators in that integral by 1,
	\begin{align}
	& - \int_{\OmegaLv} \int_{\Rd} |\gchi(\boldsymbol{v})\|\gchi(\boldsymbol{v}_*)\|\boldsymbol{v} - \boldsymbol{v}_*|^{\lambda} \langle \boldsymbol{v} \rangle^{k} ~\dv_* \dv \nonumber \\
	\leq & ~\frac{1}{1 - 2\varepsilon} \left(\int_{\Rd} |g(\boldsymbol{v})| \langle \boldsymbol{v} \rangle^{k} ~\dv\right) \left(\int_{\OmegaLv} g_0(\boldsymbol{v}_*) \langle \boldsymbol{v}_* \rangle^{2} ~\dv_*\right) \nonumber \\
	&~~~~~~~~~~~~~~~ - c_{\lambda} \left(\int_{\Rd} |\gchi(\boldsymbol{v})| \langle \boldsymbol{v} \rangle^{k + \lambda} ~\dv\right) \left(\int_{\Rd} |\gchi(\boldsymbol{v}_*)| ~\dv_*\right) \nonumber \\
	\leq & ~\frac{1}{1 - 2\varepsilon} m_{k}(g) \|g_0\|_{L^1_2(\OmegaLv)} - c_{\lambda} m_{0}\left(\gchi \right) m_{k + \lambda}\left(\gchi \right). \label{Q_sgn_integral1a_bound}
	\end{align}
	
	This means, by retracing steps and first using the bounds (\ref{Q_sgn_integral1a_bound}) and (\ref{Q_sgn_integral1b_bound}) in (\ref{Q_sgn_integral1_bound}),
	\begin{align*}
	& \int_{\OmegaLv} \int_{\Rd} |\gchi(\boldsymbol{v})\|\gchi(\boldsymbol{v}_*)\|\boldsymbol{v} - \boldsymbol{v}_*|^{\lambda} \langle \boldsymbol{v} \rangle^{k - 2} \biggl(-2|\boldsymbol{v}|^{2} + 2|\boldsymbol{v}_*|^{2} \nonumber \\
	&~~~~~~~~~~~~~~~~~~~~ + (k - 2) \frac{|\boldsymbol{v}|^{2} |\boldsymbol{v}_*|^{2} - (\boldsymbol{v} \cdot \boldsymbol{v}_*)^2}{1 + |\boldsymbol{v}|^{2}} \biggr)  ~\dv_* \dv \nonumber \\
	\leq & ~\frac{K_k}{1 - 2\varepsilon} m_{k}(g) \|g_0\|_{L^1_2(\OmegaLv)} - c_{\lambda} K_k m_{0}\left(\gchi \right) m_{k + \lambda}\left(\gchi \right) \nonumber \\
	& + 2C^1_k \bigl(m_{k + \lambda - 1}(g)m_{1}(g) + m_{k - 1}(g)m_{\lambda + 1}(g) \bigr);
	\end{align*}
	then using this and (\ref{Q_sgn_integral2_bound}) in (\ref{Q_sgn_bound1}) gives the bound on the remaining integral with the sgn$(g)$ term in (\ref{moment_derivative_bound1}) as
	\begin{align*}
	& \int_{\OmegaLv} Q(\gchi, \gchi) \textrm{sgn}(g)(\boldsymbol{v}) \langle \boldsymbol{v} \rangle^{k} ~\dv \nonumber \\
	\leq & k \Biggl(~\frac{K_k}{1 - 2\varepsilon} m_{k}(g) \|g_0\|_{L^1_2(\OmegaLv)} - c_{\lambda} K_k m_{0}\left(\gchi \right) m_{k + \lambda}\left(\gchi \right) \nonumber \\
	&~~~ + 2C^1_k \bigl(m_{k + \lambda - 1}(g)m_{1}(g) + m_{k - 1}(g)m_{\lambda + 1}(g) \bigr) \\
	&~~~ + \frac{4}{1 - 2\varepsilon} \bigl(m_{k + \lambda}(g) + m_{k}(g) \bigr) \|g_0\|_{L^1_2(\OmegaLv)} \Biggr) \\
	= & k \Biggl(- c_{\lambda} K_k m_{0}\left(\gchi \right) m_{k + \lambda}\left(\gchi \right) + 2C^1_k \bigl(m_{k + \lambda - 1}(g)m_{1}(g) + m_{k - 1}(g)m_{\lambda + 1}(g) \bigr) \\
	&~~~ + \frac{4}{1 - 2\varepsilon} \biggl(m_{k + \lambda}(g) + \Bigl(1 + \frac{1}{4} K_k \Bigr)m_{k}(g) \biggr) \|g_0\|_{L^1_2(\OmegaLv)} \Biggr) \\
	\leq &  k \Biggl( - c_{\lambda} K_k m_{0}\left(\gchi \right) m_{k + \lambda}\left(\gchi \right) + 2C^1_k \bigl(m_{k + \lambda - 1}(g)m_{1}(g) + m_{k - 1}(g)m_{\lambda + 1}(g) \bigr) \\
	&~~~ + \frac{4(2 + \frac{1}{4} K_k)}{1 - 2\varepsilon} m_{k + \lambda}(g) \|g_0\|_{L^1_2(\OmegaLv)} \Biggr),
	\end{align*}
	because $m_k(g) \leq m_{k + \lambda}$.  
	
	Next, for some $\varepsilon_{\chi} \in (0, 1)$ that can be chosen as small as necessary, choose $\varepsilon$ to satisfy
	\begin{equation*}
	\frac{1}{1 - 2\varepsilon} \leq \frac{(1 - 2\varepsilon_{\chi}) c_{\lambda} K_k m_0(g_0)}{4(2 + \frac{1}{4} K_k)\|g_0\|_{L^1_2(\OmegaLv)}},
	\end{equation*}
	which is equivalent to taking $\varepsilon \leq \varepsilon_0$ with $\varepsilon_0$ chosen as
	\begin{equation}
	\varepsilon_0 := \frac{1}{2} \left(1 - \frac{4(2 + \frac{1}{4} K_k)\|g_0\|_{L^1_2(\OmegaLv)}}{(1 - 2\varepsilon_{\chi}) c_{\lambda} K_k m_0(g_0)} \right). \label{epsilon_0}
	\end{equation}
	If this is true then, by using results (\ref{moment_cutoff_bound_0}) and (\ref{moment_cutoff_bound_k}) from the appendix to bound $- m_{0}\left(\gchi \right) m_{k + \lambda}\left(\gchi \right)$ from above,
	\begin{align} \label{moment_derivative_bound_term1}
	& \int_{\OmegaLv} Q(\gchi, \gchi) \textrm{sgn}(g)(\boldsymbol{v}) \langle \boldsymbol{v} \rangle^{k} ~\dv \nonumber \\
	\leq &~ - (1 - \varepsilon_{\chi})^2 k c_{\lambda} K_k m_{0}(g_0) m_{k + \lambda}(g) + (1 - 2\varepsilon_{\chi}) k c_{\lambda} K_k m_0(g_0) m_{k + \lambda}(g) \nonumber \\
	&~~~+ k 2C^1_k \bigl(m_{k + \lambda - 1}(g)m_{1}(g) + m_{k - 1}(g)m_{\lambda + 1}(g) \bigr) \\
	= & ~ - {\varepsilon_{\chi}}^2 k c_{\lambda} K_k m_{0}(g_0) m_{k + \lambda}(g) + 2k C^1_k\bigl(m_{k + \lambda - 1}(g)m_{1}(g) + m_{k - 1}(g)m_{\lambda + 1}(g) \bigr). \nonumber 
		\end{align}
	
	Now, returning to the bound (\ref{moment_derivative_bound1}) on the derivatives of the moments of $g$, by the Cauchy-Schwarz inequality,
	\begin{align*}
	\left\|(Q_c(g,g) - Q_u(g,g)) \langle \boldsymbol{v} \rangle^{k} \right\|_{L^1(\OmegaLv)} 
	&\leq \left\|(Q_c(g,g) - Q_u(g,g)) \langle \boldsymbol{v} \rangle^{k} \right\|_{L^2(\OmegaLv)} \left\|1 \right\|_{L^2(\OmegaLv)} \\
	&=(2 \Lv)^{\frac{d}{2}} \left\|(Q_c(g,g) - Q_u(g,g)) \langle \boldsymbol{v} \rangle^{k} \right\|_{L^2(\OmegaLv)}.
	\end{align*}
	Then, using the result from Theorem \ref{Q_moment_theorem} gives that, for any $k' > 2$,
	\begin{align*}
	&\left\|(Q_c(g,g) - Q_u(g,g)) \langle \boldsymbol{v} \rangle^{k} \right\|_{L^1(\OmegaLv)} \\
	\leq &~ \frac{C_d (2 \Lv)^{\frac{d}{2}}}{\sqrt{2k + 1}} \biggl(O_{-k} \bigl| \bigl| \left(\Pi^N_{2\Lv} - 1 \right) Q(\gchi, \gchi) \bigr| \bigr|_{L^2(\OmegaLv)} \\
	&~~~~~~~~~~~~~~~~~ + O_{\frac{d}{2} + k' - k} \bigl(m_{0}(g) m_{k' + \lambda}(g) + m_{\lambda}(g) m_{k'}(g) \nonumber \\
	&~~~~~~~~~~~~~~~~~~~~~~~~~~~~~~~ + m_{2}(g) m_{k' + \lambda - 2}(g) + m_{\lambda + 2}(g) m_{k' - 2}(g)\bigr)\biggr).
	\end{align*}
	So, by choosing $k' = k > 2$, which implies $C_d \left(2 \Lv\right)^{\frac{d}{2}}O_{\frac{d}{2} + k' - k} = C^2_d$ for some new constant $C^2_d$,
	\begin{align}
	&\left\|(Q_c(g,g) - Q_u(g,g)) \langle \boldsymbol{v} \rangle^{k} \right\|_{L^1(\OmegaLv)} \nonumber \\
	\leq &~\frac{1}{\sqrt{2k + 1}} \biggl(2^{\frac{d}{2}} C_{d} \mathcal{O} \left({\Lv}^{k + \frac{d}{2}} \right) \bigl| \bigl| \left(\Pi^N_{2\Lv} - 1 \right) Q(\gchi, \gchi)  \bigr| \bigr|_{L^2(\OmegaLv)} \nonumber \\
	&~~~~~~~~~~~~~~~ + C^2_d \bigl(m_{0}(g) m_{k + \lambda}(g) + m_{\lambda}(g) m_{k}(g) \nonumber \\
	&~~~~~~~~~~~~~~~~~~~~~~~~~~ + m_{2}(g) m_{k + \lambda - 2}(g) + m_{\lambda + 2}(g) m_{k - 2}(g) \bigr)\biggr). \label{moment_derivative_bound_term2_0}
	\end{align}
	
	Also, by noting that $\lambda \leq 1$ implies, $m_{\lambda}(g) \leq m_{1}(g)$, $m_{k + \lambda - 2}(g) \leq m_{k - 1}(g)$ and $m_{\lambda + 2}(g) \leq m_3(g)$ and taking $k \geq 3$, note that
	\begin{align}
	m_{\lambda}(g) m_{k}(g) &+ m_{2}(g) m_{k + \lambda - 2}(g) + m_{\lambda + 2}(g) m_{k - 2}(g)  \nonumber \\
&\leq \sum_{j = 0}^{2} m_{j+1}(g)m_{k-j}(g)	\ \leq \ \sum_{j = 0}^{k - 1} {k \choose j} m_{j+1}(g)m_{k-j}(g) \nonumber \\
	&\qquad \ \ \leq C_k \bigl(m_0(g) + m_k(g) \bigr), \label{Zk_bound}
	\end{align}
	for some constant $C_k$ which depends only on $k$.  Then, using this in the bound (\ref{moment_derivative_bound_term2_0}) gives
	\begin{align}
	&\left\|(Q_c(g,g) - Q_u(g,g)) \langle \boldsymbol{v} \rangle^{k} \right\|_{L^1(\OmegaLv)} \nonumber \\
	&\qquad\ \leq \frac{1}{\sqrt{2k + 1}} \biggl(2^{\frac{d}{2}} C_{d} \mathcal{O} \left({\Lv}^{k + \frac{d}{2}} \right) \bigl| \bigl| \left(\Pi^N_{2\Lv} - 1 \right) Q(\gchi, \gchi)  \bigr| \bigr|_{L^2(\OmegaLv)} \nonumber \\
	&~~~~~~~~~~~~~~~ + C^2_d \Bigl(m_{0}(g) m_{k + \lambda}(g) + C_k \bigl(m_0(g) + m_k(g) \bigr) \Bigr)\biggr). \label{moment_derivative_bound_term2}
	\end{align}
	
	Next, by again using the Cauchy-Schwarz inequality on the remaining term in (\ref{moment_derivative_bound1}) as well as the bound on $\left\| \langle \boldsymbol{v} \rangle^{k} \right\|_{L^2(\OmegaLv)}$ in (\ref{v^k_L2_bound}),
	\begin{align}
	&\left\|\left(1\! -\! \Pi^N_{2\Lv}\right) \!Q(\gchi, \gchi) \langle \boldsymbol{v} \rangle^{k}\! \right\|_{L^1(\OmegaLv)} 
	\!\leq \! \left\|\left(1 \!-\! \Pi^N_{2\Lv}\right) Q(\gchi, \gchi) \right\|_{L^2(\OmegaLv)} \left\| \langle \boldsymbol{v} \rangle^{k}\! \right\|_{L^2(\OmegaLv)} \nonumber \\
	&\qquad\ \ \leq  \left(\frac{\omega_{d-1} d^{\frac{d}{2}}}{2k + 1} \sum_{j = 0}^{2k} {2k + 1 \choose j + 1} {\Lv}^{j}\right)^{\frac{1}{2}} {\Lv}^{\frac{d}{2}} \left\|\left(1 - \Pi^N_{2\Lv}\right) Q(\gchi, \gchi) \right\|_{L^2(\OmegaLv)} \nonumber \\
	= &~ \frac{2d^{-1}C_d}{\sqrt{2k + 1}} \mathcal{O} \left({\Lv}^{k + \frac{d}{2}} \right) \left\|\left(1 - \Pi^N_{2\Lv}\right) Q(\gchi, \gchi) \right\|_{L^2(\OmegaLv)}. \label{moment_derivative_bound_term3}
	\end{align}
	
	This means, by using the bounds (\ref{moment_derivative_bound_term1}), (\ref{moment_derivative_bound_term2}) and (\ref{moment_derivative_bound_term3}) in the overall bound (\ref{moment_derivative_bound1}) for the derivative of the moments of $g$, and noting that $\lambda \leq 1$ implies $m_{k + \lambda - 1}(g)m_{1}(g) + m_{k - 1}(g)m_{\lambda + 1}(g) \leq C_k \bigl(m_0(g) + m_k(g) \bigr)$ by the same argument that led to (\ref{Zk_bound}),
	\begin{align}
	\frac{\textrm{d}}{\textrm{d} t} \bigl(m_k(g) \bigr) \leq& - {\varepsilon_{\chi}}^2 K_{\lambda, k} m_{0}(g_0) m_{k + \lambda}(g) + \frac{C^2_d}{\sqrt{2k + 1}} m_{0}(g) m_{k + \lambda}(g) \nonumber \\
	&~ + C^1_{d,k} \bigl(m_0(g) + m_k(g) \bigr) \nonumber \\
	&~~ + \frac{C^1_d}{\sqrt{2k + 1}} \mathcal{O} \left({\Lv}^{k + \frac{d}{2}} \right) \bigl| \bigl| \left(\Pi^N_{2\Lv} - 1 \right) Q(\gchi, \gchi)  \bigr| \bigr|_{L^2(\OmegaLv)}, \label{moment_derivative_bound2}
	\end{align}
	where $K_{\lambda, k} := k c_{\lambda} K_k$, $C^1_d := 2^{\frac{d}{2}} C_{d} + 2d^{-1}C_d$ and $C^1_{d,k} := \left(2k C^1_k + \frac{C^2_d}{\sqrt{2k + 1}}\right)C_k$.
	
	Here it should be noted that, when the solution $g$ is bounded under the Gaussian (\ref{Gaussian_bound}), $m_0(g) \leq C r^{\frac{d}{2}} m_0(g_0)$.  So, if it is further assumed that $k \geq k_0$ with $k_0$ chosen as
	\begin{equation}
	k_0 := \frac{1}{2} \left(\left(\frac{2 C^2_d C r^{\frac{d}{2}}}{{\varepsilon_{\chi}}^2 K_{\lambda, k}} \right)^2 - 1 \right) \label{k_0}
	\end{equation}
	then
	\begin{align}
	\frac{C^2_d}{\sqrt{2k + 1}} m_{0}(g) m_{k + \lambda}(g) \leq&~ \frac{{\varepsilon_{\chi}}^2 K_{\lambda, k}}{2 C r^{\frac{d}{2}}} \left(C r^{\frac{d}{2}} m_0(g_0) \right) m_{k + \lambda}(g) \nonumber \\
	=&~ \frac{1}{2} {\varepsilon_{\chi}}^2 K_{\lambda, k} m_{0}(g_0) m_{k + \lambda}(g) \label{highest moment bound}
	\end{align}
	and this term can be combined with the negative contribution to give
	\begin{align}
	\frac{\textrm{d}}{\textrm{d} t} \bigl(m_k(g) \bigr) \leq& - \frac{1}{2} {\varepsilon_{\chi}}^2 K_{\lambda, k} m_{0}(g_0) m_{k + \lambda}(g) + C^1_{d,k} \bigl(m_0(g) + m_k(g) \bigr) \nonumber \\
	&~ + \frac{C^1_d}{\sqrt{2k + 1}} \mathcal{O} \left({\Lv}^{k + \frac{d}{2}} \right) \bigl| \bigl| \left(\Pi^N_{2\Lv} - 1 \right) Q(\gchi, \gchi)  \bigr| \bigr|_{L^2(\OmegaLv)}. \label{moment_derivative_bound3}
	\end{align}
	
	Finally, for the term involving the partial Fourier series of $Q$, as is shown in \cite{BoltzmannConvergence} by multiple uses of Parseval's theorem, for any $s > 0$,
	\begin{equation}
	\bigl| \bigl| \left(\Pi^N_{2\Lv} - 1 \right) Q(\gchi, \gchi)  \bigr| \bigr|_{L^2(\Omega_{2\Lv})} \leq \frac{1}{N^s} \bigl| \bigl| Q(\gchi, \gchi)  \bigr| \bigr|_{\dot{H}^s(\Rd)}. \label{Fourier_projection_tail1}
	\end{equation}
	Then, if it is also true that for any $f \in L^2_{\lambda + 1}(\Rd)$ and some constant $C_Q$ that will depend on the Landau operator $Q$, 
	\begin{equation*}
	\bigl| \bigl| Q(\gchi, \gchi)  \bigr| \bigr|_{\dot{H}^{\frac{d - 1}{2}}(\Rd)} \leq C_Q \bigl| \bigl| \gchi  \bigr| \bigr|^2_{L^2_{\lambda + 1}(\Rd)},
	\end{equation*}
	by choosing $s = \frac{d - 1}{2}$ in identity (\ref{Fourier_projection_tail1}),
	\begin{align}
	\bigl| \bigl| \left(\Pi^N_{2\Lv} - 1 \right) Q(\gchi, \gchi)  \bigr| \bigr|_{L^2(\OmegaLv)} &\leq \bigl| \bigl| \left(\Pi^N_{2\Lv} - 1 \right) Q(\gchi, \gchi)  \bigr| \bigr|_{L^2(\Omega_{2\Lv})} \nonumber \\
	&\leq \frac{C_Q}{N^{\frac{d - 1}{2}}} \bigl| \bigl| \gchi \bigr| \bigr|^2_{L^2_{\lambda + 1}(\Rd)}. \nonumber \\
	&= \frac{C_Q}{N^{\frac{d - 1}{2}}} \bigl| \bigl| \gchi \bigr| \bigr|^2_{L^2_{\lambda + 1}(\OmegaLv)}. \label{Fourier_projection_tail2}
	\end{align}
	Therefore, by using this in (\ref{moment_derivative_bound3}) and defining the constant $C^2_{d,k} := \frac{C^1_d C_Q}{\sqrt{2k + 1}}$,
	\begin{multline*}
	\frac{\textrm{d}}{\textrm{d} t} \bigl(m_k(g) \bigr) \leq - \frac{1}{2} {\varepsilon_{\chi}}^2 K_{\lambda, k} m_{0}(g_0) m_{k + \lambda}(g) + C^1_{d,k} \bigl(m_0(g) + m_k(g) \bigr) \nonumber \\
	+ C^2_{d,k} \frac{\mathcal{O} \left({\Lv}^{k + \frac{d}{2}} \right)}{N^{\frac{d - 1}{2}}} \bigl| \bigl| \gchi \bigr| \bigr|^2_{L^2_{\lambda + 1}(\OmegaLv)},
	\end{multline*}
	which is the required result (\ref{moment_derivative_bound}) in Lemma \ref{moment_derivative_lemma}.
\end{proof}

\subsection{Estimates on the Time Derivative of the $L^2$-norm}
\begin{lemma} \label{L2_norm_derivative_lemma}
	For a solution $g$ of the semi-discrete problem (\ref{semi-discrete2}) which satisfies the stability condition (\ref{stability}) with $\varepsilon \leq \frac{1}{4}$,
	\begin{align}
	&\frac{d}{d t} \left(\left\| \gchi \right\|_{L^2(\OmegaLv)} \right) \nonumber \\
	\leq& - \frac{K^S_{d, \lambda}(g_0)}{{\Lv}^{2}} \left\| \gchi \right\|_{L^2(\OmegaLv)} + K_{d, \lambda}(g_0) {\Lv}^{\frac{d}{2}} + 2 C_{\lambda} \left(2 \Lv \right)^{\frac{d}{2}} \left(1 + {\Lv}^2 \right) \|g_0\|_{L^1_{2}(\OmegaLv)} \nonumber \\
	&~ + C_d O_{\frac{d}{2} + 2 - \lambda}\|g_0\|^2_{L^1_2(\OmegaLv)} + \frac{C^3_d}{N^{\frac{d - 1}{2}}} \bigl|\bigl| \gchi \bigr|\bigr|^{2}_{L^2_{\lambda + 1}(\OmegaLv)}, \label{L2_norm_derivative_bound}
	\end{align}
	where the moment operator $m_k$ and $L^2_k$-norm are defined by (\ref{mk_def}) and (\ref{L2k_def}), respectively; $O_r$ denotes a constant that is $\mathcal{O}({\Lv}^{-r})$; $g_0$ is the initial condition and the constants $K_{d, \lambda}(g_0)$, $K^S_{d, \lambda}(g_0)$, $C_{\lambda}$, $C_{d}$, $C^3_{d} > 0$ with $C_{d}$ and $C^3_{d}$ depending on the dimension $d$; $C_{\lambda}$ depending on the potential $\lambda$; and $K_{d, \lambda}(g_0)$ and $K^S_{d, \lambda}(g_0)$ depending on $d$ and $\lambda$, as well as $g_0$.
\end{lemma}

\begin{proof}
	First, multiplying the expansion (\ref{Q_expansion}) by \mbox{\Large$\chi$}$g(\boldsymbol{v})$ and integrating with respect to $\boldsymbol{v}$ over $\OmegaLv$ gives
	\begin{align}
	&\int_{\OmegaLv} \frac{\partial g}{\partial t} \gchi ~\dv \nonumber \\
	=& \int_{\OmegaLv} Q(\gchi, \gchi) \gchi ~\dv \nonumber \\
	&+ \int_{\OmegaLv} \Bigl((Q_c(g,g) - Q_u(g,g)) - \left(1 - \Pi^N_{2\Lv}\right) Q(\gchi, \gchi) \Bigr) \gchi ~\dv. \label{L2_derivative_bound1}
	\end{align}
	
	
	Here it should be noted that, since the extension operator \mbox{\Large$\chi$} is independent of $t$, by two applications of the chain rule,
	\begin{align*}
	\int_{\OmegaLv} \frac{\partial g}{\partial t} \gchi ~\dv = \frac{1}{2} \int_{\OmegaLv}  \frac{\partial}{\partial t} \left(\gchi^2 \right) ~\dv \approx& \frac{1}{2} \frac{d}{d t} \left(\left\| \gchi \right\|^{2}_{L^2(\OmegaLv)} \right) \\
	=& \left\| \gchi \right\|_{L^2(\OmegaLv)} \frac{d}{d t} \left(\left\| \gchi \right\|_{L^2(\OmegaLv)} \right),
	\end{align*}
	where the $\approx$ is due to the requirement of an extra multiplication by \mbox{\Large$\chi$} at that stage and that $\mbox{\Large$\chi$}^2 \neq$ \mbox{\Large$\chi$} only on a very small subset of $\OmegaLv$.  As a result, (\ref{L2_derivative_bound1}) can be written as
	\begin{align}
	&\left\| \gchi \right\|_{L^2(\OmegaLv)} \frac{d}{d t} \left(\left\| \gchi \right\|_{L^2(\OmegaLv)} \right) \nonumber \\
	\leq& \int_{\OmegaLv} Q(\gchi, \gchi) \gchi(\boldsymbol{v}) ~\dv \nonumber \\
	&+ \int_{\OmegaLv} \Bigl((Q_c(g,g) - Q_u(g,g)) - \left(1 - \Pi^N_{2\Lv}\right) Q(\gchi, \gchi) \Bigr) \gchi(\boldsymbol{v}) ~\dv. \label{L2_derivative_bound2}
	\end{align}
	
	Now, by the result of Proposition \ref{Halpha_estimate_prop}(b)$(i)$,
	\begin{multline}
	\int_{\OmegaLv} Q(\gchi, \gchi) \gchi(\boldsymbol{v}) \langle \boldsymbol{v} \rangle^{2 k} ~\dv \leq -K_{\lambda}(g_0) \|\gchi\|^2_{\dot{H}^1_{k + \frac{\lambda}{2}}(\OmegaLv)} \\
	+ C_{\lambda} \|\gchi\|^{2}_{L^1_{\frac{5}{2}\left(k + \frac{\lambda}{2}\right)}(\OmegaLv)}. \label{D&V_L2_bound}
	\end{multline}
	It should be noted that, by definition, the $\dot{H}^s$-norm of a smooth function $h$ with $s = 1$ is simply
	\begin{align*}
	\|h\|^2_{\dot{H}^1(\OmegaLv)} = \int_{\Omega_{L_{\boldsymbol{\xi}}}} |\boldsymbol{\xi}|^2 \left(\widehat{h} (\boldsymbol{\xi}) \right)^2 ~\textrm{d}\boldsymbol{\xi} &= \int_{\Omega_{L_{\boldsymbol{\xi}}}} \left|\widehat{\nabla h} (\boldsymbol{\xi}) \right|^2 ~\textrm{d}\boldsymbol{\xi} \\
	&= \int_{\OmegaLv} \left|\nabla h (\boldsymbol{v}) \right|^2 ~\dv = \|\left|\nabla h \right\||^2_{L^2(\OmegaLv)}
	\end{align*}
	and, by the same argument, the weighted $\dot{H}^1_{k + \frac{\lambda}{2}}$-norm is
	\begin{equation*}
	\|h\|^2_{\dot{H}^1_{k + \frac{\lambda}{2}}(\OmegaLv)} = \|\left|\nabla h \right\||^2_{L^2_{k + \frac{\lambda}{2}}(\OmegaLv)},
	\end{equation*}
	\begin{flalign}
	\textrm{so } && \|\gchi\|^2_{\dot{H}^1_{k + \frac{\lambda}{2}}(\OmegaLv)} = \int_{\OmegaLv} \left|\nabla \left(\gchi\right) (\boldsymbol{v}) \left(1 + |\boldsymbol{v}|^2 \right)^{\frac{k}{2} + \frac{\lambda}{4}} \right|^2 ~\dv. && \label{H1k_1}
	\end{flalign}
	
	Here, by the product rule,
	\begin{multline*}
	\nabla \left(\gchi\right) (\boldsymbol{v}) \left(1 + |\boldsymbol{v}|^2 \right)^{\frac{k}{2} + \frac{\lambda}{4}} = \nabla \left(\gchi (\boldsymbol{v}) \left(1 + |\boldsymbol{v}|^2 \right)^{\frac{k}{2} + \frac{\lambda}{4}} \right) \\
	- \left(k + \frac{\lambda}{2} \right) \boldsymbol{v} \gchi (\boldsymbol{v}) \left(1 + |\boldsymbol{v}|^2 \right)^{\frac{k}{2} + \frac{\lambda}{4} - 1}
	\end{multline*}
	\begin{flalign}
	\textrm{and so} && &\left| \nabla \left(\gchi\right) (\boldsymbol{v}) \left(1 + |\boldsymbol{v}|^2 \right)^{\frac{k}{2} + \frac{\lambda}{4}} \right|^2 && \nonumber \\
	&& =& \left| \nabla \left(\gchi (\boldsymbol{v}) \left(1 + |\boldsymbol{v}|^2 \right)^{\frac{k}{2} + \frac{\lambda}{4}} \right) \right|^2 && \nonumber \\
	&& &- 2 \left(k + \frac{\lambda}{2} \right) \boldsymbol{v} \cdot \nabla \left(\gchi (\boldsymbol{v}) \left(1 + |\boldsymbol{v}|^2 \right)^{\frac{k}{2} + \frac{\lambda}{4}} \right) \left(1 + |\boldsymbol{v}|^2 \right)^{\frac{k}{2} + \frac{\lambda}{4} - 1} && \nonumber \\
	&& &+ \left(k + \frac{\lambda}{2} \right)^2 |\boldsymbol{v}|^2 \left(\gchi (\boldsymbol{v}) \right)^2 \left(1 + |\boldsymbol{v}|^2 \right)^{k + \frac{\lambda}{2} - 2}. && \label{H1k_2}
	\end{flalign}
	This means, by using expression (\ref{H1k_2}) in (\ref{H1k_1}), the negative term in (\ref{D&V_L2_bound}) can be written as
	\begin{align*}
	&-K_{\lambda}(g_0) \|\gchi\|^2_{\dot{H}^1_{k + \frac{\lambda}{2}}(\OmegaLv)} \\
	=& K_{\lambda}(g_0) \Biggl( - \int_{\OmegaLv} \left| \nabla \left(\gchi (\boldsymbol{v}) \left(1 + |\boldsymbol{v}|^2 \right)^{\frac{k}{2} + \frac{\lambda}{4}} \right) \right|^2 ~\dv \\
	&~~~~~~+ 2\left(k + \frac{\lambda}{2} \right) \int_{\OmegaLv} \boldsymbol{v} \cdot \nabla \left(\gchi (\boldsymbol{v}) \left(1 + |\boldsymbol{v}|^2 \right)^{\frac{k}{2} + \frac{\lambda}{4}} \right) \left(1 + |\boldsymbol{v}|^2 \right)^{\frac{k}{2} + \frac{\lambda}{4} - 1} ~\dv \\
	&~~~~~~- \left(k + \frac{\lambda}{2} \right)^2 \int_{\OmegaLv} |\boldsymbol{v}|^2 \left(\gchi (\boldsymbol{v}) \right)^2 \left(1 + |\boldsymbol{v}|^2 \right)^{k + \frac{\lambda}{2} - 2} ~\dv \Biggr).
	\end{align*}
	
	Here, by the divergence theorem on the second integral and noting that the boundary terms go to zero due to the extension operator \mbox{\Large$\chi$},
	\begin{align*}
	&\int_{\OmegaLv} \boldsymbol{v} \cdot \nabla \left(\gchi (\boldsymbol{v}) \left(1 + |\boldsymbol{v}|^2 \right)^{\frac{k}{2} + \frac{\lambda}{4}} \right) \left(1 + |\boldsymbol{v}|^2 \right)^{\frac{k}{2} + \frac{\lambda}{4} - 1} ~\dv \\
	=& - \int_{\OmegaLv} \nabla \cdot \left(\left(1 + |\boldsymbol{v}|^2 \right)^{\frac{k}{2} + \frac{\lambda}{4} - 1} \boldsymbol{v} \right) \gchi (\boldsymbol{v}) \left(1 + |\boldsymbol{v}|^2 \right)^{\frac{k}{2} + \frac{\lambda}{4}} ~\dv.
	\end{align*}
	Then, by noting that $\nabla \cdot \left(\left(1 + |\boldsymbol{v}|^2 \right)^{\frac{k}{2} + \frac{\lambda}{4} - 1} \boldsymbol{v} \right) = \left(k + \frac{\lambda}{2} - 2 \right)\left(1 + |\boldsymbol{v}|^2 \right)^{\frac{k}{2} + \frac{\lambda}{4} - 2} |\boldsymbol{v} |^2$ $+ d \left(1 + |\boldsymbol{v}|^2 \right)^{\frac{k}{2} + \frac{\lambda}{4} - 1}$,
	\begin{align*}
	&\int_{\OmegaLv} \boldsymbol{v} \cdot \nabla \left(\gchi (\boldsymbol{v}) \left(1 + |\boldsymbol{v}|^2 \right)^{\frac{k}{2} + \frac{\lambda}{4}} \right) \left(1 + |\boldsymbol{v}|^2 \right)^{\frac{k}{2} + \frac{\lambda}{4} - 1} ~\dv \\
	=& \left(2 - \left(k + \frac{\lambda}{2} \right)\right) \int_{\OmegaLv} |\boldsymbol{v}|^2 \gchi (\boldsymbol{v}) \left(1 + |\boldsymbol{v}|^2 \right)^{k + \frac{\lambda}{2} - 2} ~\dv \\
	& - d \int_{\OmegaLv} \gchi (\boldsymbol{v}) \left(1 + |\boldsymbol{v}|^2 \right)^{k + \frac{\lambda}{2} - 1} ~\dv
	\end{align*}
	and this means
	\begin{align*}
	-K_{\lambda}(g_0) \|\gchi\|^2_{\dot{H}^1_{k + \frac{\lambda}{2}}(\OmegaLv)} =& K_{\lambda}(g_0) \Biggl( - \int_{\OmegaLv} \left| \nabla \left(\gchi (\boldsymbol{v}) \left(1 + |\boldsymbol{v}|^2 \right)^{\frac{k}{2} + \frac{\lambda}{4}} \right) \right|^2 ~\dv \\
	&~~+ 4 \left(k + \frac{\lambda}{2} \right) \int_{\OmegaLv} |\boldsymbol{v}|^2 \gchi (\boldsymbol{v}) \left(1 + |\boldsymbol{v}|^2 \right)^{k + \frac{\lambda}{2} - 2} ~\dv \\
	&~~- 2 \left(k + \frac{\lambda}{2} \right)^2 \int_{\OmegaLv} |\boldsymbol{v}|^2 \gchi (\boldsymbol{v}) \left(1 + |\boldsymbol{v}|^2 \right)^{k + \frac{\lambda}{2} - 2} ~\dv \\
	&~~ - 2 d \left(k + \frac{\lambda}{2} \right) \int_{\OmegaLv} \gchi (\boldsymbol{v}) \left(1 + |\boldsymbol{v}|^2 \right)^{k + \frac{\lambda}{2} - 1} ~\dv \\
	&- \left(k + \frac{\lambda}{2} \right)^2 \int_{\OmegaLv} |\boldsymbol{v}|^2 \left(\gchi (\boldsymbol{v}) \right)^2 \left(1 + |\boldsymbol{v}|^2 \right)^{k + \frac{\lambda}{2} - 2} ~\dv \Biggr).
	\end{align*}
	
	Now, setting $k = 0$ like in (\ref{L2_derivative_bound2}), as well as dropping the last term which is truly negative, gives
	\begin{align*}
	-K_{\lambda}(g_0) \|\gchi\|^2_{\dot{H}^1_{\frac{\lambda}{2}}(\OmegaLv)} =& K_{\lambda}(g_0) \Biggl( - \int_{\OmegaLv} \left| \nabla \left(\gchi (\boldsymbol{v}) \left(1 + |\boldsymbol{v}|^2 \right)^{\frac{\lambda}{4}} \right) \right|^2 ~\dv \\
	&~~~~~~+ 4 \left(\frac{\lambda}{2} \right) \int_{\OmegaLv} |\boldsymbol{v}|^2 \gchi (\boldsymbol{v}) \left(1 + |\boldsymbol{v}|^2 \right)^{\frac{\lambda}{2} - 2} ~\dv \\
	&~~~~~~ - 2 \left(\frac{\lambda}{2} \right)^2 \int_{\OmegaLv} |\boldsymbol{v}|^2 \gchi (\boldsymbol{v}) \left(1 + |\boldsymbol{v}|^2 \right)^{\frac{\lambda}{2} - 2} ~\dv \\
	&~~~~~~ - 2 d \left(\frac{\lambda}{2} \right) \int_{\OmegaLv} \gchi (\boldsymbol{v}) \left(1 + |\boldsymbol{v}|^2 \right)^{\frac{\lambda}{2} - 1} ~\dv \Biggr).
	\end{align*}
	Then, by coarsely bounding by the positive version of the negative integrals and noting that $0 \leq \lambda \leq 1$ so that $\left(\frac{\lambda}{2} \right)^2 \leq \frac{\lambda}{2}$, this means
	\begin{align}
	-K_{\lambda}(g_0) \|\gchi\|^2_{\dot{H}^1_{\frac{\lambda}{2}}(\OmegaLv)} =& K_{\lambda}(g_0) \Biggl( - \int_{\OmegaLv} \left| \nabla \left(\gchi (\boldsymbol{v}) \left(1 + |\boldsymbol{v}|^2 \right)^{\frac{\lambda}{4}} \right) \right|^2 ~\dv \nonumber \\
	&~~~~~~+ 4 \frac{\lambda}{2} \int_{\OmegaLv} |\boldsymbol{v}|^2 \left| \gchi (\boldsymbol{v}) \right| \left(1 + |\boldsymbol{v}|^2 \right)^{\frac{\lambda}{2} - 2} ~\dv \nonumber \\
	&~~~~~~ + 2 \frac{\lambda}{2} \int_{\OmegaLv} |\boldsymbol{v}|^2 \left| \gchi (\boldsymbol{v}) \right| \left(1 + |\boldsymbol{v}|^2 \right)^{\frac{\lambda}{2} - 2} ~\dv \nonumber \\
	&~~~~~~ + 2 d \frac{\lambda}{2} \int_{\OmegaLv} \left| \gchi (\boldsymbol{v}) \right| \left(1 + |\boldsymbol{v}|^2 \right)^{\frac{\lambda}{2} - 1} ~\dv \Biggr) \nonumber \\
	\leq& K_{\lambda}(g_0) \Biggl( - \int_{\OmegaLv} \left| \nabla \left(\gchi (\boldsymbol{v}) \left(1 + |\boldsymbol{v}|^2 \right)^{\frac{\lambda}{4}} \right) \right|^2 ~\dv \nonumber \\
	&~~~~~~+ (3 + d) \lambda \int_{\OmegaLv} \left| \gchi (\boldsymbol{v}) \right| \left(1 + |\boldsymbol{v}|^2 \right)^{\frac{\lambda}{2} - 1} ~\dv \Biggr). \label{D&V_L2_bound_neg_term}
	\end{align}
	Here, by the Cauchy-Schwarz inequality and noting that $\left\| 1 \right\|_{L^2(\OmegaLv)} = (2 \Lv)^{\frac{d}{2}}$,
	\begin{equation*}
	\int_{\OmegaLv} \left| \gchi (\boldsymbol{v}) \right| \left(1 + |\boldsymbol{v}|^2 \right)^{\frac{\lambda}{2} - 1} ~\dv \leq (2 \Lv)^{\frac{d}{2}} \left\| \gchi \right\|_{L^2_{\frac{\lambda}{2} - 1}(\OmegaLv)} \leq (2 \Lv)^{\frac{d}{2}} \left\| \gchi \right\|_{L^2(\OmegaLv)},
	\end{equation*}
	since $\lambda \leq 1$ implies that $\langle \boldsymbol{v} \rangle^{\frac{\lambda}{2} - 1} \leq 1$ and so the weight can be removed from the $L^2$-norm.  
	
	Also, by using $h = \gchi (\boldsymbol{v}) \left(1 + |\boldsymbol{v}|^2 \right)^{\frac{\lambda}{4}}$ and $p = 2$ in the Sobolev inequality from Lemma \ref{Sobolev_result} in the appendix, for $C^S_d := C^S_{2, d}$,
	\begin{align*}
	\int_{\OmegaLv} \left| \nabla \left(\gchi (\boldsymbol{v}) \left(1 + |\boldsymbol{v}|^2 \right)^{\frac{\lambda}{4}} \right) \right|^2 ~\dv &= \left\| \gchi (\boldsymbol{v}) \left(1 + |\boldsymbol{v}|^2 \right)^{\frac{\lambda}{4}} \right\|^{2}_{\dot{H}^1(\OmegaLv)} \\
	&\geq \frac{1}{\left(C^S_d \right)^2} \left\| \gchi (\boldsymbol{v}) \left(1 + |\boldsymbol{v}|^2 \right)^{\frac{\lambda}{4}} \right\|^{2}_{L^{\frac{2d}{d - 2}}(\OmegaLv)} \\
	&\geq \frac{1}{\left(C^S_d \right)^2} \left\| \gchi \right\|^{2}_{L^{\frac{2d}{d - 2}}(\OmegaLv)} \\
	&\geq \frac{1}{\left(2 C^S_d \Lv\right)^2} \left\| \gchi \right\|^{2}_{L^{2}(\OmegaLv)},
	\end{align*}
	by considering the identity on the nesting of $L^p$ norms over bounded spaces (or a calculation involving an application of H\"{o}lder's inequality).  
	
	So, using these last two inequalities in (\ref{D&V_L2_bound_neg_term}) means that identity (\ref{D&V_L2_bound}) with $k = 0$ leads to
	\begin{align}
	\int_{\OmegaLv} Q(\gchi, \gchi) \gchi(\boldsymbol{v}) ~\dv \leq& - \frac{K^S_{d, \lambda}(g_0)}{{\Lv}^{2}} \left\| \gchi \right\|^{2}_{L^2(\OmegaLv)} \nonumber \\
	&+ K_{d, \lambda}(g_0) {\Lv}^{\frac{d}{2}} \left\| \gchi \right\|_{L^2(\OmegaLv)} + C_{\lambda} \|\gchi\|^{2}_{L^1_{\frac{5 \lambda}{4}}(\OmegaLv)}, \label{D&V_L2_bound2}
	\end{align}
	\begin{flalign}
	\textrm{for } && K^S_{d, \lambda}(g_0) := \frac{K_{\lambda}(g_0)}{\left(2 C^S_d \right)^2} ~~~\textrm{and}~~~ K_{d, \lambda}(g_0) := 2^{\frac{d}{2}} K_{\lambda}(g_0) (3 + d) \lambda. && \label{K_dlambda_def}
	\end{flalign}
	
	Finally for this term, note that by using the Cauchy-Schwarz inequality on one of the $\|\gchi\|_{L^1_{\frac{5 \lambda}{4}}(\OmegaLv)}$ terms and then pulling out the maximum value of the weight,
	\begin{align*}
	\|\gchi\|^{2}_{L^1_{\frac{5 \lambda}{4}}(\OmegaLv)} &\leq \left(2 \Lv \right)^{\frac{d}{2}} \|\gchi\|_{L^1_{\frac{5 \lambda}{4}}(\OmegaLv)} \|\gchi\|_{L^2_{\frac{5 \lambda}{4}}(\OmegaLv)} \\
	&\leq \left(2 \Lv \right)^{\frac{d}{2}} \left(1 + {\Lv}^2 \right) \|\gchi\|_{L^1_{2}(\OmegaLv)} \|\gchi\|_{L^2(\OmegaLv)} \\
	&\leq 2 \left(2 \Lv \right)^{\frac{d}{2}} \left(1 + {\Lv}^2 \right) \|g_0\|_{L^1_{2}(\OmegaLv)} \|\gchi\|_{L^2(\OmegaLv)},
	\end{align*}
	after applying the bound (\ref{stability_consequence}) resulting from the stability condition (\ref{stability}).  This means that (\ref{D&V_L2_bound2}) becomes
	\begin{align}
	\int_{\OmegaLv} Q(\gchi, \gchi) \gchi(\boldsymbol{v}) ~\dv &\leq \biggl(- \frac{K^S_{d, \lambda}(g_0)}{{\Lv}^{2}} \left\| \gchi \right\|_{L^2(\OmegaLv)} + K_{d, \lambda}(g_0) {\Lv}^{\frac{d}{2}} \nonumber \\
	&~~~ + 2 C_{\lambda} \left(2 \Lv \right)^{\frac{d}{2}} \left(1 + {\Lv}^2 \right) \|g_0\|_{L^1_{2}(\OmegaLv)} \biggr) \|\gchi\|_{L^2(\OmegaLv)}. \label{D&V_L2_bound3}
	\end{align}
	
	Also, by using the Cauchy-Schwarz inequality on the second integral in (\ref{L2_derivative_bound2}) and the triangle inequality,
	\begin{align*}
	&\int_{\OmegaLv} \Bigl((Q_c(g,g) - Q_u(g,g)) - \left(1 - \Pi^N_{2\Lv}\right) Q(\gchi, \gchi) \Bigr) \gchi(\boldsymbol{v}) ~\dv \\
	\leq& \int_{\OmegaLv} \left| \Bigl((Q_c(g,g) - Q_u(g,g)) - \left(1 - \Pi^N_{2\Lv}\right) Q(\gchi, \gchi) \Bigr) \gchi(\boldsymbol{v}) \right| ~\dv \\
	\leq& \left\|\Bigl((Q_c(g,g) - Q_u(g,g)) - \left(1 - \Pi^N_{2\Lv}\right) Q(\gchi, \gchi) \Bigr) \right\|_{L^2(\OmegaLv)} \|\gchi\|_{L^2(\OmegaLv)} \\
	\leq& \Bigl(\left\|(Q_c(g,g) - Q_u(g,g)) \right\|_{L^2(\OmegaLv)} \\
	&~~~~~~~~~~~~~~~~~~~~~~~~~~~~~~ + \left\|\left(1 - \Pi^N_{2\Lv}\right) Q(\gchi, \gchi) \right\|_{L^2(\OmegaLv)} \Bigr) \|\gchi\|_{L^2(\OmegaLv)}.	
	\end{align*}
	Then, by using the result from Lemma \ref{Q_moment_theorem} on the first term with $k = 0$ and $k' = 2$, and leaving in the extension operators in the moment terms before they were bounded by 1 in the last line of that proof,
	\begin{align*}
	&\int_{\OmegaLv} \Bigl((Q_c(g,g) - Q_u(g,g)) - \left(1 - \Pi^N_{2\Lv}\right) Q(\gchi, \gchi) \Bigr) \gchi(\boldsymbol{v}) ~\dv \\
	\leq& \Biggl(C_d \biggl(\bigl| \bigl| \left(\Pi^N_{2\Lv} - 1 \right) Q(\gchi, \gchi) \bigr| \bigr|_{L^2(\OmegaLv)} \\
	&~~~~~~~~~ + O_{\frac{d}{2} + 2} \Bigl(m_{0}(\gchi) m_{2 + \lambda}(\gchi) + m_{\lambda}(\gchi) m_{2}(\gchi) \Bigr)\biggr) \\
	&~~~~~~~~~~~~~~~~~~~~~~~~~~~~~ + \left\|\left(1 - \Pi^N_{2\Lv}\right) Q(\gchi, \gchi) \right\|_{L^2(\OmegaLv)} \Biggr) \|\gchi\|_{L^2(\OmegaLv)}.
	\end{align*}
	
	So, by collecting the $\left\|\left(1 - \Pi^N_{2\Lv}\right) Q(\gchi, \gchi) \right\|_{L^2(\OmegaLv)}$ terms and then using assumption (\ref{Fourier_projection_tail2}) to bound them,
	\begin{align}
	&\int_{\OmegaLv} \Bigl((Q_c(g,g) - Q_u(g,g)) - \left(1 - \Pi^N_{2\Lv}\right) Q(\gchi, \gchi) \Bigr) \gchi(\boldsymbol{v}) ~\dv \nonumber \\
	\leq& \Biggl(C_d O_{\frac{d}{2} + 2} \Bigl(m_{0}(\gchi) m_{2 + \lambda}(\gchi) + m_{\lambda}(\gchi) m_{2}(\gchi) \Bigr) \nonumber \\
	&~~~~~~~~~~~~~~~~~~~~~~~~~~~ + \frac{C_Q(C_d + 1)}{N^{\frac{d - 1}{2}}} \bigl| \bigl| \gchi \bigr| \bigr|^2_{L^2_{\lambda + 1}(\OmegaLv)} \Biggr) \|\gchi\|_{L^2(\OmegaLv)}. \label{L2_derivative_estimate_remainder}
	\end{align}
	Here, since $\lambda \leq 1$ and \mbox{\Large$\chi$}$(\boldsymbol{v}) \leq 1$ for $\boldsymbol{v} \in \OmegaLv$ and 0 otherwise, 
	\begin{equation*}
	m_{\lambda}(\gchi) \leq \|g\|_{L^1_2(\OmegaLv)} \leq 2\|g_0\|_{L^1_2(\OmegaLv)},
	\end{equation*}
	by using the bound (\ref{stability_consequence}) resulting from the stability condition (\ref{stability}).  The upper bound is also true for $m_0(\gchi)$ which means, by applying the same estimates for the extension operator \mbox{\Large$\chi$} in the other moment terms,
	\begin{align*}
	& m_{0}(\gchi) m_{2 + \lambda}(\gchi) + m_{\lambda}(\gchi) m_{2}(\gchi) \\
	\leq& 2\|g_0\|_{L^1_2(\OmegaLv)} \left(\|g\|_{L^1_{2 + \lambda}(\OmegaLv)} + \|g\|_{L^1_2(\OmegaLv)}\right) \\
	\leq& 2\|g_0\|_{L^1_2(\OmegaLv)} \left(\left(1 + {\Lv}^{2}\right)^{\frac{\lambda}{2}}\|g\|_{L^1_{2}(\OmegaLv)} + \|g\|_{L^1_2(\OmegaLv)}\right) \\
	\leq& 4\left(1 + \left(1 + {\Lv}^{2}\right)^{\frac{\lambda}{2}}\right)\|g_0\|^2_{L^1_2(\OmegaLv)},
	\end{align*}
	after pulling out the maximum value of an order $\lambda$ portion of the weight of the $L^2_{2 + \lambda}$-norm and then using the stability result (\ref{stability_consequence}) again.  So, by noting that the coefficient of $\|g_0\|^2_{L^1_2(\OmegaLv)}$ is an $O_{-\lambda}$ term, (\ref{L2_derivative_estimate_remainder}) becomes
	\begin{align*}
	&\int_{\OmegaLv} \Bigl((Q_c(g,g) - Q_u(g,g)) - \left(1 - \Pi^N_{2\Lv}\right) Q(\gchi, \gchi) \Bigr) \gchi(\boldsymbol{v}) ~\dv \\
	\leq& \Biggl(C_d O_{\frac{d}{2} + 2 - \lambda}\|g_0\|^2_{L^1_2(\OmegaLv)} + \frac{C_Q(C_d + 1)}{N^{\frac{d - 1}{2}}} \bigl| \bigl| \gchi \bigr| \bigr|^2_{L^2_{\lambda + 1}(\OmegaLv)} \Biggr) \|\gchi\|_{L^2(\OmegaLv)}.	
	\end{align*}
	
	Therefore, by labeling $C^3_d := C_Q(C_d + 1)$ and then using this bound in (\ref{L2_derivative_bound2}) along with (\ref{D&V_L2_bound3}),
	\begin{align*}
	&\left\| \gchi \right\|_{L^2(\OmegaLv)} \frac{d}{d t} \left(\left\| \gchi \right\|_{L^2(\OmegaLv)} \right) \nonumber \\
	\leq& \biggl(- \frac{K^S_{d, \lambda}(g_0)}{{\Lv}^{2}} \left\| \gchi \right\|_{L^2(\OmegaLv)} + K_{d, \lambda}(g_0) {\Lv}^{\frac{d}{2}} + 2 C_{\lambda} \left(2 \Lv \right)^{\frac{d}{2}} \left(1 + {\Lv}^2 \right) \|g_0\|_{L^1_{2}(\OmegaLv)} \\
	&~~~~~~~ + C_d O_{\frac{d}{2} + 2 - \lambda}\|g_0\|^2_{L^1_2(\OmegaLv)} + \frac{C^3_d}{N^{\frac{d - 1}{2}}} \bigl|\bigl| \gchi \bigr|\bigr|^{2}_{L^2_{\lambda + 1}(\OmegaLv)} \Biggr) \|\gchi\|_{L^2(\OmegaLv)},
	\end{align*}
	and so dividing both sides by $\|\gchi\|_{L^2(\OmegaLv)}$ gives
	\begin{align*}
	&\frac{d}{d t} \left(\left\| \gchi \right\|_{L^2(\OmegaLv)} \right) \nonumber \\
	\leq& - \frac{K^S_{d, \lambda}(g_0)}{{\Lv}^{2}} \left\| \gchi \right\|_{L^2(\OmegaLv)} + K_{d, \lambda}(g_0) {\Lv}^{\frac{d}{2}} + 2 C_{\lambda} \left(2 \Lv \right)^{\frac{d}{2}} \left(1 + {\Lv}^2 \right) \|g_0\|_{L^1_{2}(\OmegaLv)} \\
	&~ + C_d O_{\frac{d}{2} + 2 - \lambda}\|g_0\|^2_{L^1_2(\OmegaLv)} + \frac{C^3_d}{N^{\frac{d - 1}{2}}} \bigl|\bigl| \gchi \bigr|\bigr|^{2}_{L^2_{\lambda + 1}(\OmegaLv)},
	\end{align*}
	which is the required result (\ref{L2_norm_derivative_bound}) in Lemma \ref{L2_norm_derivative_lemma}.
\end{proof}

\subsection{Control of the Negative Part of the Solution}
All the work so far has hinged on the stability assumption (\ref{stability}) on the solution $g(t)$ to the semi-discrete problem (\ref{semi-discrete2}) for all $t > 0$.  This can only be assumed on the initial condition $g_0$ (i.e. assumption (\ref{stability_init})) but it can be shown that this condition does indeed propagate throughout all time.  
\begin{rem}
	The actual proof that the negative part of the solution is bounded is written later in Section \ref{ExistenceAndRegularity} on existence and regularity but it begins with finding an estimate on the time derivative of the $L^2$-norm of the negative part of the solution, which is included here.
\end{rem}

\begin{lemma} \label{g-_L2_derivative_lemma}
	For a solution $g$ of the semi-discrete problem (\ref{semi-discrete2}), the negative part  $g^-$ (where $g = g^+ - g^-$, for $g^+, g^- \geq 0$) satisfies
	\begin{align}
	&\frac{d}{d t} \left(\left\| \gchi^- \right\|_{L^2(\OmegaLv)} \right) \nonumber \\
	\leq&~ C_{d,\lambda} \left\| g_0 \right\|_{L^1_2(\OmegaLv)} \left\| \gchi^- \right\|_{L^2(\OmegaLv)} \nonumber \\
	& + C_d O_{\frac{d}{2} + 2 - \lambda} m_{0}(g_0) m_{2}(g) + C_Q(C_d + 1) \frac{(1 + {\Lv}^2)^2}{N^{\frac{d - 1}{2}}} \bigl| \bigl| \gchi \bigr| \bigr|^2_{L^2(\OmegaLv)}, \label{g-_L2_derivative_bound}
	\end{align}
	where the moment operator $m_k$ and $L^2_k$-norm are defined by (\ref{mk_def}) and (\ref{L2k_def}), respectively; $O_r$ denotes a constant that is $\mathcal{O}({\Lv}^{-r})$; $g_0$ is the initial condition and $C_Q, C_d, C_{d,\lambda} > 0$ are constants with $C_{d, \lambda}$ depending on $d$ and the potential $\lambda$; $C_d$ defined in Lemma \ref{Q_moment_theorem}; and $C_Q$ defined in the proof of Lemma \ref{moment_derivative_lemma}.
\end{lemma}

\begin{proof}
	Similar to the proof of Lemma \ref{L2_norm_derivative_lemma} on the estimate for the time-derivative of the $L^2$-norm, begin by multiplying the expansion (\ref{Q_expansion}) by \linebreak \mbox{\Large$\chi$}$g \mathbbm{1}_{\{g<0 \}} = - g^-$ and integrating with respect to $\boldsymbol{v}$ over $\OmegaLv$ to give
	\begin{align}
	&\int_{\OmegaLv} \frac{\partial g}{\partial t} \gchi \mathbbm{1}_{\{g<0 \}} ~\dv \nonumber \\
	=& \int_{\OmegaLv} Q(\gchi, \gchi) \gchi \mathbbm{1}_{\{g<0 \}} ~\dv \nonumber \\
	&+ \int_{\OmegaLv} \Bigl((Q_c(g,g) - Q_u(g,g)) - \left(1 - \Pi^N_{2\Lv}\right) Q(\gchi, \gchi) \Bigr) \gchi \mathbbm{1}_{\{g<0 \}} ~\dv. \label{g-_derivative_bound1}
	\end{align}
	Here, by the same reasoning as in the proof of Lemma \ref{L2_norm_derivative_lemma},
	\begin{align}
	\int_{\OmegaLv} \frac{\partial g}{\partial t} \gchi \mathbbm{1}_{\{g<0 \}} ~\dv \approx& \left\| \gchi \mathbbm{1}_{\{g<0 \}} \right\|_{L^2(\OmegaLv)} \frac{d}{d t} \left(\left\| \gchi \mathbbm{1}_{\{g<0 \}} \right\|_{L^2(\OmegaLv)} \right) \nonumber \\
	=& \left\| \gchi^- \right\|_{L^2(\OmegaLv)} \frac{d}{d t} \left(\left\| \gchi^- \right\|_{L^2(\OmegaLv)} \right). \label{g-_derivative_approx}
	\end{align}
	
	Next, replacing \mbox{\Large$\chi$}$g$ by $\gchi \mathbbm{1}_{\{g<0 \}} = -\gchi^-$ in identity (\ref{L2_derivative_estimate_remainder}),
	\begin{align}
	&\int_{\OmegaLv} \Bigl((Q_c(g,g) - Q_u(g,g)) - \left(1 - \Pi^N_{2\Lv}\right) Q(\gchi, \gchi) \Bigr) \gchi \mathbbm{1}_{\{g<0 \}} ~\dv \nonumber \\
	\leq& \Biggl(C_d O_{\frac{d}{2} + 2} \Bigl(m_{0}(\gchi) m_{2 + \lambda}(\gchi) + m_{\lambda}(\gchi) m_{2}(\gchi) \Bigr) \nonumber \\
	&~~~~~~~~~~~~~~~~~~~~~~~~~~~ + \frac{C_Q(C_d + 1)}{N^{\frac{d - 1}{2}}} \bigl| \bigl| \gchi \bigr| \bigr|^2_{L^2_{\lambda + 1}(\OmegaLv)} \Biggr) \|\gchi^-\|_{L^2(\OmegaLv)}. \label{g-_derivative_estimate_remainder1}
	\end{align}
	Then, by pulling out the maximum value of the weight $\langle \Lv \rangle^{\lambda}$ from each of the moment terms (after noting that \mbox{\Large$\chi$}$g(\boldsymbol{v}) = 0$ when $\boldsymbol{v} \notin \OmegaLv$) and bounding the extension operator \mbox{\Large$\chi$} by 1,
	\begin{equation*}
	m_{0}(\gchi) m_{2 + \lambda}(\gchi) + m_{\lambda}(\gchi) m_{2}(\gchi) \leq O_{-\lambda} m_{0}(g) m_{2}(g) \leq O_{-\lambda} m_{0}(g_0) m_{2}(g),
	\end{equation*}
	since $m_{0}(g) \leq Cr^{-\frac{d}{2}} m_{0}(g_0)$ from assumption (\ref{Gaussian_bound}).  Also,
	\begin{equation*}
	\bigl| \bigl| \gchi \bigr| \bigr|^2_{L^2_{\lambda + 1}(\OmegaLv)} \leq \bigl| \bigl| \gchi \bigr| \bigr|^2_{L^2_2(\OmegaLv)} \leq (1 + {\Lv}^2)^2 \bigl| \bigl| \gchi \bigr| \bigr|^2_{L^2(\OmegaLv)}
	\end{equation*}
	which means that (\ref{g-_derivative_estimate_remainder1}) becomes
	\begin{align}
	&\int_{\OmegaLv} \Bigl((Q_c(g,g) - Q_u(g,g)) - \left(1 - \Pi^N_{2\Lv}\right) Q(\gchi, \gchi) \Bigr) \gchi \mathbbm{1}_{\{g<0 \}} ~\dv \nonumber \\
	\leq& \Biggl(C_d O_{\frac{d}{2} + 2 - \lambda} m_{0}(g_0) m_{2}(g) \nonumber \\
	&~~~~~~~~~~~~~ + \frac{C_Q(C_d + 1)(1 + {\Lv}^2)^2}{N^{\frac{d - 1}{2}}} \bigl| \bigl| \gchi \bigr| \bigr|^2_{L^2_{\lambda + 1}(\OmegaLv)} \Biggr) \|\gchi^-\|_{L^2(\OmegaLv)}. \label{g-_derivative_estimate_remainder2}
	\end{align}
	
	
	Now, for the first integral on the right-hand side of identity (\ref{g-_derivative_bound1}), a little more care must be taken with the characteristic functions $\mathbbm{1}_{\{g<0 \}}$.  This is done by introducing a second order smoothing approximation to $\mathbbm{1}_{\{g<0 \}}$, similar to the function $H_{\delta}$ defined by (\ref{H_delta}) in the proof of Lemma \ref{moment_derivative_lemma}.  In particular, for any $\delta_c> 0$, define the monotone decreasing function $H^-_{\delta}: \mathbb{R} \to [0, 1]$ by
	\begin{equation*}
	H^-_{\delta}(y) := \begin{cases}
	1, &\textrm{when }~ y \leq -2\delta_c, \\
	-\frac{1}{2 {\delta_c}^2} y^2 - \frac{2}{\delta_c} y - 1, &\textrm{when }~ -2\delta_c < y \leq -\delta_c, \\
	\frac{1}{2 {\delta_c}^2} y^2, &\textrm{when }~ -\delta_c < y \leq 0, \\
	0, &\textrm{when }~ y > 0.
	\end{cases}
	\end{equation*}
	\begin{rem}
		This approximation to the jump is made when $y \leq 0$ to allow $H^-_{\delta}(y) = 0$ when $y > 0$.  This ensures that $(H^-_{\delta})'(y) = 0$ when $y > 0$ which will be important later.
	\end{rem}
	
	This means, after using the weak form, inserting the approximation $\mathbbm{1}_{\{g<0 \}} = \lim_{\delta_c \to 0} H^-_{\delta}$ and then differentiating with the product rule,
	\begin{align*}
	&\int_{\OmegaLv} Q(\gchi, \gchi) \gchi \mathbbm{1}_{\{g<0 \}} ~\dv \\
	=& -\int_{\OmegaLv} \left(\bar{a} \nabla \left(\gchi \right) - \bar{b} \gchi \right) \cdot \nabla \left(\gchi \mathbbm{1}_{\{g<0 \}} \right) ~\dv \nonumber \\
	=& -\int_{\OmegaLv} \left(\bar{a} \nabla \left(\gchi \right) - \bar{b} \gchi \right) \cdot \nabla \left(\lim_{\delta_c \to 0} \gchi H^-_{\delta} \right) ~\dv \nonumber \\
	=& \lim_{\delta_c \to 0} \Biggl(-\int_{\{g < 0\}} \left(\bar{a} \nabla \left(\gchi \right) \right) \cdot \nabla \left(\gchi \right)H^-_{\delta} ~\dv - \int_{\{g < 0\}} \gchi \bar{a}  \nabla \left(\gchi \right)\cdot \nabla \left(H^-_{\delta} \right) ~\dv \nonumber \\
	&~~~~~~~ + \int_{\OmegaLv} \gchi \bar{b} \cdot \nabla \left(\gchi \right) H^-_{\delta} ~\dv + \int_{\{g < 0\}} \left(\gchi \right)^2 \bar{b} \cdot \nabla \left(H^-_{\delta} \right) ~\dv \Biggr), \nonumber
	\end{align*}
	assuming that Lebesgue's dominated convergence theorem can be used.  Next, note that by the chain rule and then the divergence theorem and product rule, since $\gchi = 0$ on the boundary of $\OmegaLv$,
	\begin{align*}
	\int_{\OmegaLv} \gchi \bar{b} \cdot \nabla \left(\gchi \right) H^-_{\delta} ~\dv =& \frac{1}{2} \int_{\OmegaLv} H^-_{\delta} \bar{b} \cdot \nabla \left(\left(\gchi \right)^2 \right) ~\dv \\
	=& - \frac{1}{2} \int_{\{g < 0\}} \left(\gchi \right)^2 \bar{b} \cdot \nabla \left(H^-_{\delta} \right) ~\dv \\
	&~ - \frac{1}{2} \int_{\{g < 0\}} \left(\gchi \right)^2 \left(\nabla \cdot \bar{b} \right) H^-_{\delta} ~\dv.
	\end{align*}
	So, by using $\bar{c} = \nabla \cdot \bar{b}$,
	\begin{align}
	&\int_{\OmegaLv} Q(\gchi, \gchi) \gchi \mathbbm{1}_{\{g<0 \}} ~\dv \nonumber \\
	=& \lim_{\delta_c \to 0} \Biggl(-\int_{\{g < 0\}} \left(\bar{a} \nabla \left(\gchi \right) \right) \cdot \nabla \left(\gchi \right)H^-_{\delta} ~\dv - \int_{\{g < 0\}} \gchi \bar{a}  \nabla \left(\gchi \right)\cdot \nabla \left(H^-_{\delta} \right) ~\dv \nonumber \\
	&~~~~~~~ - \frac{1}{2} \int_{\{g < 0\}} \left(\gchi \right)^2 \bar{c} H^-_{\delta} ~\dv + \frac{1}{2} \int_{\{g < 0\}} \left(\gchi \right)^2 \bar{b} \cdot \nabla \left(H^-_{\delta} \right) ~\dv \Biggr), \label{g_neg_Qbound}
	\end{align}
	
	Then, by similar calculations to those in the proof of Lemma \ref{moment_derivative_lemma},
	\begin{multline}
	- \int_{\OmegaLv} \gchi \bar{a} \nabla \left(\gchi \right) \cdot \nabla \left(H^-_{\delta} \right) ~\dv \leq \frac{2 d^2 {\delta_c}^2}{\mathcal{O} \left(\delchi \right)} \Bigl(\|\left|\nabla g \right\||_{L^1_{\lambda + 2}(\{-2\delta_c < g < 0\})} m_0(g) \\
	+ \|\left|\nabla g \right\||_{L^1(\{-2\delta_c < g < 0\})} m_{\lambda + 2}(g) \Bigr) \label{g_neg_bound1}
	\end{multline}
	and
	\begin{multline}
	\int_{\OmegaLv} \left(\gchi \right)^2 \bar{b} \cdot \nabla \left(H^-_{\delta} \right) ~\dv \leq \delta_c d^2 \Bigl(\|\left|\nabla g \right\||_{L^1_{\lambda + 1}(\{-2\delta_c < g < 0\})} m_0(g) \\
	+ \|\left|\nabla g \right\||_{L^1(\{-2\delta_c < g < 0\})} m_{\lambda + 1}(g) \Bigr), \label{g_neg_bound2}
	\end{multline}
	which are both $\mathcal{O} (\delta_c)$ terms if the gradient $\nabla g$ remains bounded as previously assumed.  This means both the terms (\ref{g_neg_bound1}) and (\ref{g_neg_bound2}) approach zero as $\delta_c \to 0$ and so can be dropped when the limit is taken.
	\begin{rem}
		In order to reach (\ref{g_neg_bound1}), the product rule was used on $\nabla \left(\gchi \right)$ and the ellipticity of $\bar{a}$ in Proposition \ref{ellipticity_prop} used to drop the part involving $\nabla g$.  This is where it helped that $(H^-)'(y) = 0$ when $y \geq 0$ to ensure that \mbox{\Large$\chi$}$g (H^-)'(g)(\boldsymbol{v}) \geq 0$ for all $\boldsymbol{v} \in \OmegaLv$.
	\end{rem}  
	
	Furthermore, for the first integral on the right-hand side of (\ref{g_neg_Qbound}), by again using the ellipticity of $\bar{a}$ in Proposition \ref{ellipticity_prop},
	\begin{align*}
	&-\int_{\{g<0 \}} \left(\bar{a} \nabla \left(\gchi \right) \right) \cdot \nabla \left(\gchi \right) H^-_{\delta} ~\dv \\
	\leq& -K_{\lambda}(g_0) \int_{\{g<0 \}} H^-_{\delta} \left| \nabla \left(\gchi \right) \right|^2 (1 + |\boldsymbol{v}|^{\lambda}) ~\dv \leq 0.
	\end{align*}
	This means, after dropping this term in (\ref{g_neg_Qbound}), as well as taking the limit as $\delta_c \to 0$ to remove the $\mathcal{O} (\delta_c)$ terms resulting from (\ref{g_neg_bound1}) and (\ref{g_neg_bound2}), then using Lebesgue's dominated convergence theorem to reintroduce the characteristic function,
	\begin{align*}
	&\int_{\OmegaLv} Q(\gchi, \gchi) \gchi \mathbbm{1}_{\{g<0 \}} ~\dv \nonumber \\
	=& - \frac{1}{2} \int_{\{g < 0\}} \left(\gchi \right)^2 \bar{c} ~\dv \nonumber \\
	=&~ (\lambda + 3) \int_{\{g<0 \}} \left(\int_{\OmegaLv} \gchi_* |\boldsymbol{v} - \boldsymbol{v}_*|^{\lambda} ~\dv_* \right) \left(\gchi \right)^2 ~\dv
	\end{align*}
	after writing out the definition of $\bar{c}$ from expression (\ref{c_def}).  Then, since $|\boldsymbol{v} - \boldsymbol{v}_*| \leq \langle \boldsymbol{v} \rangle \langle \boldsymbol{v}_* \rangle$ and applying the positive exponent $\lambda$ is a monotone operation,
	\begin{align*}
	&\int_{\OmegaLv} Q(\gchi, \gchi) \gchi \mathbbm{1}_{\{g<0 \}} ~\dv \nonumber \\
	\leq&~ (\lambda + 3) \left(\int_{\OmegaLv} \left|\gchi_* \right| \langle \boldsymbol{v}_* \rangle^{\lambda} ~\dv_* \right) \left(\int_{\{g<0 \}} \left(\gchi \right)^2 \langle \boldsymbol{v} \rangle^{\lambda} ~\dv \right) \\
	\leq&~ (\lambda + 3) \left\| g \right\|_{L^1_{\frac{\lambda}{2}}(\OmegaLv)} \left\| \gchi^- \right\|^2_{L^2(\OmegaLv)} \\
	\leq&~ Cr^{-\left(\frac{d}{2} + 1 \right)}(\lambda + 3) \left\| g_0 \right\|_{L^1_2(\OmegaLv)} \left\| \gchi^- \right\|^2_{L^2(\OmegaLv)},
	\end{align*}
	because $\left\| g \right\|_{L^1_{\frac{\lambda}{2}}(\OmegaLv)} \leq \left\| g \right\|_{L^1_2(\OmegaLv)} \leq Cr^{-\left(\frac{d}{2} + 1 \right)} \left\| g_0 \right\|_{L^1_2(\OmegaLv)}$ from assumption (\ref{Gaussian_bound}).
	
	So, by using this bound, along with (\ref{g-_derivative_approx}) and (\ref{g-_derivative_estimate_remainder2}) in the estimate (\ref{g-_derivative_bound1}) gives
	\begin{align*}
	&\left\| \gchi^- \right\|_{L^2(\OmegaLv)} \frac{d}{d t} \left(\left\| \gchi^- \right\|_{L^2(\OmegaLv)} \right) \nonumber \\
	\leq& Cr^{-\left(\frac{d}{2} + 1 \right)}(\lambda + 3) \left\| g_0 \right\|_{L^1_2(\OmegaLv)} \left\| \gchi^- \right\|^2_{L^2(\OmegaLv)} \nonumber \\
	&+ \Biggl(C_d O_{\frac{d}{2} + 2 - \lambda} m_{0}(g_0) m_{2}(g) \\
	&~~~~~~~~~~~~~~~ + \frac{C_Q(C_d + 1)(1 + {\Lv}^2)^2}{N^{\frac{d - 1}{2}}} \bigl| \bigl| \gchi \bigr| \bigr|^2_{L^2(\OmegaLv)} \Biggr) \left\| \gchi^- \right\|_{L^2(\OmegaLv)}.
	\end{align*}
	Equivalently, by dividing through by $\left\| \gchi^- \right\|_{L^2(\OmegaLv)}$ and defining the constant $C_{d,\lambda} := Cr^{-\left(\frac{d}{2} + 1 \right)}(\lambda + 3)$,
	\begin{align*}
	&\frac{d}{d t} \left(\left\| \gchi^- \right\|_{L^2(\OmegaLv)} \right) \\
	\leq&~ C_{d,\lambda} \left\| g_0 \right\|_{L^1_2(\OmegaLv)} \left\| \gchi^- \right\|_{L^2(\OmegaLv)} \nonumber \\
	& + C_d O_{\frac{d}{2} + 2 - \lambda} m_{0}(g_0) m_{2}(g) + C_Q(C_d + 1) \frac{(1 + {\Lv}^2)^2}{N^{\frac{d - 1}{2}}} \bigl| \bigl| \gchi \bigr| \bigr|^2_{L^2(\OmegaLv)},
	\end{align*}
	which is the required result (\ref{g-_L2_derivative_bound}) in Lemma \ref{g-_L2_derivative_lemma}.
\end{proof}

\subsection{Propagation of Moments and $L^2$-norm}
The final step in this section is to show that the moments and $L^2$-norm of the solution $g$ of the semi-discrete problem remain bounded by using the estimates on the time derivatives of those quantities.  In order to do this, the results from Lemmas \ref{moment_derivative_lemma} and \ref{L2_norm_derivative_lemma} are added and an estimate found on the time derivative of the sum $m_k(g) + \left\| \gchi \right\|_{L^2(\OmegaLv)}$.  Before this can happen, however, the results of these lemmas must be adjusted slightly.

First, for the result on the time derivative of the moments from Lemma \ref{moment_derivative_lemma}, after pulling out the maximum value of $(1 + {\Lv^2})^{\lambda + 1}$ given by the weights in the $L^2_{\lambda + 1}$-norm and noting that $m_{k + \lambda}(g) \geq m_{k}(g)$, this gives
\begin{multline*}
\frac{\textrm{d}}{\textrm{d} t} \bigl(m_k(g) \bigr) \leq - \frac{1}{2} {\varepsilon_{\chi}}^2 K_{\lambda, k} m_{0}(g_0) m_{k}(g) + C^1_{d,k} \bigl(m_0(g) + m_k(g) \bigr) \nonumber \\
+ C^2_{d,k} \frac{\mathcal{O} \left({\Lv}^{k + \frac{d}{2} + 2\lambda + 2} \right)}{N^{\frac{d - 1}{2}}} \bigl| \bigl| \gchi \bigr| \bigr|^2_{L^2(\OmegaLv)}.
\end{multline*}
The $C^1_{d,k} \bigl(m_0(g) + m_k(g) \bigr)$ term here can also be bounded by the initial data by noting that the these moments first appeared as functions of \mbox{\Large$\chi$}$g$ in Lemma \ref{moment_derivative_lemma} (or by considering that $g$ can be assumed to be zero outside of $\OmegaLv$) so they can be replaced by $L^1_k$-norms.  Then, by again pulling out the maximum value of the weight and using the bound (\ref{stability_consequence}) resulting from the stability condition (\ref{stability}),
\begin{align*}
C^1_{d,k} \bigl(m_0(g) + m_k(g) \bigr) &\leq C^1_{d,k} \bigl(\|g\|_{L^1_2(\OmegaLv)} + (1 + {\Lv}^2)^{\frac{k}{2} - 1} \|g\|_{L^1_2(\OmegaLv)} \bigr) \\
&\leq 2 C^1_{d,k} \bigl(1 + (1 + {\Lv}^2)^{\frac{k}{2} - 1} \bigr) \|g_0\|_{L^1_2(\OmegaLv)}.
\end{align*}
So, the estimate on the time derivative of the moments becomes
\begin{multline}
\frac{\textrm{d}}{\textrm{d} t} \bigl(m_k(g) \bigr) \leq - \frac{1}{2} {\varepsilon_{\chi}}^2 K_{\lambda, k} m_{0}(g_0) m_{k}(g) + 2 C^1_{d,k} \bigl(1 + (1 + {\Lv}^2)^{\frac{k}{2} - 1} \bigr) \|g_0\|_{L^1_2(\OmegaLv)} \\
+ C^2_{d,k} \frac{\mathcal{O} \left({\Lv}^{k + \frac{d}{2} + 2\lambda + 2} \right)}{N^{\frac{d - 1}{2}}} \bigl| \bigl| \gchi \bigr| \bigr|^2_{L^2(\OmegaLv)}. \label{moment_derivative_bound_final}
\end{multline}

Similarly, by pulling out the maximum value of $(1 + {\Lv^2})^{\frac{1}{2}(\lambda + 1)}$ given by the weights in the $L^2_{\lambda + 1}$-norm in the result on the time derivative of the $L^2$-norm from Lemma \ref{L2_norm_derivative_lemma},
\begin{align}
&\frac{d}{d t} \left(\left\| \gchi \right\|_{L^2(\OmegaLv)} \right) \nonumber \\
\leq& - \frac{K^S_{d, \lambda}(g_0)}{{\Lv}^{2}} \left\| \gchi \right\|_{L^2(\OmegaLv)} + K_{d, \lambda}(g_0) {\Lv}^{\frac{d}{2}} + 2 C_{\lambda} \left(2 \Lv \right)^{\frac{d}{2}} \left(1 + {\Lv}^2 \right) \|g_0\|_{L^1_{2}(\OmegaLv)} \nonumber \\
&~ + C_d O_{\frac{d}{2} + 2 - \lambda}\|g_0\|^2_{L^1_2(\OmegaLv)} + C^3_d \frac{\mathcal{O} \left({\Lv}^{2\lambda + 2} \right)}{N^{\frac{d - 1}{2}}} \bigl|\bigl| \gchi \bigr|\bigr|^{2}_{L^2(\OmegaLv)}. \label{L2_derivative_bound_final}
\end{align}

Here, by defining the variable $X(t) := m_k(g) + \left\| \gchi \right\|_{L^2(\OmegaLv)}$ and adding the estimates (\ref{moment_derivative_bound_final}) and (\ref{L2_derivative_bound_final}), this gives the ordinary differential inequality (ODI)
\begin{equation}
\frac{d X}{d t} \leq A^0_{d,k}(g_0) - A^1_{d,k}(g_0) X + A^2_{d,k} \frac{\mathcal{O} \left({\Lv}^{\kappa} \right)}{N^{\frac{d - 1}{2}}} X^2, \label{moment_L2_ODI}
\end{equation}
\begin{flalign*}
\textrm{for } && \kappa := \begin{cases}
k + \frac{d}{2} + 2\lambda + 2, &~~~\textrm{if } \Lv \geq 1, \\
2\lambda + 2, &~~~\textrm{if } \Lv < 1,
\end{cases} && &&
\end{flalign*}
\begin{flalign}
&\textrm{where } & A^0_{d,k}(g_0) &:= 2 C^1_{d,k} \bigl(1 + (1 + {\Lv}^2)^{\frac{k}{2} - 1} \bigr) \|g_0\|_{L^1_2(\OmegaLv)} + K_{d, \lambda}(g_0) {\Lv}^{\frac{d}{2}} \nonumber \\
&& &~~~~ + 2 C_{\lambda} \left(2 \Lv \right)^{\frac{d}{2}} \left(1 + {\Lv}^2 \right) \|g_0\|_{L^1_{2}(\OmegaLv)} + C_d O_{\frac{d}{2} + 2 - \lambda}\|g_0\|^2_{L^1_2(\OmegaLv)}, \nonumber \\
&& A^1_{d,k}(g_0) &:= \min \left(\frac{1}{2} {\varepsilon_{\chi}}^2 K_{\lambda, k} m_{0}(g_0), \frac{K^S_{d, \lambda}(g_0)}{{\Lv}^{2}} \right) \label{ODI_coefficients} \\
& \textrm{and} & A^2_{d,k} &:= \max \left(C^2_{d,k}, C^3_d\right). \nonumber
\end{flalign}
\begin{rem}
	The coefficient $A^0_{d,k}(g_0)$ will always be large, no matter the size of $\Lv$.  If $\Lv > 1$ then the powers of $\Lv$ with positive exponents will dominate.  On the other hand, if $\Lv < 1$ then the $O_{\frac{d}{2} + 2 - \lambda} = \mathcal{O} \left({\Lv}^{-\frac{d}{2} - 2 + \lambda} \right)$ terms dominate as this exponent is always negative when $\lambda \leq 1$.
\end{rem}

\begin{theorem} \label{propagation_theorem}
	For a solution $g$ of the semi-discrete problem (\ref{semi-discrete2}) which satisfies the stability condition (\ref{stability}) with $\varepsilon \leq \min \left(\frac{1}{4}, \varepsilon_0 \right)$, given any initial condition $g_0$ satisfying the stability condition (\ref{stability_init}), if the number of Fourier modes $N \geq N_0$, for some base number of modes $N_0 > 0$ to be defined in the proof, the numerical moments $m_k(g)$ and $L^2$-norm $\left\| \gchi \right\|_{L^2(\OmegaLv)}$ defined by (\ref{mk_def}) and (\ref{L2k_def}), respectively, satisfy
	\begin{equation}
	\sup_{t \geq 0} m_k(g) = \sup_{t \geq 0} \left\| \gchi \right\|_{L^2(\OmegaLv)} \leq \zeta(g_0), ~~~~\textrm{for any } k \geq \max (3, k_0), \label{propagation_bound}
	\end{equation}
	\begin{flalign}
	\textrm{where } && \zeta(g_0) := \max \left(m_k(g_0) + \left\| g_0 \right\|_{L^2(\OmegaLv)}, 2 \frac{A^0_{d,k}(g_0)}{A^1_{d,k}(g_0)} \right), && \label{zeta_def}
	\end{flalign}
	for $A^0_{d,k}(g_0)$ and $A^1_{d,k}(g_0)$ defined as the coefficients in (\ref{ODI_coefficients}); and $\varepsilon_0$ and $k_0$ are constants defined in the proof of Lemma \ref{moment_derivative_lemma} in expressions (\ref{epsilon_0}) and (\ref{k_0}), respectively.
	
	Furthermore, $\sup_{t \geq 0} m_k(g) = \sup_{t \geq 0} \left\| \gchi \right\|_{L^2(\OmegaLv)} \leq \widetilde{\zeta}(g_0)$, for any $k \geq \max (3, k_0)$, with $\widetilde{\zeta}(g_0)$ given by
	\begin{equation}
	\widetilde{\zeta}(g_0) := \max \left(m_k(g_0) + \left\| g_0 \right\|_{L^2(\OmegaLv)}, 2 \frac{A^1_{d,k}(g_0) N^{\frac{d - 1}{2}}}{A^2_{d,k} \mathcal{O} \left({\Lv}^{\kappa} \right)} \right), \label{zeta_def2}
	\end{equation}
	for $A_2$ defined as the coefficient in (\ref{ODI_coefficients}).
\end{theorem}
\begin{rem}
	Whereas the bound using $\zeta$ defined in (\ref{zeta_def}) may be cleaner, the bound with $\widetilde{\zeta}$ defined in (\ref{zeta_def2}) will be more useful later as it is easier to see that $\widetilde{\zeta}(h_0) \to 0$ as $\|h_0\|_{L^2(\OmegaLv)} \to 0$, for fixed $\Lv$ and $N$.  This is because $h_0$ only appears in the numerator of $\widetilde{\zeta}(h_0)$, but is in both the numerator and denominator of $\zeta(h_0)$.
\end{rem}

\begin{proof}
	First, the second order polynomial on the right-hand side of the ODI (\ref{moment_L2_ODI}) has roots
	\begin{equation*}
	X^{\pm} = \frac{A^1_{d,k}(g_0) \pm \sqrt{\left(A^1_{d,k}(g_0) \right)^2 - 4 A^0_{d,k}(g_0) A^2_{d,k} \frac{\mathcal{O} \left({\Lv}^{\kappa} \right)}{N^{\frac{d - 1}{2}}}}}{2 A^2_{d,k} \frac{\mathcal{O} \left({\Lv}^{\kappa} \right)}{N^{\frac{d - 1}{2}}}}.
	\end{equation*}
	Here, since all coefficients $A^0_{d,k}(g_0)$, $A^1_{d,k}(g_0)$ and $A^2_{d,k}(g_0)$ are positive, the polynomial is guaranteed to have two distinct real (and positive) roots by taking the number of Fourier modes $N$ large enough.  In addition, as $N \to \infty$, $X^+ \nearrow \infty$ and $X^- \to \frac{A^0_{d,k}(g_0)}{A^1_{d,k}(g_0)}$ by an application of L'H\^{o}pital's rule.  
	
	Furthermore, the lower root $X^-$ is in fact a decreasing function of $N$ for large enough $\Lv$.  To see this, by denoting $\widetilde{N} := N^{\frac{d - 1}{2}}$ and dropping the dependence on $g_0$ for convenience, note that the derivative of $X^-$ with respect to $\widetilde{N}$ is
	\begin{align*}
	\frac{\partial X^-}{\partial \widetilde{N}} &= \frac{A^1_{d,k}}{2 A^2_{d,k} \mathcal{O} \left({\Lv}^{\kappa} \right)} - \frac{\left(A^1_{d,k} \right)^2 \widetilde{N} - 2 A^0_{d,k} A^2_{d,k} \mathcal{O} \left({\Lv}^{\kappa} \right)}{\widetilde{N} \sqrt{\left(A^1_{d,k} \right)^2 - 4 A^0_{d,k} A^2_{d,k} \frac{\mathcal{O} \left({\Lv}^{\kappa} \right)}{\widetilde{N}}}} \\
	&\leq \frac{A^1_{d,k}}{2 A^2_{d,k} \mathcal{O} \left({\Lv}^{\kappa} \right)} - \frac{\left(A^1_{d,k} \right)^2 \widetilde{N} - 2 A^0_{d,k} A^2_{d,k} \mathcal{O} \left({\Lv}^{\kappa} \right)}{A^1_{d,k} \widetilde{N}} \\
	&= \frac{\left(A^1_{d,k} \right)^2 \widetilde{N} - 2\left(A^1_{d,k} \right)^2 A^2_{d,k} \mathcal{O} \left({\Lv}^{\kappa} \right) \widetilde{N} + 4 A^0_{d,k} \left(A^2_{d,k} \mathcal{O} \left({\Lv}^{\kappa} \right) \right)^2}{2A^1_{d,k} A^2_{d,k} \mathcal{O} \left({\Lv}^{\kappa} \right) \widetilde{N}},
	\end{align*}
	since $N$ has been chosen large enough that the quantity under the square root is positive.  Another consequence of this fact is that $4 A^0_{d,k} A^2_{d,k} \mathcal{O} \left({\Lv}^{\kappa} \right) < \left(A^1_{d,k} \right)^2 \widetilde{N}$ and so
	\begin{align*}
	\frac{\partial X^-}{\partial \widetilde{N}} &< \frac{\left(A^1_{d,k} \right)^2 \widetilde{N} - 2\left(A^1_{d,k} \right)^2 A^2_{d,k} \mathcal{O} \left({\Lv}^{\kappa} \right) \widetilde{N} + \left(A^1_{d,k} \right)^2 A^2_{d,k} \mathcal{O} \left({\Lv}^{\kappa} \right) \widetilde{N}}{2A^1_{d,k} A^2_{d,k} \mathcal{O} \left({\Lv}^{\kappa} \right) \widetilde{N}} \\
	&= \frac{\left(1 - A^2_{d,k} \mathcal{O} \left({\Lv}^{\kappa} \right) \right)\left(A^1_{d,k} \right)^2 \widetilde{N}}{2A^1_{d,k} A^2_{d,k} \mathcal{O} \left({\Lv}^{\kappa} \right) \widetilde{N}}.
	\end{align*}
	The denominator here is always positive and so, for large enough $\Lv$ to cause $A^2_{d,k} \mathcal{O} \left({\Lv}^{\kappa} \right) \geq 1$, this means that $\frac{\partial X^-}{\partial \widetilde{N}} < 0$ and so the lower root $X^-$ decreases in $\widetilde{N} = N^{\frac{d - 1}{2}}$, or equivalently as the number of Fourier modes $N$ increases.  This means that $X^- \searrow \frac{A^0_{d,k}(g_0)}{A^1_{d,k}(g_0)}$ as $N \to \infty$.
	
	Then, since $X^+ \nearrow \infty$ as $N \to \infty$, there exists some $N_0 > 0$ such that, for $N \geq N_0$, the initial sum of $k$-th moment and $L^2$-norm $X(0) = m_k(g_0) + \left\| \gchi_0 \right\|_{L^2(\OmegaLv)}$ satisfies $X(0) \leq X^+$.  If, in addition, $X(0) \geq X^-$ the derivative $X'(t)$ must be negative, since it is bounded above by the polynomial which is negative between the roots, and so $X(t)$ is decreasing.  If, for some $T > 0$, the solution drops to $X(T) < X^-$ then it could be that $X'(t) \leq 0$ or $X'(t) > 0$ for $t \geq T$.  If $X'(t) \leq 0$ then the solution will remain bounded but will also never drop below zero as, by definition of the moments and $L^2$-norm in (\ref{mk_def}) and (\ref{L2k_def}), respectively, $X(t) = m_k(g) + \left\| \gchi \right\|_{L^2(\OmegaLv)} \geq 0$.  If $X'(t) > 0$, however, as soon as the solution reaches $X^-$ it must decrease again as once more the derivative is forced to be negative.  This means that when $X(t)$ drops below the lower root $X^-$ it must remain bounded above by $X^-$.  In particular, this means that if $X(0) \leq X^-$ then $X(t)$ may increase above $X(0)$ but $X(t) \leq X^-$ and so
	\begin{equation}
	X(t) \leq \max \left(X(0), X^-\right) \label{X_bound}.
	\end{equation}
	
	A sketch of this argument is shown in Fig.\ \ref{ODIDiagram} where, given the initial condition $X(0) \leq X^+$, the hatched region shows the possible values for the derivative $X'(t)$. In addition, by some simple factorisation, note that when $N \geq N_0$ and so the discriminant of the polynomial on the right-hand side of (\ref{moment_L2_ODI}) is positive,
	\begin{align}
	X^{-} &= \frac{A^1_{d,k}(g_0) - \sqrt{\left(A^1_{d,k}(g_0) \right)^2 - 4 A^0_{d,k}(g_0) A^2_{d,k} \frac{\mathcal{O} \left({\Lv}^{\kappa} \right)}{N^{\frac{d - 1}{2}}}}}{2 A^2_{d,k} \frac{\mathcal{O} \left({\Lv}^{\kappa} \right)}{N^{\frac{d - 1}{2}}}} \label{X-_bound} \\
	&= \frac{A^1_{d,k}(g_0) \left(1 - \sqrt{1 - 4 \frac{A^0_{d,k}(g_0) A^2_{d,k}}{\left(A^1_{d,k}(g_0) \right)^2} \frac{\mathcal{O} \left({\Lv}^{\kappa} \right)}{N^{\frac{d - 1}{2}}}} \right)}{2 A^2_{d,k} \frac{\mathcal{O} \left({\Lv}^{\kappa} \right)}{N^{\frac{d - 1}{2}}}} \nonumber \\
	&\leq \frac{A^1_{d,k}(g_0) \left(1 - 1 + 4 \frac{A^0_{d,k}(g_0) A^2_{d,k}}{\left(A^1_{d,k}(g_0) \right)^2} \frac{\mathcal{O} \left({\Lv}^{\kappa} \right)}{N^{\frac{d - 1}{2}}} \right)}{2 A^2_{d,k} \frac{\mathcal{O} \left({\Lv}^{\kappa} \right)}{N^{\frac{d - 1}{2}}}} = 2 \frac{A^0_{d,k}(g_0)}{A^1_{d,k}(g_0)}, \nonumber
	\end{align}
	since a positive discriminant is only possible when $1 - 4 \frac{A^0_{d,k}(g_0) A^2_{d,k}}{\left(A^1_{d,k}(g_0) \right)^2} \frac{\mathcal{O} \left({\Lv}^{\kappa} \right)}{N^{\frac{d - 1}{2}}} \leq 1$ and so the square root of this quantity is larger.  So, using this bound on $X^-$ in identity (\ref{X_bound}), as well as the fact that $X(0) = m_k(g_0) + \left\| \gchi_0 \right\|_{L^2(\OmegaLv)} \leq m_k(g_0) + \left\| g_0 \right\|_{L^2(\OmegaLv)}$, gives
	\begin{equation*}
	X(t) \leq \max \left(m_k(g_0) + \left\| g_0 \right\|_{L^2(\OmegaLv)}, 2 \frac{A^0_{d,k}(g_0)}{A^1_{d,k}(g_0)} \right) \label{X_bound2}.
	\end{equation*}
	
	In conclusion, since both $m_k(g)$ and $\left\| \gchi \right\|_{L^2(\OmegaLv)}$ are positive, they each have the same bound as $X(t)$, which gives the required result (\ref{propagation_bound}) coupled with (\ref{zeta_def}) in Theorem \ref{propagation_theorem}. 
	\begin{figure}[!hbt]
		\centering
		\includegraphics[width=0.9\linewidth]{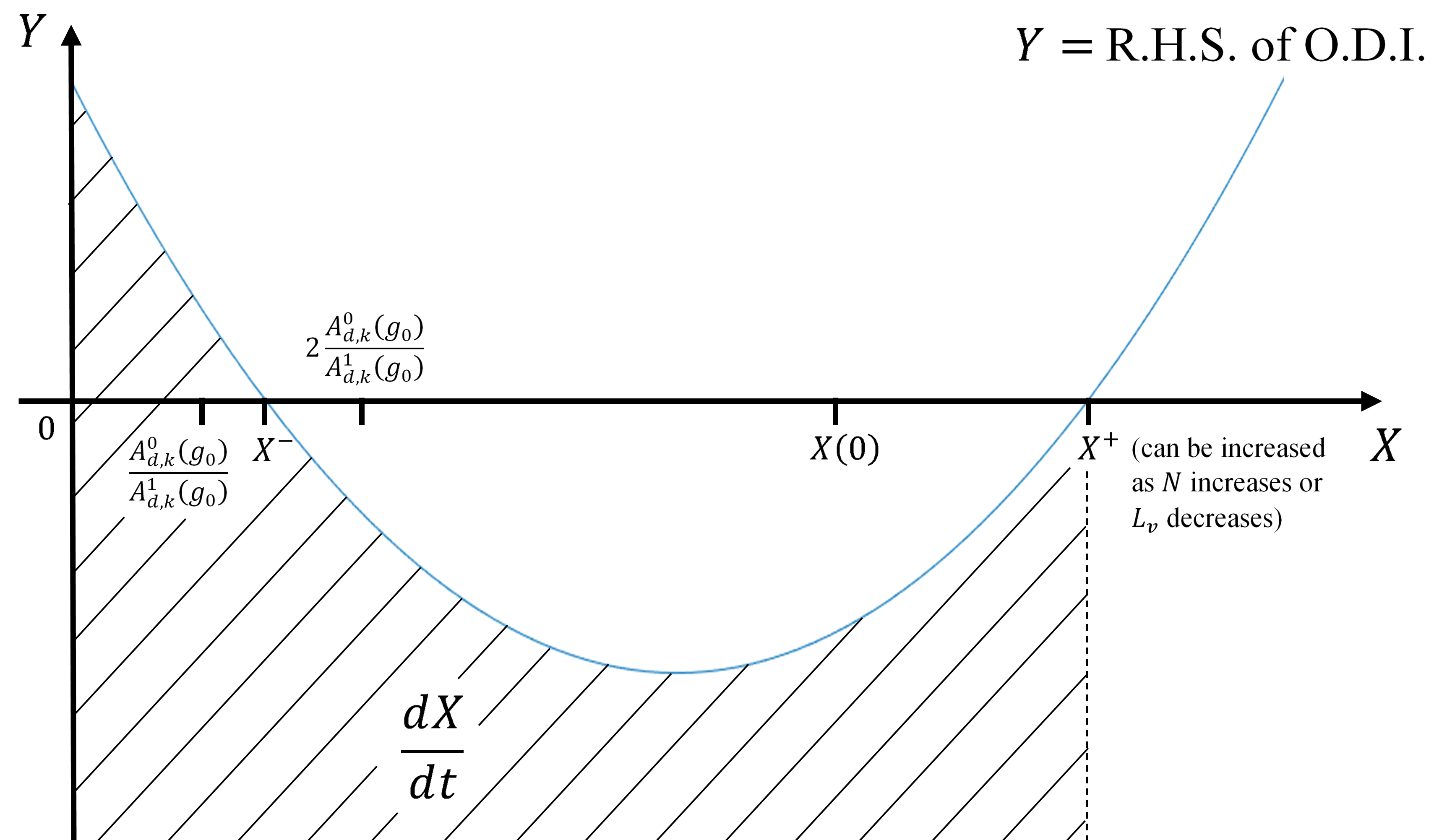} 
		\caption{A sketch of the parabola which is the right-hand side of the ODI (\ref{moment_L2_ODI}) for large enough number of Fourier modes $N$ to give two distinct real roots.  Since the derivative of $X(t) = m_k(g) + \left\| \gchi \right\|_{L^2(\OmegaLv)}$ is bounded above by this curve, if the initial condition satisfies $X(0) < X^+$, the derivative remains in the hatched region.}
		\label{ODIDiagram}
	\end{figure}
	
	Additionally, the result involving $\tilde{\zeta}$ defined in (\ref{zeta_def2}) is true by using the expression (\ref{X_bound}) and trivially bounding the lower root $X^-$ by removing the square root in expression (\ref{X-_bound})
\end{proof}


\section{Existence and Uniqueness for the Semi-discrete Problem} \label{ExistenceAndRegularity}
The results on the propagation of moments and $L^2$-norm from Theorem \ref{propagation_theorem} can now be used to show that the semi-discrete problem (\ref{semi-discrete2}) has a unique solution, for some initial condition $f_0 \in L^2(\Rd)$ which has support contained in $\OmegaLv$.  The proof of this is an application of the contraction mapping theorem, which first requires the definition of a Banach space in which the solution will live.  For this purpose, first let $g_0 = \Pi^N_{\Lv}f_0$ and then define the Banach space $\mathcal{B}_k \subset C([0,T]; L^2(\OmegaLv))$, for any required $k \geq \max (3, k_0)$ with $k_0$ defined by (\ref{k_0}), as
\begin{align*}
\mathcal{B}_k := \Bigl\{f \in C([0,T]; L^2(\OmegaLv)): &\sup_{t \in [0, T]} m_k(f(t)) \leq 2\zeta(g_0), \\
&\sup_{t \in [0, T]} \left\| \Chi f(t) \right\|_{L^2(\OmegaLv)} \leq 2\zeta(g_0) \Bigr\},
\end{align*}
where $\zeta(g_0)$ is defined by (\ref{zeta_def}) in Theorem \ref{propagation_theorem}.

To introduce an operator to be used in the contraction mapping theorem, also notice that the semi-discrete problem (\ref{semi-discrete2}) is equivalent to stating that its solution $g$ satisfies $g(t) = g_0 + \int_{0}^{t} Q_c(g,g)(s) ~\textrm{d}s$.  As a result, define the operator $\mathcal{T}: \mathcal{B}_k \to C([0,T]; L^2(\OmegaLv))$ by
\begin{equation}
\mathcal{T}(g)(t) := g_0 + \int_{0}^{t} Q_c(g,g)(s) ~\textrm{d}s. \label{T_map}
\end{equation}

\begin{theorem} \label{existence_theorem}
	For an initial condition $g_0 = \Pi^N_{2\Lv} f_0 \in L^1_k \bigcap L^2(\OmegaLv)$, where $k \geq \max (3, k_0)$, which satisfies the stability condition (\ref{stability_init}) with $\varepsilon \leq \min \left(\frac{1}{4}, \varepsilon_0 \right)$, if the number of Fourier modes $N \geq N_0$ and $\Lv \geq L_0$, there is a unique solution $g \in C([0,T]; L^2(\OmegaLv))$ to the semi-discrete problem (\ref{semi-discrete2}), for any $T > 0$, which satisfies
	\begin{flalign*}
	&& m_k(g) \leq \zeta(g_0) ~~\textrm{and} ~~\left\| \gchi \right\|_{L^2(\OmegaLv)} \leq \zeta(g_0), && \textrm{for all } t \leq T,
	\end{flalign*}
	with $\varepsilon_0$ and $k_0$ are defined in the proof of Lemma \ref{moment_derivative_lemma} in expressions (\ref{epsilon_0}) and (\ref{k_0}), respectively; $\zeta(g_0)$ defined by expression (\ref{zeta_def}) and $N_0$ described in Theorem \ref{propagation_theorem}; and $L_0 > 0$ large enough which will be described in the current proof.
\end{theorem}
\begin{proof}
	First it will be shown that the map $\mathcal{T}$ defined in (\ref{T_map}) is a contraction.  To see this, for $f$, $\widetilde{f} \in \mathcal{B}_k$, note that
	\begin{align}
	&\mathcal{T}(f) - \mathcal{T}(\widetilde{f}) \nonumber \\
	&\qquad= \int_{0}^{t} \left(Q_c(f,f)(s) - Q_c(\widetilde{f},\widetilde{f})(s) \right) ~\textrm{d}s \label{T_diff} \\
	&\qquad\ \ = \int_{0}^{t} \Biggl(\Pi^N_{2\Lv} \left(Q(\Chi f, \Chi f) - Q(\Chi \widetilde{f}, \Chi \widetilde{f}) \right) \nonumber \\
	&~~~~~~~~~~~~~~ - \frac{1}{2}\biggl(\left(\gamma_1 - \widetilde{\gamma}_1 \right) + \sum_{j = 1}^{d} \left(\gamma_{j+1} - \widetilde{\gamma}_{j+1} \right) v_j + \left(\gamma_{d+2} - \widetilde{\gamma}_{d+2} \right) |\boldsymbol{v}|^2 \biggr) \Biggr) ~\textrm{d}s \nonumber,
	\end{align}
	by using expression (\ref{Conserved_Q}) for $Q_c$.  So, by pulling out the supremum from the integral with respect to $s$, the triangle inequality and noting that the Lagrange multipliers $\gamma_j$ are constants, for $j = 1, 2, \ldots, d+2$,
	\begin{align*}
	&\left\|\mathcal{T}(f) - \mathcal{T}(\widetilde{f}) \right\|_{L^2(\OmegaLv)} \\
	\leq&~ t \sup_{s \in [0,t]} \Biggl(\left\| \Pi^N_{2\Lv} \left(Q(\Chi f, \Chi f) - Q(\Chi \widetilde{f}, \Chi \widetilde{f}) \right) \right\|_{L^2(\OmegaLv)} \\
	&~~~~~~~~~~~~ + \frac{1}{2}\biggl(\left|\gamma_1 - \widetilde{\gamma}_1 \right| \left\| 1 \right\|_{L^2(\OmegaLv)} + \sum_{j = 1}^{d} \left|\gamma_{j+1} - \widetilde{\gamma}_{j+1} \right| \left\| v_j \right\|_{L^2(\OmegaLv)} \\
	&~~~~~~~~~~~~~~~~~~~~~~~~~~~~~~~~~~~~~~~~~~~~~~+ \left|\gamma_{d+2} - \widetilde{\gamma}_{d+2} \right| \left\| |\boldsymbol{v}|^2 \right\|_{L^2(\OmegaLv)} \biggr) \Biggr). 
	\end{align*}
	
	Now, by using the definition of $\gamma_1$ in (\ref{Lagrange_multipliers}), and explicitly writing out the moments $M_{\phi(\boldsymbol{v})}$, after using the triangle inequality,
	\begin{align*}
	\left|\gamma_1 - \widetilde{\gamma}_1 \right| \leq& O_d \int_{\OmegaLv} \left|\Pi^N_{2\Lv} \left(Q(\Chi f, \Chi f) - Q(\Chi \widetilde{f}, \Chi \widetilde{f}) \right) \right| ~\dv \\
	&~~ + O_{d+2} \int_{\OmegaLv} \left|\Pi^N_{2\Lv} \left(Q(\Chi f, \Chi f) - Q(\Chi \widetilde{f}, \Chi \widetilde{f}) \right) \right|  |\boldsymbol{v}|^2 ~\dv \\
	&~~~~~\leq O_d \left\| \Pi^N_{2\Lv} \left(Q(\Chi f, \Chi f) - Q(\Chi \widetilde{f}, \Chi \widetilde{f}) \right) \right\|_{L^2(\OmegaLv)}\! \left\| 1 \right\|_{L^2(\OmegaLv)} \\
	& ~~~~~~~+ O_{d+2} \left\| \Pi^N_{2\Lv} \left(Q(\Chi f, \Chi f) - Q(\Chi \widetilde{f}, \Chi \widetilde{f}) \right) \right\|_{L^2(\OmegaLv)} \left\| |\boldsymbol{v}|^2 \right\|_{L^2(\OmegaLv)},
	\end{align*}
	by the Cauchy-Schwarz inequality.  Similarly, by using the definitions in (\ref{Lagrange_multipliers}) for the remaining Lagrange multipliers,
	\begin{equation*}
	\left|\gamma_{j+1} - \widetilde{\gamma}_{j+1} \right| \leq O_{d+2} \left\| \Pi^N_{2\Lv} \left(Q(\Chi f, \Chi f) - Q(\Chi \widetilde{f}, \Chi \widetilde{f}) \right) \right\|_{L^2(\OmegaLv)} \left\| v_j \right\|_{L^2(\OmegaLv)},
	\end{equation*}
	for $j = 1, 2, \ldots, d$, and
	\begin{align*}
	\left|\gamma_{d+2} - \widetilde{\gamma}_{d+2} \right| &\leq O_{d+2} \left\| \Pi^N_{2\Lv} \left(Q(\Chi f, \Chi f) - Q(\Chi \widetilde{f}, \Chi \widetilde{f}) \right) \right\|_{L^2(\OmegaLv)} \left\| 1 \right\|_{L^2(\OmegaLv)} \\
	+& O_{d+4} \left\| \Pi^N_{2\Lv} \left(Q(\Chi f, \Chi f) - Q(\Chi \widetilde{f}, \Chi \widetilde{f}) \right) \right\|_{L^2(\OmegaLv)} \left\| |\boldsymbol{v}|^2 \right\|_{L^2(\OmegaLv)}.
	\end{align*}
	
	Next, by considering that the expressions (\ref{collision_invariant_L2}) show $\| 1 \|_{L^2(\OmegaLv)} = O_{-\frac{d}{2}}$, $\| v_j \|_{L^2(\OmegaLv)} = O_{-(\frac{d}{2} + 1 )}$, for $j = 1, 2, \ldots, d$, and $\| |\boldsymbol{v}|^2 \|_{L^2(\OmegaLv)} = O_{\left(-\frac{d}{2} + 2 \right)}$, all of the $O_d$ terms cancel with the $O_{-d}$ terms from the $L^2$-norms of the collision invariants and this means that there exists some constant $\widetilde{C}_1 > 0$ such that
	\begin{align}
	\left\|\mathcal{T}(f) - \mathcal{T}(\widetilde{f}) \right\|_{L^2(\OmegaLv)} &\leq \widetilde{C}_1 t \sup_{s \in [0,t]} \left\| \Pi^N_{2\Lv} \left(Q(\Chi f, \Chi f) - Q(\Chi \widetilde{f}, \Chi \widetilde{f}) \right) \right\|_{L^2(\OmegaLv)} \nonumber \\
	&\leq \widetilde{C}_1 t \sup_{s \in [0,t]} \left\| Q(\Chi f, \Chi f) - Q(\Chi \widetilde{f}, \Chi \widetilde{f}) \right\|_{L^2(\OmegaLv)}, \label{Contraction_Bound}
	\end{align}
	by Parseval's theorem, and so proving that $\mathcal{T}$ is a contractive map reduces to showing that $Q$ is.  To do this, notice that by the bi-linearity of $Q$,
	\begin{equation*}
	Q(\Chi f, \Chi f) - Q(\Chi \widetilde{f}, \Chi \widetilde{f}) = Q(\Chi f - \Chi \widetilde{f}, \Chi f + \Chi \widetilde{f})
	\end{equation*}
	and so an estimate is required on $\|Q(F, G)\|_{L^2(\OmegaLv)}$, where $F = \Chi f - \Chi \widetilde{f} = \Chi (f - \widetilde{f})$ and $G = \Chi f + \Chi \widetilde{f} = \Chi (f + \widetilde{f})$.  Well, by using these functions in the $L^2$-estimate (\ref{Q_L2_estimate}) resulting from Proposition \ref{LBH_prop},
	\begin{align*}
	&\|Q(\Chi (f - \widetilde{f}), \Chi (f + \widetilde{f}))\|_{L^2(\OmegaLv)} \\
	&\qquad\leq ~ C_H \left(\|\Chi (f - \widetilde{f})\|_{L^1_{\lambda + 2}} + \|\Chi (f - \widetilde{f})\|_{L^2}\right) \|\Chi (f + \widetilde{f})\|_{H^2_{\lambda + 2}}.
	\end{align*}
	
	Then, using this estimate for the $L^2$-norm in (\ref{Contraction_Bound}) and the Cauchy-Schwarz inequality, a bound for the maximum value of the weights and the triangle inequality on the $H^2$-norms gives
	\begin{align*}
	&\left\|\mathcal{T}(f) - \mathcal{T}(\widetilde{f}) \right\|_{L^2(\OmegaLv)} \\
	&
\ \leq \widetilde{C}_1 C_H t \sup_{s \in [0,t]} \left(\|\Chi (f - \widetilde{f})\|_{L^1_{\lambda + 2}} + \|\Chi (f - \widetilde{f})\|_{L^2}\right) \|\Chi (f + \widetilde{f})\|_{H^2_{\lambda + 2}} \nonumber \\
	&\ \ \leq \widetilde{C}_1 C_H t \left((2\Lv)^{\frac{d}{2}}(1 + \left(\Lv)^2 \right)^{\lambda + 2} + 1\right)\sup_{s \in [0,t]} \|\Chi (f - \widetilde{f})\|_{L^2} \|\Chi (f + \widetilde{f})\|_{H^2} \\
	&\ \ \ \leq \widetilde{C}_1 C_H C_{\chi} t \left((2\Lv)^{\frac{d}{2}}(1 + \left(\Lv)^2 \right)^{\lambda + 2} + 1\right) \!\sup_{s \in [0,t]} \!\left(\|f\|_{H^2} + \|\widetilde{f}\|_{H^2}\right) \|f - \widetilde{f}\|_{L^2},
	\end{align*}
	by bounding the extension operators by 1 in the $L^2$-norm and using the bound (\ref{extension_bound}) in the $H^2$-norms.  This means, if it can also be assumed that the $H^2$-norms of the functions $f, \widetilde{f} \in \mathcal{B}_k$ are uniformly bounded up to some time $T_{\mathcal{T}}$ (which will be proven in the next section and does not rely on this theorem directly), for a constant $C_{\lambda, \Lv} > 0$ depending on $\lambda$ and $\Lv$,
	\begin{equation}
	\sup_{t \in [0,T_{\mathcal{T}}]} \left\|\mathcal{T}(f) - \mathcal{T}(\widetilde{f}) \right\|_{L^2(\OmegaLv)} \leq T_{\mathcal{T}} C_{\lambda, \Lv}\sup_{t \in [0,T_{\mathcal{T}}]} \|f - \widetilde{f}\|_{L^2}. \label{contraction_statement}
	\end{equation}
	So, if $T_{\mathcal{T}}$ is chosen small enough so that $T_{\mathcal{T}} C_{\lambda, \Lv} < 1$ then the operator $\mathcal{T}$ is indeed a contraction.
	
	Furthermore, setting $\widetilde{f} = 0$ in the above estimates which prove that $Q$ is a contraction,
	\begin{align*}
	&\sup_{t \in [0,T_{\mathcal{T}}]} \left\|\mathcal{T}(f) \right\|_{L^2(\OmegaLv)} \leq \left\|g_0 \right\|_{L^2(\OmegaLv)} + \sup_{t \in [0,T_{\mathcal{T}}]}  \left\| \int_{0}^{t} Q_c(f,f)(s) ~\textrm{d}s  \right\|_{L^2(\OmegaLv)} \\
	&\qquad\ \ \leq \left\|g_0 \right\|_{L^2(\OmegaLv)} + T_{\mathcal{T}} C_{\lambda, \Lv} \sup_{t \in [0,T_{\mathcal{T}}]} \|f\|_{L^2}  \leq \left\|g_0 \right\|_{L^2(\OmegaLv)} + T_{\mathcal{T}} C_{\lambda, \Lv} \zeta(g_0),
	\end{align*}
	because $f \in \mathcal{B}_k$.  That means, since $\left\|g_0 \right\|_{L^2(\OmegaLv)} \leq \zeta(g)$ (by definition of $\zeta$ in (\ref{zeta_def})) and $T_{\mathcal{T}}$ has already been chosen such that $T_{\mathcal{T}} C_{\lambda, \Lv} \leq 1$ in the contraction statement (\ref{contraction_statement}), $\sup_{t \in [0,T_{\mathcal{T}}]} \left\|\mathcal{T}(f) \right\|_{L^2(\OmegaLv)} \leq 2 \zeta(g_0)$.

	Also, notice that
	\begin{align*}
	m_k(\mathcal{T}(g)(t)) &= \int_{\Rd} \left|g_0 + \int_{0}^{t} Q_c(f,f)(s) ~\textrm{d}s \right| \langle \boldsymbol{v} \rangle^{k} ~\dv \\
	&\leq m_k(g_0) + t \sup_{s \in [0,t]} \int_{\Rd} \left|Q_c(f,f)(s) \right| \langle \boldsymbol{v} \rangle^{k} ~\dv.
	\end{align*}
	Then, by using the expanded identity (\ref{Q_expansion}) for $Q_c$, as well as the triangle inequality and noting that $f$ will only be defined inside $\OmegaLv$ so that $Q$ has support in $\Omega_{2\Lv}$,
	\begin{align}
	&\int_{\Rd} \left|Q_c(f,f)(s) \right| \langle \boldsymbol{v} \rangle^{k} ~\dv \nonumber \\
	&\qquad\leq \int_{\Rd} \left|Q(\Chi f, \Chi f)(s) \right| \langle \boldsymbol{v} \rangle^{k} ~\dv + \left\|(Q_c(f,f) - Q_u(f,f)) \langle \boldsymbol{v} \rangle^{k} \right\|_{L^1(\Omega_{2\Lv})} \nonumber \\
	&\qquad\ \ + \left\|\left(1 - \Pi^N_{2\Lv}\right) Q(\Chi f, \Chi f) \langle \boldsymbol{v} \rangle^{k} \right\|_{L^1(\Omega_{2\Lv})}. \label{T_moment_bound0}
	\end{align}
	Here, by the arguments in the proof of Lemma \ref{moment_derivative_lemma}, the $L^1$-norm terms in (\ref{T_moment_bound0}) are bounded as
	\begin{align}
	&\left\|(Q_c(f,f) - Q_u(f,f)) \langle \boldsymbol{v} \rangle^{k} \right\|_{L^1(\Omega_{2\Lv})} + \left\|\left(1 - \Pi^N_{2\Lv}\right) Q(\Chi f, \Chi f) \langle \boldsymbol{v} \rangle^{k} \right\|_{L^1(\Omega_{2\Lv})} \nonumber \\
	&\ \ \leq  \frac{C^2_d C r^{\frac{d}{2}} m_0(g_0)}{\sqrt{2k + 1}} m_{k+\lambda}(f) + \frac{C_k C_d}{\sqrt{2k + 1}}\bigl(m_0(f) + m_{k+\lambda}(f) \bigr) \nonumber \\
	&~~~~~+ \frac{C^1_d C_Q}{N^{\frac{d-1}{2}} \sqrt{2k + 1}} \mathcal{O} \left({\Lv}^{k + \frac{d}{2} + \lambda + 2} \right) \bigl| \bigl| \Chi f \bigr| \bigr|^2_{L^2(\OmegaLv)} \nonumber \\
	&\qquad\ \  \leq  \frac{C^2_d C r^{\frac{d}{2}} (1 \!+\! {\Lv}^2)^{\lambda} m_0(g_0)}{\sqrt{2k \!+\! 1}} m_{k}(f) + \frac{C_k C_d}{\sqrt{2k \!+ \!1}}\bigl(m_k(f) +(1 + {\Lv}^2)^{\lambda} m_{k}(f) \bigr) \nonumber \\
	&~~~~~~~~~~~~+ \frac{C^1_d C_Q}{N^{\frac{d-1}{2}} \sqrt{2k \!+\! 1}} \mathcal{O} \left({\Lv}^{k + \frac{d}{2} + \lambda + 2} \right) \bigl| \bigl| \Chi f \bigr| \bigr|^2_{L^2(\OmegaLv)}. \label{T_moment_bound2}
	\end{align}

	Now, if it can also be assumed that the moments of the collision operator, namely $\int_{\Rd} \left|Q(\Chi f, \Chi f)(s) \right| \langle \boldsymbol{v} \rangle^{k} ~\dv$, can be bounded by moments of $f$ then, by combining that with the estimates in (\ref{T_moment_bound2}), since $f \in \mathcal{B}_k$ and every term will involve $m_k(f)$ or $\bigl| \bigl| \Chi f \bigr| \bigr|_{L^2(\OmegaLv)}$, this means that the supremum of these quantities are bounded by some function of $\zeta(g_0)$, which also depends on $k$, $\lambda$, $\Lv$ and $N$, say $\mathcal{C}_{k, \lambda, \Lv, N}(\zeta(g_0))$.  More precisely, this means
	\begin{flalign*}
	&& \sup_{s \in [0,t]} \int_{\Rd} \left|Q_c(f,f)(s) \right| \langle \boldsymbol{v} \rangle^{k} ~\dv \leq \mathcal{C}_{k, \lambda, \Lv, N}(\zeta(g_0)),&& \textrm{for } t \leq T_{\mathcal{T}},
	\end{flalign*}
	and so
	\begin{align*}
	\sup_{t \in [0,T_{\mathcal{T}}]} m_k(\mathcal{T}(f)(t)) &\leq \sup_{t \in [0,T_{\mathcal{T}}]} \left(m_k(g_0) + t \sup_{s \in [0,t]} \int_{\Rd} \left|Q_c(f,f)(s) \right| \langle \boldsymbol{v} \rangle^{k} ~\dv \right) \\
	&\leq m_k(g_0) + T_{\mathcal{T}} \mathcal{C}_{k, \lambda, \Lv, N}(\zeta(g_0)).
	\end{align*}
	Here, since $m_k(g_0) \leq \zeta(g_0)$ by definition of $\zeta$ in (\ref{zeta_def}), if $T_{\mathcal{T}}$ is additionally chosen small enough so that $T_{\mathcal{T}} \mathcal{C}_{k, \lambda, \Lv, N}(\zeta(g_0)) \leq \zeta(g_0)$ then it is also true that $\sup_{t \in [0,T_{\mathcal{T}}]} m_k(\mathcal{T}(f)(t)) \leq 2 \zeta(g_0)$.  
	
	This means that $\mathcal{T}(f) \in \mathcal{B}_k$ because $\sup_{t \in [0,T_{\mathcal{T}}]} m_k(\mathcal{T}(f)(t)) \leq 2 \zeta(g_0)$ and $\sup_{t \in [0,T_{\mathcal{T}}]} \left\|\mathcal{T}(f)(t) \right\|_{L^2(\OmegaLv)} \leq 2 \zeta(g_0)$.  Therefore, since $\mathcal{T}: \mathcal{B}_k \to \mathcal{B}_k$ is a contraction, there is a unique solution to the semi-discrete problem (\ref{semi-discrete2}) for $0 < t \leq T_{\mathcal{T}}$ by the contraction mapping theorem.
	
	The next step is to show that there is a unique solution for all time $t > 0$.  This is achieved by first proving that the stability assumption (\ref{stability}) is true for all $t > 0$.  In particular, by using the result of Lemma \ref{g-_L2_derivative_lemma} in the case where $g \in \mathcal{B}_k$,
	\begin{align*}
	&\frac{d}{d t} \left(\left\| \gchi^- \right\|_{L^2(\OmegaLv)} \right) \leq C_{d,\lambda} \left\| g_0 \right\|_{L^1_2(\OmegaLv)} \left\| \gchi^- \right\|_{L^2(\OmegaLv)} \nonumber \\
	&\qquad + 2C_d O_{\frac{d}{2} + 2 - \lambda} m_{0}(g_0) \zeta(g_0) + 4C_Q(C_d + 1) \frac{(1 + {\Lv}^2)^2}{N^{\frac{d - 1}{2}}} \left(\zeta(g_0) \right)^2.
	\end{align*}
	So, by Gr\"{o}nwall's inequality,
	\begin{equation*}
	\left\| \gchi^-(t) \right\|_{L^2(\OmegaLv)} \leq \epsilon(t,\Lv,N),
	\end{equation*}
	where
	\begin{flalign*}
	\epsilon(t,\Lv,N) &= e^{C_{d,\lambda} \left\| g_0 \right\|_{L^1_2(\OmegaLv)} t} \Biggl(\left\| \gchi_0^- \right\|_{L^2(\OmegaLv)} + 2C_d O_{\frac{d}{2} + 2 - \lambda} m_{0}(g_0) \zeta(g_0) \\
	&~~~~~~~~~~~~~~~~~~~~~~~~~~~ + 4C_Q(C_d + 1) \frac{(1 + {\Lv}^2)^2}{N^{\frac{d - 1}{2}}} \left(\zeta(g_0) \right)^2 \Biggr). 
	\end{flalign*}
	
	Now, since $\epsilon(t,\Lv,N)$ is an increasing function of $t$, $\epsilon(t,\Lv,N) \leq \epsilon(T_0,\Lv,N)$ for some $T_0 > 0$ and the negative part of the solution satisfies
	\begin{equation*}
	\int_{\{g(t,\boldsymbol{v}) < 0\}} |g(t, \boldsymbol{v})|\langle\boldsymbol{v}\rangle^2 ~\dv \leq (1 + {\Lv}^2) \left\| \gchi^-(t) \right\|_{L^2(\OmegaLv)} \leq (1 + {\Lv}^2) \epsilon(T_0,\Lv,N).
	\end{equation*}
	This means
	\begin{align}
	\frac{\int_{\{g(t,\boldsymbol{v}) < 0\}} |g(t, \boldsymbol{v})|\langle\boldsymbol{v}\rangle^2 ~\dv}{\int_{\{g(t,\boldsymbol{v}) \geq 0\}} g(t, \boldsymbol{v})\langle\boldsymbol{v}\rangle^2 ~\dv} &= \frac{\int_{\{g(t,\boldsymbol{v}) < 0\}} |g(t, \boldsymbol{v})|\langle\boldsymbol{v}\rangle^2 ~\dv}{\int_{\OmegaLv} g(t, \boldsymbol{v})\langle\boldsymbol{v}\rangle^2 ~\dv - \int_{\{g(t,\boldsymbol{v}) < 0\}} g(t, \boldsymbol{v})\langle\boldsymbol{v}\rangle^2 ~\dv} \nonumber \\
	&\leq \frac{(1 + {\Lv}^2) \epsilon(T_0,\Lv,N)}{\int_{\OmegaLv} g(t, \boldsymbol{v})\langle\boldsymbol{v}\rangle^2 ~\dv - (1 + {\Lv}^2) \epsilon(T_0,\Lv,N)}. \label{neg_ratio}
	\end{align}
	Then, for fixed $T_0 > 0$ small enough, $\epsilon(T_0,\Lv,N)$ can be made smaller by increasing $\Lv$ and $N$.  This means that $T_0$, $\Lv$ and $N$ can be chosen so that the right-hand side of (\ref{neg_ratio}) is less than $\varepsilon$ in the assumption (\ref{stability_init}).  This is only true for times $0 < t \leq T_0$, however, but the argument can be repeated from time $t = T_0$ to give the same result for times $T_0 < t \leq 2T_0$ and so on.  This then gives the result for all time $t > 0$.  Equivalently, this means that the stability condition (\ref{stability}) is true for all $t > 0$.  
	
	Therefore, the only assumption of Theorem \ref{propagation_theorem} is true and so, for any $k \geq \max (3, k_0)$, it is in fact true that
	\begin{flalign*}
	&& m_k(g) \leq \zeta(g_0) ~~\textrm{and} ~~\left\| \gchi \right\|_{L^2(\OmegaLv)} \leq \zeta(g_0), && \textrm{for all } t \leq T_{\mathcal{T}}.
	\end{flalign*}
	This means that the set $\frac{1}{2} \mathcal{B}_k$ is in fact stable and the argument can be repeated, starting at time $T_{\mathcal{T}}$, to give a unique solution by the contraction mapping theorem up to time $2T_{\mathcal{T}}$.  Therefore, by repeating this argument, there is global existence and uniqueness to the semi-discrete problem (\ref{semi-discrete2}), as required by the statement of Theorem \ref{existence_theorem}.
\end{proof}

\begin{corollary} \label{existence_corollary}
	For the unique solution $g^N = g$ from Theorem \ref{existence_theorem} to the semi-discrete problem (\ref{semi-discrete2}), as the number of Fourier modes $N \to \infty$, $g^N \to \bar{g}$ strongly in $C(0, T, L^2(\OmegaLv))$ where $\bar{g}$ is the solution to
	\begin{equation*}
	\frac{\partial \bar{g}}{\partial t} = Q(\Chi \bar{g}, \Chi \bar{g})(\boldsymbol{v}) - \frac{1}{2}\left(\bar{\gamma}_1 + \sum_{j = 1}^{d} \bar{\gamma}_{j+1} v_j + \bar{\gamma}_{d+2} |\boldsymbol{v}|^2 \right)
	\end{equation*}
	with initial condition $g_0 = f_0$ and the Lagrange multipliers $\bar{\gamma}_j$ are defined by evaluating the expressions (\ref{Lagrange_multipliers}) at $Q(\Chi \bar{g}, \Chi \bar{g})(\boldsymbol{v})$ (instead of $Q_u$), for $j = 1, 2, \ldots, d+2$.
\end{corollary}

\begin{proof}
	First, by explicitly writing the dependence of $g$ on the solution with $N$ Fourier modes as $g^N$, $g^N(t) = g^N_0 + \int_{0}^{t} Q_c(g^N,g^N)(s) ~\textrm{d}s$, this means that, for $N, M \in \mathbb{N}$,
	\begin{equation*}
	g^N(t) - g^M(t) = g^N_0 - g^M_0 + \int_{0}^{t} \left(Q_c(g^N,g^N)(s) - Q_c(g^M,g^M)(s) \right) ~\textrm{d}s
	\end{equation*}
	and so, by using identity (\ref{T_diff}) with $f = g^N$ and $\widetilde{f} = g^M$ and then following through the contraction argument to get to result (\ref{contraction_statement}) without taking the supremum in $s$,
	\begin{align*}
	\left\|g^N(t) - g^M(t) \right\|_{L^2(\OmegaLv)} &\leq \left\|g^N_0 - g^M_0 \right\|_{L^2(\OmegaLv)} \\
	&~~ + \left\| \int_{0}^{t} \left(Q_c(g^N,g^N)(s) - Q_c(g^M,g^M)(s) \right) ~\textrm{d}s \right\|_{L^2(\OmegaLv)} \\
	&\leq \left\|g^N_0 - g^M_0 \right\|_{L^2(\OmegaLv)} + C_{\lambda, \Lv} \int_{0}^{t} \|g^N(s) - g^M(s)\|_{L^2} ~\textrm{d}s.
	\end{align*}
	
	So, by using Gr\"{o}nwall's inequality again, 
	\begin{equation*}
	\left\|g^N(t) - g^M(t) \right\|_{L^2(\OmegaLv)} \leq \left\|g^N_0 - g^M_0 \right\|_{L^2(\OmegaLv)} e^{C_{\lambda, \Lv} t}.
	\end{equation*}
	This means, since $g^N_0 \to g_0$ as $N \to \infty$, $\left\|g^N_0 - g^M_0 \right\|_{L^2(\OmegaLv)} \to 0$ as $N \to \infty$ which means $\left\|g^N(t) - g^M(t) \right\|_{L^2(\OmegaLv)} \to 0$ as $N \to \infty$, because $e^{C_{\lambda, \Lv} t} < \infty$ is constant for fixed $t > 0$.  Therefore the sequence $\{g^N\}$ is a Cauchy sequence which lives in the Banach space $\mathcal{B}_k$ and converges to some function $\bar{g}$ which must solve the semi-discrete problem (\ref{semi-discrete2}) in the limit as $N \to \infty$, as required by the statement of Corollary \ref{existence_corollary}.
\end{proof}

\section{Regularity of the Approximate Solution} \label{Regularity}
Here an estimate will be obtained on the $H^s$-norm of any solution $g$ to the semi-discrete problem (\ref{semi-discrete2}) which satisfies the bounds on the moments and $L^2$-norm given by Theorem \ref{propagation_theorem}, for any $s > 0$.  
\begin{rem}
	The result of the following theorem was actually required to prove the existence of a unique solution in Theorem \ref{existence_theorem}.  For that reason, the following result does not reference Theorem \ref{existence_theorem} directly, but it will indeed apply to the unique solution found there.
\end{rem}
\begin{theorem} \label{regularity_theorem}
	For a solution $g$ of the semi-discrete problem (\ref{semi-discrete2}), with initial condition $g_0$, which satisfies the result of the moment and $L^2$-norm propagation theorem \ref{propagation_theorem}, if the number of Fourier modes $N \geq \widetilde{N}_0$ and domain length $\Lv > \widetilde{L}_0$, for some base number of modes $\widetilde{N}_0 > 0$ and length $\widetilde{L}_0 > 0$ to be defined in the proof, the $H^s$-norm $\left\| g \right\|_{H^s(\OmegaLv)}$ satisfies
	\begin{flalign}
	&&\left\|g \right\|_{H^s(\OmegaLv)} \leq \eta(g_0), ~~~~~~~~~~\textrm{for any } s \geq 0, t \geq 0, \label{regularity_bound}
	\end{flalign}
	\begin{flalign}
	&\textrm{where } & \eta(g_0) &:= \max \left(\left\|g_0 \right\|_{H^s(\OmegaLv)}, 2\frac{C^0_{\lambda,d,\Lv,N}(g_0)}{K_{\lambda, s, \nu}(g_0)}\right), \label{eta_def}\\
	&\textrm{for } & C^0_{\lambda,d,\Lv,N}(g_0) &:= \mathcal{O} \left({\Lv}^{2\lambda + 4} \right) + C_d \left(C_Q \frac{\mathcal{O} \left({\Lv}^{2\lambda + 2} \right)}{N^{\frac{d - 1}{2}}} + O_{\frac{d}{2} + 2} \right)\left(\zeta(g_0) \right)^2 \nonumber \\
	& \textrm{and} & K_{\lambda, s, \nu}(g_0) &:= (s+1)^{1 - \nu} \min_{j \leq s}\{K_{\lambda, j}(g_0) \}, \nonumber
	\end{flalign}
	where the $H^s$-norm is defined by (\ref{Hsk_def}); $O_r$ denotes a constant that is $\mathcal{O}({\Lv}^{-r})$; $\nu > 1$ and $K_{\lambda, j}(g_0)$ are from Proposition \ref{Halpha_estimate_prop}(b)$(ii)$;  $C_d$ is defined in the proof of Theorem \ref{Q_moment_theorem}; $C_Q$ is defined in the proof of Lemma \ref{moment_derivative_lemma}; and $\zeta(g_0)$ is the bound on the moments and $L^2$-norm defined by expression (\ref{zeta_def}) in Theorem \ref{propagation_theorem}.
\end{theorem}
\begin{rem} \label{Hsk_remark}
	It is believed that a similar bound should hold for the weighted $H^s_k$-norm, for $k \geq \lambda + 2$, especially since such a result exists for the true solution to the Fokker-Planck-Landau type equation associated to hard potentials, as given in Proposition \ref{Halpha_estimate_prop}(a)$(ii)$.  In particular, there will exist some $\eta_k(g_0)$ depending on $\lambda$, $d$, $s$ $k$, $\Lv$ and $N$ such that $\left\|g \right\|_{H^s_k(\OmegaLv)} \leq \eta_k(g_0)$.
\end{rem}

\begin{proof}
	For some multi-index $\alpha \in \Nd$, first differentiate the expansion (\ref{Q_expansion}) with respect to $\boldsymbol{v}$ of order $\alpha$ to give
	\begin{multline*}
	\frac{\partial (D^{\alpha} g)}{\partial t} = D^{\alpha} \bigl(Q_c(g,g) - Q_u(g,g) \bigr) + D^{\alpha} Q(\gchi, \gchi) \\
	- D^{\alpha} \Bigr(\left(1 - \Pi^N_{2\Lv}\right) Q(\gchi, \gchi) \Bigr). 
	\end{multline*}
	Then, multiplying this by $D^{\alpha} g$ and integrating with respect to $\boldsymbol{v}$ over $\OmegaLv$ and noting that $D^{\alpha} g \frac{\partial }{\partial t}(D^{\alpha} g) = \frac{1}{2}\frac{\partial }{\partial t}\bigl((D^{\alpha} g)^2 \bigr)$ by the chain rule gives
	\begin{align}
	&\frac{1}{2} \frac{d}{d t} \left(\left\|D^{\alpha} g \right\|^2_{L^2(\OmegaLv)} \right) \ =\  \int_{\OmegaLv} D^{\alpha} \bigl(Q_c(g,g) - Q_u(g,g) \bigr) D^{\alpha} g ~\dv \nonumber\\
	&\quad+ \int_{\OmegaLv} D^{\alpha} Q(\gchi, \gchi) D^{\alpha} g ~\dv + \int_{\OmegaLv} D^{\alpha} \Bigl(\left(1 - \Pi^N_{2\Lv}\right) Q(\gchi, \gchi) \Bigr) D^{\alpha} g ~\dv \nonumber \\
	&\quad\ \leq \int_{\OmegaLv} D^{\alpha} Q(\gchi, \gchi) D^{\alpha} g ~\dv + \biggl(\left\|D^{\alpha} \bigl(Q_u(g,g) - Q_c(g,g) \bigr)\right\|_{L^2(\OmegaLv)} \nonumber \\
	&\qquad \qquad\qquad + \left\|D^{\alpha} \!\Bigl(\left(1\! -\! \Pi^N_{2\Lv} \right) Q(\gchi, \gchi) \Bigr) \right\|_{L^2(\OmegaLv)} \biggr) \!\left\|D^{\alpha} g \right\|_{L^2(\OmegaLv)}, \label{Halpha_estimate1}
	\end{align}
	by the Cauchy-Schwarz inequality.
	
	Now, by the result of Lemma \ref{Isoperimetric_Lemma}, $Q_u(g,g) - Q_c(g,g)$ is simply a second order polynomial in $\boldsymbol{v}$ and so taking derivatives would only reduce the order of the moments involved in the $L^2$-norm bound of the $D^{\alpha}$ derivatives this term.  This means that the result of Theorem \ref{Q_moment_theorem} with $k = 0$ and $k' = 2$ can be still be applied here to give
	\begin{align}
	&\left\|D^{\alpha} \bigl(Q_u(g,g) - Q_c(g,g) \bigr)\right\|_{L^2(\OmegaLv)}   \leq C_d \biggl(\bigl| \bigl| \left(\Pi^N_{2\Lv} - 1 \right) Q(\Chi f, \Chi f) \bigr| \bigr|_{L^2(\OmegaLv)} \nonumber \\
	&~~~~~~~~~~~~~~~~~~~~~~~~~~~~~~~~~~~~~ + O_{\frac{d}{2} + 2} \Bigl(m_{0}(f) m_{2 + \lambda}(f) + m_{\lambda}(f) m_{2}(f) \Bigr)\biggr) \nonumber \\
	&~~~~~\leq C_d \left(C_Q \frac{\mathcal{O} \left({\Lv}^{2\lambda + 2} \right)}{N^{\frac{d - 1}{2}}} \bigl| \bigl| \gchi \bigr| \bigr|^2_{L^2(\OmegaLv)} + O_{\frac{d}{2} + 2} \Bigl(m_{0}(g) m_{2 + \lambda}(g) + m_{\lambda}(g) m_{2}(g) \Bigr) \right)\nonumber \\
		&~~~~~~~~~~~~~~~~~\leq C_d \left(C_Q \frac{\mathcal{O} \left({\Lv}^{2\lambda + 2} \right)}{N^{\frac{d - 1}{2}}} + O_{\frac{d}{2} + 2} \right)\left(\zeta(g_0) \right)^2, \label{Halpha_estimate1a}
	\end{align}
	after using assumption (\ref{Fourier_projection_tail2}); pulling out the maximum value of $(1 + {\Lv^2})^{\lambda + 1}$ given by the weights in the $L^2_{\lambda + 1}$-norm; and using $m_k(g) \leq \zeta(g_0)$ and \linebreak $\left\| \gchi \right\|_{L^2(\OmegaLv)} \leq \zeta(g_0)$ since $g$ satisfies the result of Theorem \ref{propagation_theorem}.
	
	Next, since derivatives $D^{\alpha}$ commute with the mode projection operator $\Pi^N_{2\Lv}$ by the result (\ref{derivative_commuting}), then distributing the derivatives over the arguments of $Q$ using the Leibniz formula,
	\begin{align}
	&\left\|D^{\alpha} \Bigl(\left(1 - \Pi^N_{2\Lv} \right) Q(\gchi, \gchi) \Bigr) \right\|_{L^2(\OmegaLv)} \ =\  \left\|\left(1 - \Pi^N_{2\Lv} \right) D^{\alpha} Q(\gchi, \gchi) \right\|_{L^2(\OmegaLv)} \nonumber \\
	&~~~~~\leq ~C_Q \frac{\mathcal{O} \left({\Lv}^{2\lambda + 2} \right)}{N^{\frac{d - 1}{2}}} \bigl| \bigl|D^{\alpha} \left(\gchi \right)\bigr| \bigr|^2_{L^2(\OmegaLv)} \ \leq \ C_Q \frac{\mathcal{O} \left({\Lv}^{2\lambda + 2} \right)}{N^{\frac{d - 1}{2}}} \bigl| \bigl|\gchi \bigr| \bigr|^2_{H^{|\alpha|}(\OmegaLv)} \nonumber \\
	&~~~~~~~~~~\leq~ C_Q C_{\chi} \frac{\mathcal{O} \left({\Lv}^{2\lambda + 2} \right)}{N^{\frac{d - 1}{2}}} \bigl| \bigl| g \bigr| \bigr|^2_{H^{|\alpha|}(\OmegaLv)}, \label{Halpha_estimate1b}
	\end{align}
	after using the assumption (\ref{Fourier_projection_tail2}) and also (\ref{extension_bound}) to pull the extension operator out of the $H^{|\alpha|}$-norm.
	
	Finally, by Proposition \ref{Halpha_estimate_prop}(b)$(ii)$ with $k = 0$, for some $\nu > 1$,
	\begin{align}
	&\sum_{\substack{\alpha \in \Nd: \\ |\alpha| = s}} \int_{\OmegaLv} D^{\alpha} Q(\gchi, \gchi) D^{\alpha} g ~\dv \nonumber \\
	&~~~~~\leq -K_{\lambda, s}(g_0) \left(\left\| g \right\|^2_{\dot{H}^{s}(\OmegaLv)} \right)^{\nu} + C_{\lambda, s} \left\| g \right\|_{\dot{H}^{s}} \left\| g \right\|^2_{H^{s-1}_{\frac{\lambda}{2} + 1}(\OmegaLv)} \nonumber \\
	&~~~~~~~~~~~\leq -K_{\lambda, s}(g_0) \left(\left\| g \right\|^2_{\dot{H}^{s}(\OmegaLv)} \right)^{\nu} + C_{\lambda, s} \mathcal{O} \left({\Lv}^{\lambda + 2} \right) \left\| g \right\|_{\dot{H}^{s}}, \label{Halpha_estimate1c}
	\end{align}
	after pulling the maximum value of $(1 + {\Lv^2})^{\frac{\lambda}{2} + 1}$ given by the weights in the $H^{s-1}_{\frac{\lambda}{2} + 1}$-norm and using an inductive argument to bound $\left\| g \right\|^2_{H^{|\alpha|-1}}$. 
	\begin{rem}
		Proposition \ref{Halpha_estimate_prop} is in fact an estimate on the integral of $D^{\alpha} Q(g, g) D^{\alpha} g$, without the extension operators $\Chi $.  They can be added without loss of generality, however, because an application of Leibniz' formula and definition of $\Chi $, including the property (\ref{extension_bound}), would lead to the same result.
	\end{rem}
	
	So, by summing up (\ref{Halpha_estimate1}) over all multi-indices $\alpha \in \Nd$ with $|\alpha| = s$ and using linearity of the time derivative, 
	\begin{align*}
	\frac{1}{2} \frac{d}{d t} \left(\left\|g \right\|^2_{\dot{H}^s(\OmegaLv)} \right) =& \sum_{\substack{\alpha \in \Nd: \\ |\alpha| = s}} \frac{1}{2} \frac{d}{d t} \left(\left\|D^{\alpha} g \right\|^2_{L^2(\OmegaLv)} \right)\nonumber \\
	\leq& -K_{\lambda, s}(g_0) \left(\left\| g \right\|^2_{\dot{H}^{s}(\OmegaLv)} \right)^{\nu} + C_{\lambda, s} \mathcal{O} \left({\Lv}^{\lambda + 2} \right) \left\| g \right\|_{\dot{H}^{s}} \nonumber \\
	&+ ~d^s\Biggl(C_d \left(C_Q \frac{\mathcal{O} \left({\Lv}^{2\lambda + 2} \right)}{N^{\frac{d - 1}{2}}} + O_{\frac{d}{2} + 2} \right)\left(\zeta(g_0) \right)^2 \nonumber \\
	&~~~~~~~~~~~~~~~~+ C_Q C_{\chi} \frac{\mathcal{O} \left({\Lv}^{2\lambda + 2} \right)}{N^{\frac{d - 1}{2}}} \bigl| \bigl| g \bigr| \bigr|^2_{H^{s}(\OmegaLv)} \Biggr) \left\| g \right\|_{\dot{H}^{s}}, 
	\end{align*}
	after using the bounds (\ref{Halpha_estimate1a}-\ref{Halpha_estimate1c}) and then noting that $\left\|D^{\alpha} g \right\|_{L^2(\OmegaLv)} \leq \left\|g \right\|_{\dot{H}^s(\OmegaLv)}$ when $|\alpha| = s$, so that $d^s$ many of the same term are added together.  Similarly, after summing up all of the order $j$ derivative bounds, for $j \leq s$, this gives
	\begin{align}
	\frac{1}{2} \frac{d}{d t} \left(\left\|g \right\|^2_{H^s(\OmegaLv)} \right) =& \sum_{j = 0}^s \frac{1}{2} \frac{d}{d t} \left(\left\|g \right\|^2_{\dot{H}^j(\OmegaLv)} \right) \nonumber \\
	\leq& - K_{\lambda, s, \nu}(g_0) \left(\left\| g \right\|^2_{H^s(\OmegaLv)} \right)^{\nu} + \mathcal{O} \left({\Lv}^{2\lambda + 4} \right) \left\|g \right\|_{H^s(\OmegaLv)} \nonumber \\
	&+ ~\Biggl(C_d \left(C_Q \frac{\mathcal{O} \left({\Lv}^{2\lambda + 2} \right)}{N^{\frac{d - 1}{2}}} + O_{\frac{d}{2} + 2} \right)\left(\zeta(g_0) \right)^2 \nonumber \\
	&~~~~+ C_Q C_{\chi} \frac{\mathcal{O} \left({\Lv}^{2\lambda + 2} \right)}{N^{\frac{d - 1}{2}}} \bigl| \bigl| g \bigr| \bigr|^2_{H^s(\OmegaLv)} \Biggr) \left\|g \right\|_{H^s(\OmegaLv)}, \label{Halpha_estimate2}
	\end{align}
	where $K_{\lambda, s, \nu}(g_0) := (s+1)^{1 - \nu} \min_{j \leq s}\{K_{\lambda, j}(g_0) \}$.  Note that the appearance of this constant is a consequence of the fact that, for any $p > 1$ and vector $(z_0, z_1, \ldots z_s) \in \mathbb{R}^{s+1}$,
	\begin{equation*}
	\bigg|\sum_{j = 0}^{s} z_j \bigg|^p \leq (s+1)^{p - 1} \sum_{j = 0}^{s} \left|z_j \right|^p ~~~~~\textrm{and so } ~~~~~ -\sum_{j = 0}^{s} \left|z_j \right|^p \leq - (s+1)^{1 - p}\bigg|\sum_{j = 0}^{s} z_j \bigg|^p.
	\end{equation*}
	This can easily proven using H\"{o}lder's inequality for sums and is then applied, after the minimum value of $K_{\lambda, j}(g_0)$ is pulled out of the sum as a lower bound, with $z_j = \left\|g \right\|^2_{\dot{H}^j(\OmegaLv)}$ and $p = \nu$.
	
	Equivalently, since $\frac{1}{2} \frac{d}{d t} \left(\left\|g \right\|^2_{H^s(\OmegaLv)} \right) = \left\|g \right\|_{H^s(\OmegaLv)} \frac{d}{d t} \left(\left\|g \right\|_{H^s(\OmegaLv)} \right)$ by the chain rule, dividing both sides of the bound (\ref{Halpha_estimate2}) by $\left\|g \right\|_{H^s(\OmegaLv)}$ gives
	\begin{align}
	\frac{d}{d t} \left(\left\|g \right\|_{H^s(\OmegaLv)} \right) \leq& - K_{\lambda, s, \nu}(g_0) \left\| g \right\|^{2\nu - 1}_{H^s(\OmegaLv)} + \mathcal{O} \left({\Lv}^{2\lambda + 4} \right) \nonumber \\
	&+ ~C_d \left(C_Q \frac{\mathcal{O} \left({\Lv}^{2\lambda + 2} \right)}{N^{\frac{d - 1}{2}}} + O_{\frac{d}{2} + 2} \right)\left(\zeta(g_0) \right)^2 \nonumber \\
	&~~~~+ C_Q C_{\chi} \frac{\mathcal{O} \left({\Lv}^{2\lambda + 2} \right)}{N^{\frac{d - 1}{2}}} \bigl| \bigl| g \bigr| \bigr|^2_{H^s(\OmegaLv)}. \label{Halpha_estimate3}
	\end{align}
	
	Now, by defining the variable $Y(t) := \left\|g \right\|_{H^s(\OmegaLv)}$, the bound (\ref{Halpha_estimate3}) can be simply written as the ODI
	\begin{equation}
	\frac{dY}{dt} \leq C^0_{\lambda,d,\Lv,N}(g_0) - K_{\lambda, s, \nu}(g_0) Y^{2\nu - 1} + \widetilde{C}_{\chi} \frac{\mathcal{O} \left({\Lv}^{2\lambda + 2} \right)}{N^{\frac{d - 1}{2}}} Y^2, \label{Halpha_ODI}
	\end{equation}
	\begin{flalign*}
	&\textrm{where } & C^0_{\lambda,d,\Lv,N}(g_0) &:= \mathcal{O} \left({\Lv}^{2\lambda + 4} \right) + C_d \left(C_Q \frac{\mathcal{O} \left({\Lv}^{2\lambda + 2} \right)}{N^{\frac{d - 1}{2}}} + O_{\frac{d}{2} + 2} \right)\left(\zeta(g_0) \right)^2 \\
	& \textrm{and} & \widetilde{C}_{\chi} &:= C_Q C_{\chi}.
	\end{flalign*}
	The ODI (\ref{Halpha_ODI}) can be handled in a similar way to the ODI (\ref{moment_L2_ODI}) on $X(t) = m_k(g) + \left\| \gchi \right\|_{L^2(\OmegaLv)}$ because it takes a similar form but with an extra power $2\nu - 1$ on the negative term.  This only helps to control the derivative further, however, because $2\nu - 1 > 1$ here and so only restricts $Y$ from increasing too much.  First, if $\nu > \frac{3}{2}$ so that $2\nu - 1 > 2$ then, no matter what size the coefficients are on the right-hand side of (\ref{Halpha_ODI}), eventually the negative term will dominate and force $\frac{dY}{dt} < 0$.  This means that, if the initial data $Y(0)$ is below the root of the polynomial on the right-hand side of (\ref{Halpha_ODI}), say $Y^-$, then $Y$ remains bounded by $Y^-$.  If $Y(0) > Y^-$, however, then $\frac{dY}{dt} < 0$ from the start and so $Y$ cannot increase past $Y(0)$.  This means that $\left\|g \right\|_{H^s(\OmegaLv)} = Y(t) \leq \max \left(Y(0), Y^-\right)$ when $\nu > \frac{3}{2}$.
	
	On the other hand, for any $\nu > 1$, it is always true that
	\begin{flalign}
	~~~&& \frac{dY}{dt} \leq&~ C^0_{\lambda,d,\Lv,N}(g_0) - K_{\lambda, s, \nu}(g_0) Y^{2\nu - 1} + \widetilde{C}_{\chi} \frac{\mathcal{O} \left({\Lv}^{2\lambda + 2} \right)}{N^{\frac{d - 1}{2}}} Y^2 && \nonumber \\
	&& \leq&~ C^0_{\lambda,d,\Lv,N}(g_0) - K_{\lambda, s, \nu}(g_0) Y + \widetilde{C}_{\chi} \frac{\mathcal{O} \left({\Lv}^{2\lambda + 2} \right)}{N^{\frac{d - 1}{2}}} Y^2, &&\textrm{when } Y \geq 1. \label{Halpha_ODI_variant}
	\end{flalign}
	This means that, for $1 < \nu < \frac{3}{2}$, the same argument can be made as for the ODI (\ref{moment_L2_ODI}) on $X$.  In particular, $N$ can be chosen large enough for this quadratic polynomial to have two distinct and real roots.  Then, if $\Lv$ is large enough to cause $\widetilde{C}_{\chi} \mathcal{O} \left({\Lv}^{2\lambda + 2} \right) \geq 1$ then the lower root of the quadratic polynomial on the right-hand side of (\ref{Halpha_ODI_variant}), say $\widetilde{Y}^-$, satisfies $\frac{C^0_{\lambda,d,\Lv,N}(g_0)}{K_{\lambda, s, \nu}(g_0)} < \widetilde{Y}^- < 2\frac{C^0_{\lambda,d,\Lv,N}(g_0)}{K_{\lambda, s, \nu}(g_0)}$.  Given some initial condition $Y(0) = \left\|g_0 \right\|_{H^s(\OmegaLv)}$, $N$ can also be chosen larger than some base number of modes $\widetilde{N}_0$ so that the larger root of the quadratic polynomial on the right-hand side of (\ref{Halpha_ODI_variant}), say $\widetilde{Y}^+$, satisfies $\tilde{Y}^+ > Y(0)$.  
	
	Further, assuming  $\Lv > \widetilde{L}_0$, for $\widetilde{L}_0$ large enough to have $\frac{C^0_{\lambda,d,\Lv,N}(g_0)}{K_{\lambda, s, \nu}(g_0)} \geq 1$, then this means that the roots of the polynomial on the right-hand side of (\ref{Halpha_ODI}), say $Y^-$ and $Y^+$ satisfy $Y^- < \widetilde{Y}^- < 2\frac{C^0_{\lambda,d,\Lv,N}(g_0)}{K_{\lambda, s, \nu}(g_0)}$ and $Y^+ > \widetilde{Y}^+ > Y(0)$, by the second inequality in (\ref{Halpha_ODI_variant}).  This means that the exact same conclusion can be made as for the ODI (\ref{moment_L2_ODI}) and once again $\|g \|_{H^s(\OmegaLv)} = Y(t) \leq \max\big\{Y(0), 2\frac{C^0_{\lambda,d,\Lv,N}(g_0)}{K_{\lambda, s, \nu}(g_0)}\big\}$, when $1 < \nu < \frac{3}{2}$.  A sketch of this argument is shown in Fig.\ \ref{ODIDiagram_Halpha}(a) where, given the initial condition $Y(0) \leq Y^+$, the hatched region shows the possible values for the derivative $Y'(t)$.
	
	Note that for the case when $\nu > \frac{3}{2}$, and so $2\nu-1 > 2$, if it is still assumed that $N > \widetilde{N}_0$ and $\Lv > \widetilde{L}_0$ then the only root of the polynomial on the right-hand side of (\ref{Halpha_ODI}) once again satisfies $Y^- < \widetilde{Y}^- < 2\frac{C^0_{\lambda,d,\Lv,N}(g_0)}{K_{\lambda, s, \nu}(g_0)}$.  In this case, the number of Fourier modes $N$ does not need to be chosen quite as high because there is no second root and so the initial condition can be much larger than $Y^-$ and it will always be true that $Y'(t) < 0$.  A sketch of this can be seen in Fig.\ \ref{ODIDiagram_Halpha}(b) where the hatched region shows the possible values for the derivative $Y'(t)$ given any initial $Y(0)$.
	
	Therefore, no matter the size of $\nu$, if $\Lv > \widetilde{L}_0$ and $N > \widetilde{N}_0$ then, since $Y(0) = \left\|g_0 \right\|_{H^s(\OmegaLv)}$,
	\begin{equation*}
	\left\|g \right\|_{H^s(\OmegaLv)} = Y(t) \leq \max \left(\left\|g_0 \right\|_{H^s(\OmegaLv)}, 2\frac{C^0_{\lambda,d,\Lv,N}(g_0)}{K_{\lambda, s, \nu}(g_0)}\right),
	\end{equation*}
	which is the required result (\ref{regularity_bound}) coupled to (\ref{eta_def}) in Theorem \ref{regularity_theorem}.
	\begin{figure}[!hbt]
		\centering
		\begin{tabular}{lcr}
			\vspace*{-0.5cm}
			(a) && \\
			&\includegraphics[width=0.85\linewidth]{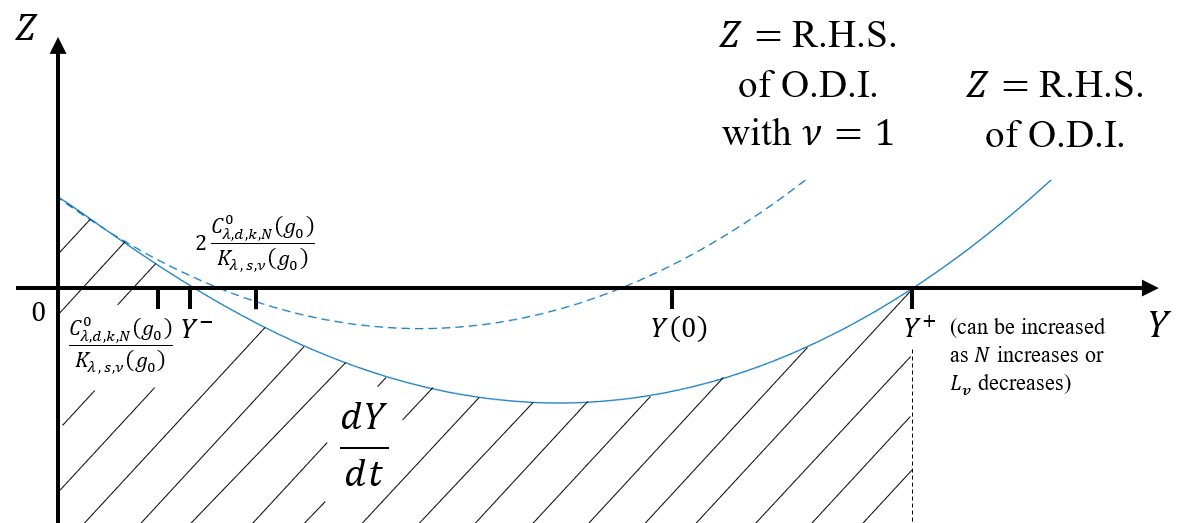}& \\ [6pt]
			\vspace*{-0.5cm}
			(b) && \\
			&\includegraphics[width=0.85\linewidth]{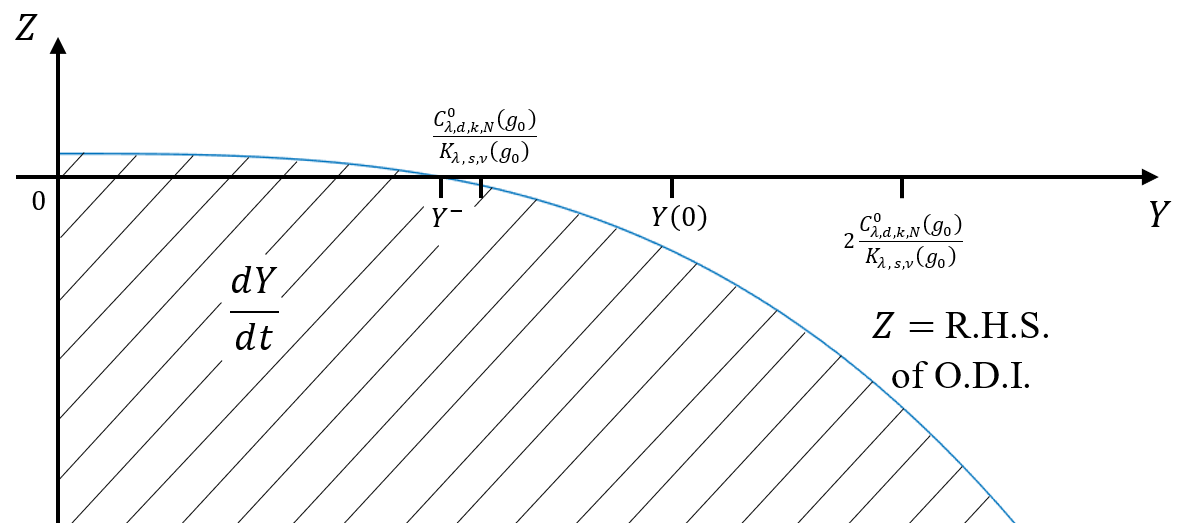}& \\
		\end{tabular}
		\caption{Sketches of the polynomial on the right-hand side of the ODI (\ref{Halpha_ODI}) for to the two possible cases (a) $1 < \nu \leq \frac{3}{2}$ and (b) $\nu > 2$, with a large enough number of Fourier modes $N$ to give two distinct real roots in case (a).  Since the derivative of $Y(t) = \left\|g \right\|_{H^s(\OmegaLv)}$ is bounded above by this curve, the derivative remains in the hatched region.  For this to be true in case (a) the initial condition must satisfy $Y(0) < Y^+$.}
		\label{ODIDiagram_Halpha}
	\end{figure}
\end{proof}

\section{The $L^2$-norm Error Estimate} \label{L2_Estimate}
It will now be shown that the unique solution $g$ from Theorem \ref{existence_theorem} does indeed converge to the true solution $f$ of the space-homogeneous Fokker-Planck-Landau type equation (\ref{Landau_homo_Kn1}) by devising an error estimate comparing $g$ and $f$ in $L^2$-norm, up to some fixed time $T > 0$.  
\begin{theorem} \label{L2_error_estimate_theorem}
	Given some initial condition $f_0 \in L^1_k \bigcap L^2(\OmegaLv)$, where $k \geq \max (3, k_0)$, if the number of Fourier modes $N \geq N_0$ and $\Lv \geq L_0$, the difference in $L^2$-norm between the solution $f$ to the space-homogeneous Fokker-Planck-Landau type equation (\ref{Landau_homo_Kn1}) associated to hard potentials with $f(0) = f_0$ and the solution $g$ to the semi-discrete problem (\ref{semi-discrete2}) with initial condition $g_0 = \Pi^N_{2\Lv} f_0$ which is assumed to satisfy the stability condition (\ref{stability_init}) with $\varepsilon \leq \min \left(\frac{1}{4}, \varepsilon_0 \right)$, each solved for $0 < t \leq T$, satisfies
	\begin{multline}
	\sup_{t \in [0, T]} \left\|f - g \right\|_{L^2(\OmegaLv)} \leq e^{C_H \bigl(C^f_{\lambda,2,\lambda + 2}(f_0) + \eta(g_0) \bigr)T} \Biggl(\left\|f_0 - g_0 \right\|_{L^2(\OmegaLv)} + O_{\frac{d}{2} + 2} \\
	+ \frac{\mathcal{O} \left({\Lv}^{2\lambda + 2} \right)}{N^{\frac{d - 1}{2}}} + \frac{1}{2} \widetilde{\zeta}(f_0 - g_0) \Biggr), \label{L2_error_estimate}
	\end{multline}
	where $C_H$ is the constant from Proposition \ref{LBH_prop}; $C^f_{\lambda,2,\lambda+2}(f_0)$ is the constant $C^f_{\lambda,s,k}(f_0)$ from Proposition \ref{Halpha_estimate_prop}(a)$(ii)$ with $s = 2$ and $k = \lambda + 2$; $\eta(g_0)$ is defined by expression (\ref{eta_def}) in Theorem \ref{regularity_theorem}; $\widetilde{\zeta}(f_0 - g_0)$ is defined by expression (\ref{zeta_def2}) in Theorem \ref{propagation_theorem}; $\varepsilon_0$ and $k_0$ are defined in the proof of Lemma \ref{moment_derivative_lemma} in expressions (\ref{epsilon_0}) and (\ref{k_0}), respectively; $\zeta(g_0)$ defined by expression (\ref{zeta_def}) and $N_0$ described in Theorem \ref{propagation_theorem}; and $L_0 > 0$ described in the current proof of Theorem \ref{existence_theorem}.
\end{theorem}

Before proving this theorem, it should be noted that the term in parenthesis in the error estimate (\ref{L2_error_estimate}) can be made as small as necessary by choosing large enough $N$ and $\Lv$, as well as $f_0$ and $g_0$ sufficiently close.  This means that, since the exponential in front is a constant for fixed $T > 0$, the approximation from the semi-discrete problem (\ref{semi-discrete2}) does indeed converge to the true solution.  

More precisely, since $g_0 = \Pi^N_{\Lv}f_0$, if $f_0$ is sufficiently smooth then choosing $N$ large enough ensures that $\|f_0 - g_0\|_{L^2(\Lv)}$ is small by the Fourier approximation result given in Lemma \ref{Fourier_approx_lemma} in the appendix.  Consequently, since $\|f_0 - g_0\|_{L^1(\Lv)} \leq (2\Lv)^{\frac{d}{2}}\|f_0 - g_0\|_{L^2(\Lv)}$ by the Cauchy-Schwarz inequality, this means that $g_0$ also converges to $f_0$ in $L^1$-norm as $N \to \infty$.

Then, it is also true that $\widetilde{\zeta}(h) \to 0$ as its argument $h = f_0 - g_0$ converges to zero in both $L^1$ and $L^2$-norm.  Clearly this is true for the first term inside the maximum function in the definition (\ref{zeta_def2}).  Then, for the second term inside the maximum function, by using the definition of $A^1_{d,k}(h)$ from the ODI (\ref{moment_L2_ODI}),
\begin{align*}
2 \frac{A^1_{d,k}(h) N^{\frac{d - 1}{2}}}{A^2_{d,k} \mathcal{O} \left({\Lv}^{\kappa} \right)} &= 2 \frac{\min \left(\frac{1}{2} {\varepsilon_{\chi}}^2 K_{\lambda, k} m_{0}(h), \frac{K^S_{d, \lambda}(h)}{{\Lv}^{2}} \right) N^{\frac{d - 1}{2}}}{A^2_{d,k} \mathcal{O} \left({\Lv}^{\kappa} \right)} \\
& = 2 \frac{\min \left(\frac{1}{2} {\varepsilon_{\chi}}^2 K_{\lambda, k}, \frac{\widetilde{K_{\lambda}}}{\left(2 C^S_d \Lv \right)^{2}} \right) N^{\frac{d - 1}{2}}}{A^2_{d,k} \mathcal{O} \left({\Lv}^{\kappa} \right)} m_0(h),
\end{align*}
after using expression (\ref{K_dlambda_def}) for $K^S_{d, \lambda}(h)$ and factoring out the $m_0(h)$ term which appears in both arguments of the minimum.  This means that this term also decreases as $m_0(f_0 - g_0)$ does, or equivalently $\|f_0 - g_0\|_{L^1(\Lv)}$ since $f_0$ and $g_0$ are assumed to have compact support.  Note that the Fourier approximation lemma \ref{Fourier_approx_lemma} also takes care of the $N$ appearing in the numerator here if it is assumed that $f_0 \in H^s(\OmegaLv)$ for $s > \frac{d - 1}{2}$.  In particular, this means that there is some constant $C_{\zeta} > 0$ such that
\begin{equation}
\widetilde{\zeta}(h) \leq C_{\zeta} \|h\|_{L^2(\OmegaLv)}. \label{C_zeta_def}
\end{equation}

\begin{proof}
	First note that the assumptions on $k$, $\varepsilon$, $\Lv$ and $N$ ensure the existence of a unique solution $g$ to the semi-discrete problem (\ref{semi-discrete2}) which satisfies the regularity properties from Theorem \ref{regularity_theorem}.  For such a function $g$, the main trick here is to write the collision operator $Q$ as 
	\begin{flalign}
	&& Q(g,g) &= Q(\gchi + (1 - \Chi )g, \gchi + (1 - \Chi )g) = Q(\gchi, \gchi) + E(g,g), && \nonumber \\
	\textrm{for} && E(g,g) &:= Q(\gchi, (1 - \Chi )g) + Q((1 - \Chi )g, \gchi) + Q((1 - \Chi )g, (1 - \Chi )g). && \label{E_Q}
	\end{flalign}
	
	Then, by definition (\ref{Q_u}) of the unconserved operator $Q_u$ in (\ref{Q_u}) and since $Q(\gchi, \gchi) = Q(g,g) - E(g,g)$,
	\begin{align}
	Q_u(g,g) &= \Pi^N_{2\Lv} Q(\gchi, \gchi). \nonumber \\
	&= Q(\gchi, \gchi) - \left(1 - \Pi^N_{2\Lv} \right) Q(\gchi, \gchi) \nonumber \\
	&= Q(g, g) - E(g,g) - \left(1 - \Pi^N_{2\Lv} \right) Q(\gchi, \gchi) \label{Q_u_expansion}. 
	\end{align}
	
	Now, by subtracting the semi-discrete equation (\ref{semi-discrete2}) (for the approximation $g$) from the homogeneous Fokker-Planck-Landau equation (\ref{Landau_homo_Kn1}) (for the true solution $f$), then adding and subtracting the unconserved operator $Q_u$,
	\begin{align*}
	\frac{\partial }{\partial t} (f - g) = Q(f,f) - Q_c(g,g) &= Q(f,f) - Q_u(g,g) + Q_u(g,g) - Q_c(g,g) \\
	& = Q(f,f) - Q(g, g) + Q_u(g,g) - Q_c(g,g) \\
	&~~~~~~~~~~~+ E(g,g) + \left(1 - \Pi^N_{2\Lv} \right) Q(\gchi, \gchi), 
	\end{align*}
	after using identity (\ref{Q_u_expansion}) to replace the first instance of $Q_u$ here.  Then, multiplying this by $(f - g)$, integrating with respect to $\boldsymbol{v}$ over $\OmegaLv$, and noting that the chain rule implies $(f - g)\frac{\partial }{\partial t}(f - g) = \frac{1}{2}\frac{\partial }{\partial t}((f - g)^2)$, gives
	\begin{equation}
	\frac{1}{2} \frac{d}{d t} \left(\left\|f - g \right\|^2_{L^2(\OmegaLv)} \right) = I_1 + I_2 + I_3, \label{L2_estimate1}
	\end{equation}
	\begin{flalign*}
	\textrm{where } && I_1 &:= \int_{\OmegaLv} (f - g) \left(Q_u(g,g) - Q_c(g,g) + \left(1 - \Pi^N_{2\Lv} \right) Q(\gchi, \gchi) \right)~\dv, && \\
	&& I_2 &:= \int_{\OmegaLv} (f - g)E(g,g) ~\dv && \\
	\textrm{and } && I_3 &:= \int_{\OmegaLv} (f - g)\left(Q(f,f) - Q(g, g) \right)  ~\dv. &&
	\end{flalign*}
	
	First, by the Cauchy-Schwarz inequality,
	\begin{multline*}
	I_1 \leq \biggl(\left\|Q_u(g,g) - Q_c(g,g) \right\|_{L^2(\OmegaLv)} \\
	+ \left\|\left(1 - \Pi^N_{2\Lv} \right) Q(\gchi, \gchi) \right\|_{L^2(\OmegaLv)} \biggr) \left\|f - g \right\|_{L^2(\OmegaLv)},
	\end{multline*}
	where, by the result from Theorem \ref{Q_moment_theorem} with $k = 0$ and $k' = 2$,
	\begin{align}
	&\left\|Q_u(g,g) - Q_c(g,g) \right\|_{L^2(\OmegaLv)} + \left\|\left(1 - \Pi^N_{2\Lv} \right) Q(\gchi, \gchi) \right\|_{L^2(\OmegaLv)} \nonumber \\
	\leq&~ O_{\frac{d}{2} + 2} \Bigl(m_{0}(g) m_{2 + \lambda}(g) + m_{\lambda}(g) m_{2}(g) \Bigr) \nonumber \\
	&~~~~~~~~~~~~~~~~+ (C_d + 1)\left\|\left(1 - \Pi^N_{2\Lv} \right) Q(\gchi, \gchi) \right\|_{L^2(\OmegaLv)} \nonumber \\
	\leq&~ O_{\frac{d}{2} + 2} \Bigl(m_{0}(g) m_{2 + \lambda}(g) + m_{\lambda}(g) m_{2}(g) \Bigr) + C_Q (C_d + 1)\frac{\mathcal{O} \left({\Lv}^{2\lambda + 2} \right)}{N^{\frac{d - 1}{2}}} \bigl| \bigl| \gchi \bigr| \bigr|^2_{L^2(\OmegaLv)} \nonumber \\
	\leq& \left(O_{\frac{d}{2} + 2} + (C_d + 1)\frac{C_Q \mathcal{O} \left({\Lv}^{2\lambda + 2} \right)}{N^{\frac{d - 1}{2}}} \right)\left(\zeta(g_0) \right)^2, \label{I1_bound_a}
	\end{align}
	after using assumption (\ref{Fourier_projection_tail2}); pulling out the largest value of the weight $\langle \boldsymbol{v} \rangle^{2(\lambda + 1)}$ from the $L^2_{\lambda + 1}$-norm; and using $m_k(g) \leq \zeta(g_0)$ and $\left\| \gchi \right\|_{L^2(\OmegaLv)} \leq \zeta(g_0)$, for $\zeta(g_0)$ defined by (\ref{zeta_def}), since $g$ is the unique solution from Theorem \ref{existence_theorem}.  This means
	\begin{equation}
	I_1 \leq \left(O_{\frac{d}{2} + 2} + (C_d + 1)\frac{C_Q \mathcal{O} \left({\Lv}^{2\lambda + 2} \right)}{N^{\frac{d - 1}{2}}} \right)\left(\zeta(g_0) \right)^2 \left\|f - g \right\|_{L^2(\OmegaLv)}. \label{I1_bound}
	\end{equation}
	
	Next, for a bound on $I_2$, note that $E(g,g)$ contains evaluations of $Q$ of the form $Q(\widetilde{\Chi }_1 g, \widetilde{\Chi }_2 g)$, where $\widetilde{\Chi }_i$ could be $\Chi $ or $1 - \Chi $, for $i = 1,2$.  Then, by the result (\ref{LBH_bound}) in Proposition \ref{LBH_prop} with $F = \widetilde{\Chi }_1 g$, $G = \widetilde{\Chi }_2 g$ and $H = f - g$,
	\begin{align*}
	&\left|\int_{\OmegaLv} (f - g) Q(\widetilde{\Chi }_1 g, \widetilde{\Chi }_2 g) ~\dv \right| \\
	\leq&~ C_H \Bigl(\|\widetilde{\Chi }_1 g\|_{L^1_{\lambda + 2}(\OmegaLv)} + \|\widetilde{\Chi }_1 g\|_{L^2(\OmegaLv)} \Bigr) \|\widetilde{\Chi }_2 g\|_{H^2_{\lambda + 2}(\OmegaLv)} \|f - g\|_{L^2(\OmegaLv)}.
	\end{align*}
	Now, in expression (\ref{E_Q}) for $E(g,g)$, each term has one of $\widetilde{\Chi }_1$ or $\widetilde{\Chi }_2$ replaced by $1 - \Chi $.  Since $\left(1 - \Chi \right)(\boldsymbol{v}) = 0$ when $\boldsymbol{v} \in \Omega_{(1 - \delchi)\Lv}$, this means that one of the factors here involving norms of $\widetilde{\Chi }_i g$, for $i = 1,2$, is an $\mathcal{O} \left(\delchi \right)$ term, while the other remains bounded.  More precisely, after pulling out the maximum value of $(1 + {\Lv^2})^{\frac{\lambda}{2} + 1}$ given by the weights and denoting the norms of $(1 - \Chi )g$ by an $\mathcal{O} \left(\delchi \right)$ term,
	\begin{equation}
	I_2 \leq C_H \max(\zeta(g_0), \eta(g_0)) \left((1 + {\Lv^2})^{\lambda + 2} + (1 + {\Lv^2})^{\frac{\lambda}{2} + 1} \right) \mathcal{O} \left(\delchi \right) \|f - g\|_{L^2(\OmegaLv)}, \label{I2_bound}
	\end{equation}
	where the quantities $\zeta(g_0)$ defined by (\ref{zeta_def}) and $\eta(g_0)$ defined by (\ref{eta_def}) have been used to bound the $L^1_k$- and $H^s$-norms of $\gchi$, respectively, after using (\ref{extension_bound}) to remove the extension operator \mbox{\Large$\chi$}.

	Finally, for the integral $I_3$ in (\ref{L2_estimate1}), first note that by the bi-linearity and symmetry of $Q$,
	\begin{equation*}
	I_3 = \int_{\OmegaLv} (f - g) Q(f - g, f + g) ~\dv.
	\end{equation*}
	This means that the result (\ref{LBH_bound}) in Proposition \ref{LBH_prop} can once again be used, this time with the choices of $F = f - g$, $G = f + g$ and $H = f - g$, to give
	\begin{align}
	I_3 \leq&~ C_H \left(\|f - g\|_{L^1_{\lambda + 2}(\OmegaLv)} + \|f - g\|_{L^2(\OmegaLv)}\right)  \|f + g\|_{H^2_{\lambda + 2}(\OmegaLv)} \|f - g\|_{L^2(\OmegaLv)} \nonumber \\
	\leq&~ C_H (C^f_{\lambda,2,\lambda + 2}(f_0) + \eta(g_0)) \widetilde{\zeta}(f_0 - g_0) \|f - g\|_{L^2(\OmegaLv)} \nonumber \\
	&~~~~~ + C_H \bigl(C^f_{\lambda,2,\lambda + 2}(f_0) + \eta(g_0) \bigr) \|f - g\|^2_{L^2(\OmegaLv)}, \label{I3_bound}
	\end{align}
	after using the triangle inequality on the $H^2_{\lambda + 2}$-norm; $C^f_{\lambda,s,k}(f_0)$ from Proposition \ref{Halpha_estimate_prop} with $s = 2$ and $k = \lambda + 2$ to bound $\|f\|_{H^2_{\lambda + 2}(\OmegaLv)}$; $\eta(g_0)$ defined by (\ref{eta_def}) in Theorem \ref{regularity_theorem} to bound $\|g\|_{H^2_{\lambda + 2}}$; and $\|f - g\|_{L^1_{\lambda + 2}} \leq \|f - g\|_{L^1_3} \leq \widetilde{\zeta}(f_0 - g_0)$, for $\widetilde{\zeta}$ defined by expression (\ref{zeta_def2}) in Theorem \ref{propagation_theorem}.  Note that here the weight is bounded as $\lambda + 2 \leq 3$, since $\lambda \leq 1$, to allow Theorem \ref{propagation_theorem} to be used because it requires $k \geq 3$.
	\begin{rem}
		To reach the bound (\ref{I3_bound}), an estimate was required on the weighted norm $\|g\|_{H^2_{\lambda + 2}}$.  Whereas this result hasn't explicitly been proven in the current work, it is expected to be true, as mentioned in Remark \ref{Hsk_remark}.  
	\end{rem}
	\begin{rem}
		$\widetilde{\zeta}$ is used here, instead of the bound with $\zeta$ defined by (\ref{zeta_def}), because it is easier to show that $\widetilde{\zeta}(h_0)$ approaches zero as the argument $h_0 := f_0 - g_0$ decreases in $L^1$-norm.  This will be explained at the end of the current proof.
	\end{rem} 
	
	Then, using the bounds (\ref{I1_bound}-\ref{I3_bound}) in (\ref{L2_estimate1}) gives
	\begin{align}
	&\frac{1}{2} \frac{d}{d t} \left(\left\|f - g \right\|^2_{L^2(\OmegaLv)} \right)\nonumber \\
	\leq&~ C_H \bigl(C^f_{\lambda,2,\lambda + 2}(f_0) + \eta(g_0) \bigr) \|f - g\|^2_{L^2} \nonumber \\
	&+ C_H \max(\zeta(g_0), \eta(g_0)) \left((1 + {\Lv^2})^{\lambda + 2} + (1 + {\Lv^2})^{\frac{\lambda}{2} + 1} \right) \mathcal{O} \left(\delchi \right) \|f - g\|_{L^2} \nonumber \\
	&+ \Biggl(C_H \bigl(C^f_{\lambda,2,\lambda + 2}(f_0) + \eta(g_0) \bigr) \widetilde{\zeta}(f_0 - g_0) \nonumber \\
	&~~~~~~~~~~~~ + \left(O_{\frac{d}{2} + 2} + (C_d + 1)\frac{C_Q \mathcal{O} \left({\Lv}^{2\lambda + 2} \right)}{N^{\frac{d - 1}{2}}} \right)\left(\zeta(g_0) \right)^2 \Biggr) \left\|f - g \right\|_{L^2(\OmegaLv)}. \label{L2_estimate2}
	\end{align}
	At this stage it should be noted that there are no more $\gchi$ terms and so the limit $\delchi \to 0$ can be taken, independently of $\Lv$, to remove any smoothing of the extension operator.  Taking this limit in the bound (\ref{L2_estimate2}) and denoting $Z := \left\|f - g \right\|^2_{L^2(\OmegaLv)}$ gives
	\begin{align}
	\frac{d Z}{d t} \leq&~ 2 C_H \bigl(C^f_{\lambda,2,\lambda + 2}(f_0) + \eta(g_0) \bigr) Z \nonumber \\
	&+ \Biggl(C_H \bigl(C^f_{\lambda,2,\lambda + 2}(f_0) + \eta(g_0) \bigr) \widetilde{\zeta}(f_0 - g_0) \nonumber \\
	&~~~~~ + \left(O_{\frac{d}{2} + 2} + (C_d + 1)\frac{C_Q \mathcal{O} \left({\Lv}^{2\lambda + 2} \right)}{N^{\frac{d - 1}{2}}} \right)\left(\zeta(g_0) \right)^2 \Biggr) \sqrt{Z}. \label{ODI_L2estimate}
	\end{align}
	To recover a bound on $Z$, using the Gronwall type result from Lemma \ref{Gronwall_lemma} in the appendix with the choices $u = Z, \alpha = \frac{1}{2}$, and the appropriate coefficients from the ODI (\ref{ODI_L2estimate}), gives
	\begin{multline*}
	Z(t) \leq e^{2 C_H \bigl(C^f_{\lambda,2,\lambda + 2}(f_0) + \eta(g_0) \bigr)t} \Biggl(\bigl(Z(0) \bigr)^{\frac{1}{2}} + O_{\frac{d}{2} + 2} + \frac{\mathcal{O} \left({\Lv}^{2\lambda + 2} \right)}{N^{\frac{d - 1}{2}}} \\
	+ \frac{1}{2} \widetilde{\zeta}(f_0 - g_0) \Biggr)^{2},
	\end{multline*}
	where any constants multiplying or dividing the $O_{\frac{d}{2} + 2}$ or $\mathcal{O} \left({\Lv}^{2\lambda + 2} \right)$ terms have been absorbed and the division of the coefficient of $t$ in the exponent cancels nicely with the factor of $\widetilde{\zeta}$. 
	
	Therefore, since $Z = \left\|f - g \right\|^2_{L^2(\OmegaLv)}$, after taking a square root, this means that
	\begin{multline*}
	\sup_{t \in [0, T]} \left\|f - g \right\|_{L^2(\OmegaLv)} \leq e^{C_H \bigl(C^f_{\lambda,2,\lambda + 2}(f_0) + \eta(g_0) \bigr)T} \Biggl(\left\|f_0 - g_0 \right\|_{L^2(\OmegaLv)} + O_{\frac{d}{2} + 2} \\
	+ \frac{\mathcal{O} \left({\Lv}^{2\lambda + 2} \right)}{N^{\frac{d - 1}{2}}} + \frac{1}{2} \widetilde{\zeta}(f_0 - g_0) \Biggr),
	\end{multline*}
	which is the required error estimate (\ref{L2_error_estimate}) stated in Theorem \ref{L2_error_estimate_theorem}.
\end{proof}

\section{Long Time Behaviour} \label{Long_Time_Behaviour}
Now that the approximation $g$ from the semi-discrete problem (\ref{semi-discrete2}) has been shown to converge to the true solution of the Fokker-Planck-Landau type equation (\ref{Landau_homo_Kn1}), the final step is to show that $g$ converges to the equilibrium Maxwellian $\Meq$ defined by expression (\ref{M_eq_hom}).  This will be achieved by setting $g = \Meq + h$, for the perturbation $h(t, \boldsymbol{v}) := g(t, \boldsymbol{v}) - \Meq(\boldsymbol{v})$.  The aim is then to show that the $L^2$-norm of $h$ can be made as small as needed, provided that the initial perturbation $h_0(\boldsymbol{v}) := h(0, \boldsymbol{v}) = g_0(\boldsymbol{v}) - \Meq(\boldsymbol{v})$ is small enough.  This is made more precise in the following result.

\begin{lemma} \label{long_time_lemma}
	For $k \geq \max (3, k_0)$, let $g_0 = \Pi^N_{2\Lv} f_0 \in L^1_k \bigcap L^2(\OmegaLv)$ be an initial condition which satisfies the stability condition (\ref{stability_init}) with $\varepsilon \leq \min \left(\frac{1}{4}, \varepsilon_0 \right)$ and $\Meq$ be the equilibrium Maxwellian associated to $g_0$, where $\varepsilon_0$ and $k_0$ are defined in the proof of Lemma \ref{moment_derivative_lemma} in expressions (\ref{epsilon_0}) and (\ref{k_0}), respectively.  Given any $\delta > 0$, if
	\begin{flalign}
	&& &\|g_0 - \Meq\|_{L^2(\OmegaLv)} \leq \min \left(\frac{1}{2 C_{\zeta}} \delta, \frac{1}{8}\delta \right) && \label{long_time_assumption1} \\
	\textrm{and } && &~~~~~~~\|g_0 - \Meq\|_{H^2_{\lambda + 2}(\OmegaLv)} \leq \delta_{\eta} \label{long_time_assumption2}, &&
	\end{flalign}
	where $\delta_{\eta} > 0$ is chosen such that $C_H C_{\chi} \eta(g_0 - \Meq) \leq \frac{1}{2} \lambda_0$, for the spectral gap $\lambda_0$, and the number of Fourier modes $N$ and domain size $\Lv$ are sufficiently large, then
	\begin{equation*}
	\|g(t) - \Meq\|_{L^2(\OmegaLv)} \leq \delta, ~~~~~\textrm{for all } t>0,
	\end{equation*}
	where $C_H$ is the constant from Proposition \ref{LBH_prop}; $C_{\chi}$ is the constant associated to the extension operator from the bound (\ref{extension_bound}); $\eta$ is defined by expression (\ref{eta_def}) in Theorem \ref{regularity_theorem}; and $C_{\zeta}$ is a constant defined in (\ref{C_zeta_def}), associated to $\zeta$ given by (\ref{zeta_def}) in Theorem \ref{propagation_theorem}.
\end{lemma}

\begin{proof}
	First, as was the case for the semi-discrete scheme associated to the Boltzmann equation in \cite{BoltzmannConvergence}, the conserved operator $Q_c$ associated to the semi-discrete scheme \ref{semi-discrete2} satisfies
	\begin{equation}
	Q_{c}(\Meq + h, \Meq + h) = Q_{c}(\Meq,\Meq) + Q_{c}(\Meq,h) + Q_{c}(h,\Meq) + Q_{c}(h,h). \label{Q_c_linearised}
	\end{equation}
	Next, define the linearised operators
	\begin{flalign}
	&& \mathcal{L}_{c}(h) &:= Q_{c}(\Meq,h) + Q_{c}(h,\Meq) \label{L_c_def}&& \\
	\textrm{and } && \mathcal{L}(\hchi) &:= Q(\Meq,\hchi) + Q(\hchi,\Meq). \label{L_def} &&
	\end{flalign}
	
	Then, by noting that $\frac{\textrm{d}}{\textrm{d}t}(\Meq) = 0$, \mbox{\Large$\chi$} is independent of $t$, and $\Meq + h = g$ is the solution to the semi-discrete problem (\ref{semi-discrete2}),
	\begin{equation*}
	\frac{\textrm{d}}{\textrm{d}t} \left(\hchi \right) = \Chi \frac{\textrm{d}}{\textrm{d}t} \left(\Meq + h \right) = \Chi Q_{c}(\Meq + h, \Meq + h). 
	\end{equation*}
	So, by inserting the expansion (\ref{Q_c_linearised}) with the notation (\ref{L_c_def}),
	\begin{align}
	\frac{\textrm{d}}{\textrm{d}t} \left(\hchi \right) &= \Chi Q_{c}(\Meq,\Meq) + \Chi \mathcal{L}_{c}(h) + \Chi Q_{c}(h,h) \nonumber \\
	&= \mathcal{L}(\hchi) + Q(\hchi, \hchi) + \mathcal{R}(h), \label{linear_Landau_eq}
	\end{align}
	\begin{flalign}
	\textrm{where } && \mathcal{R}(h) := \Chi Q_{c}(\Meq,\Meq) + \Chi \mathcal{L}_{c}(h) - \mathcal{L}(\hchi) + \Chi Q_{c}(h,h) - Q(\hchi, \hchi), \label{Remainder}
	\end{flalign}
	which is obtained by adding and subtracting $\mathcal{L}(\hchi) + Q(\hchi, \hchi)$, for $\mathcal{L}$ defined by (\ref{L_def}).
	
	Next, by adding and subtracting $Q_{u}(\Meq,\Meq)$,
	\begin{align*}
	Q_{c}(\Meq,\Meq) &= Q_{c}(\Meq,\Meq) - Q_{u}(\Meq,\Meq) + Q_{u}(\Meq,\Meq) \\
	&= Q_{c}(\Meq,\Meq) - Q_{u}(\Meq,\Meq) \\
	&~~~ - \left(1 - \Pi^N_{2\Lv} \right) Q(\Chi \Meq, \Chi \Meq) - E(\Meq, \Meq),
	\end{align*}
	after using the expansion for $Q_u$ given in expression (\ref{Q_u_expansion}), with $E$ defined by (\ref{E_Q}) and noticing that $Q(\Meq,\Meq) = 0$.  So, by the triangle inequality,
	\begin{align*}
	\left\| Q_{c}(\Meq,\Meq) \right\|_{L^2(\OmegaLv)} &\leq \left\| Q_{c}(\Meq,\Meq) - Q_{u}(\Meq,\Meq) \right\|_{L^2(\OmegaLv)} \\
	&~~~ + \left\| \left(1 - \Pi^N_{2\Lv} \right) Q(\Chi \Meq, \Chi \Meq) \right\|_{L^2(\OmegaLv)} \\
	&~~~~ + \left\| E(\Meq, \Meq) \right\|_{L^2(\OmegaLv)}.
	\end{align*}
	Here, the first two terms can be bounded by (\ref{I1_bound_a}), after noting that the moments and $L^2$-norm of $\Meq$ are bounded.  Additionally, the term involving $E$ can be bounded by expression (\ref{Q_L2_estimate}) for the $L^2$-norm of $Q$, along with the calculations leading to (\ref{I2_bound}).  Combining these gives
	\begin{equation*}
	\left\| Q_{c}(\Meq,\Meq) \right\|_{L^2(\OmegaLv)} \leq O_{\frac{d}{2} + 2} + \frac{\mathcal{O} \left({\Lv}^{2\lambda + 2} \right)}{N^{\frac{d - 1}{2}}}
	\end{equation*}
	so that, by bounding the extension operator $\Chi$ by 1, 
	\begin{equation}
	\left\| \Chi Q_{c}(\Meq,\Meq) \right\|_{L^2(\OmegaLv)} \leq \left\| Q_{c}(\Meq,\Meq) \right\|_{L^2(\OmegaLv)} \leq O_{\frac{d}{2} + 2} + \frac{\mathcal{O} \left({\Lv}^{2\lambda + 2} \right)}{N^{\frac{d - 1}{2}}}. \label{Q_c_Meq_bound}
	\end{equation}
	
	The above arguments which led to (\ref{Q_c_Meq_bound}) can then applied to $Q_{c}(h,h)$.  Here, however, $Q(h,h) \neq Q(\hchi, \hchi) \neq 0$ and so the full expansion of $Q_u$ given by (\ref{Q_u_expansion}) appears.  Then, by using expression (\ref{Q_L2_estimate}) for the $L^2$-norm of $Q$ and following through the calculations that lead to (\ref{I3_bound}), this gives
	\begin{align}
	\left\|Q(h,h) - Q(\hchi, \hchi) \right\|_{L^2(\OmegaLv)} \leq& C_H (1 + C_{\chi})\eta(h_0) \widetilde{\zeta}(h_0 - \hchi_0) \nonumber \\
	&+ C_H (1 + C_{\chi})\eta(h_0) \|h - \hchi\|_{L^2(\OmegaLv)}. \label{Q_h_error_bound}
	\end{align}
	Note that this can be made sufficiently small for $\Chi$ chosen close enough to $\mathbbm{1}_{\OmegaLv}$ because then $\|h - \hchi\|_{L^2(\OmegaLv)}$ will be almost zero.  This means that the right-hand side of (\ref{Q_h_error_bound}) would contribute an $\mathcal{O}(\delchi)$ term, which is negligible and can be dropped for clarity.  In the end, this leads to
	\begin{equation}
	\left\| \Chi Q_{c}(h,h) - Q(\hchi, \hchi) \right\|_{L^2(\OmegaLv)} \leq O_{\frac{d}{2} + 2} + \frac{\mathcal{O} \left({\Lv}^{2\lambda + 2} \right)}{N^{\frac{d - 1}{2}}}. \label{Q_c-Q_bound}
	\end{equation}
	Similarly,
	\begin{equation}
	\left\| \Chi \mathcal{L}_{c}(h) - \mathcal{L}(\hchi) \right\|_{L^2(\OmegaLv)} \leq O_{\frac{d}{2} + 2} + \frac{\mathcal{O} \left({\Lv}^{2\lambda + 2} \right)}{N^{\frac{d - 1}{2}}}. \label{L_c-L_bound}
	\end{equation}
	
	So, by using the bounds (\ref{Q_c_Meq_bound}), (\ref{Q_c-Q_bound}) and (\ref{L_c-L_bound}), the remainder term $\mathcal{R}$ defined by (\ref{Remainder}) also satisfies
	\begin{equation}
	\left\| \mathcal{R}(h) \right\|_{L^2(\OmegaLv)} \leq O_{\frac{d}{2} + 2} + \frac{\mathcal{O} \left({\Lv}^{2\lambda + 2} \right)}{N^{\frac{d - 1}{2}}}. \label{Remainder_bound}
	\end{equation}
	In addition, by using expression (\ref{Q_L2_estimate}) for the $L^2$-norm of $Q$, 
	\begin{align}
	\|Q(\hchi, \hchi)\|_{L^2(\OmegaLv)} &\leq C_H \left(\|\hchi\|_{L^1_{\lambda + 2}(\OmegaLv)} + \|\hchi\|_{L^2(\OmegaLv)}\right) \|\hchi\|_{H^2_{\lambda + 2}(\OmegaLv)} \nonumber \\
	&\leq C_H C_{\chi} \eta(h_0) \left(\widetilde{\zeta}(h_0) + \|h\|_{L^2(\OmegaLv)} \right), \label{Q_chih_estimate}
	\end{align}
	for $\eta$, the bound on $\|h\|_{H^2_{\lambda + 2}(\OmegaLv)}$, defined by (\ref{eta_def}) in Theorem \ref{regularity_theorem} and $\widetilde{\zeta}$, the bound on $\|h\|_{L^1_{\lambda + 2}(\OmegaLv)} \leq m_{\lambda + 2}(h)$, defined by (\ref{zeta_def2}) in Theorem \ref{propagation_theorem}, because $h = g - \Meq$ solves the semi-discrete problem (\ref{semi-discrete2}) since $g$ does and $\Meq$ is the equilibrium solution.  The most important thing to notice here is that both $\eta(h_0) \to 0$ and $\widetilde{\zeta}(h_0) \to 0$ as $\|h_0\|_{H^2_{\lambda + 2}(\OmegaLv)} \to 0$, which can be seen in their definitions.
	\begin{rem}
		Once again, an estimate is actually required on the weighted norm $\|h\|_{H^2_{\lambda + 2}}$, which is expected to be true, as mentioned in Remark \ref{Hsk_remark}.  
	\end{rem}
	
	Now, following the notation used for the Boltzmann equation in \cite{BoltzmannConvergence}, let $e^{\mathcal{L}(t)}$ denote the semi-group associated to the linearised operator $\mathcal{L}$.  Then, the solution to the linearised equation (\ref{linear_Landau_eq}) can be written as
	\begin{equation}
	\hchi(t) = \hchi_0 + \int_{0}^{t} e^{\mathcal{L}(t-s)} Q(\hchi, \hchi)(s) ~\textrm{d}s + \int_{0}^{t} e^{\mathcal{L}(t-s)} \mathcal{R}(h)(s) ~\textrm{d}s, \label{semigroup_solution}
	\end{equation}
	for $h_0 := g_0 - \Meq$.
	
	Next, since the remainder $\mathcal{R}$ may not have zero mass, momentum and energy, it is necessary to introduce the projection operator $\pi$ onto the null-space of the Fokker-Planck-Landau equation.  As is stated in \cite{Carrapatoso_hard}, where the form of linearised operator defined by (\ref{L_def}) is used, the null-space $\mathcal{N(L)}$ is given by 
	\begin{equation*}
	\mathcal{N(L)} = \left\{\mu(\boldsymbol{v}), v_1 \mu(\boldsymbol{v}), \ldots, v_d \mu(\boldsymbol{v}), |\boldsymbol{v}|^2 \mu(\boldsymbol{v}) \right\},
	\end{equation*}
	where $\mu$ denotes the normalised Maxwellian with mass 1, zero bulk velocity and temperature $d$.  Then, for the set of collision invariants $\mathfrak{I} := \left\{1, v_1, \ldots, v_d, |\boldsymbol{v}|^2 \right\}$, the projection operator $\pi$ is defined by
	\begin{equation*}
	\pi h := \sum_{\phi \in \mathfrak{I}} \left(\int_{\Rd} h(\widetilde{\boldsymbol{v}}) \phi(\widetilde{\boldsymbol{v}}) ~\textrm{d}\widetilde{\boldsymbol{v}} \right)\phi(\boldsymbol{v}) \mu(\boldsymbol{v}).
	\end{equation*}
	Note that this projection operator satisfies
	\begin{align}
	\| \pi h \|_{L^2(\OmegaLv)} &\leq \sum_{\phi \in \mathfrak{I}} \left(\int_{\Rd} |h(\widetilde{\boldsymbol{v}}) \phi(\widetilde{\boldsymbol{v}})| ~\textrm{d}\widetilde{\boldsymbol{v}} \right) \left\| \phi \mu \right\|_{L^2(\OmegaLv)} \leq C^4_d \left\| h \right\|_{L^1_2(\Rd)}, \label{pi_bound}
	\end{align}
	for $C^4_d := (d+2) \left\| \mu \right\|_{L^2_2(\Rd)}$, where the weights on the norms come from the fact that $\phi(\boldsymbol{v}) \leq \langle \boldsymbol{v} \rangle^2$ for each $\phi$ in the sum.
	
	Then, since the semigroup $e^{\mathcal{L}(t)}$ and null-space projection operator commute, applying $(1 - \pi)$ to both sides of equation (\ref{semigroup_solution}) gives
	\begin{multline}
	(1 - \pi)\hchi(t) = (1 - \pi)\hchi_0 + \int_{0}^{t} e^{\mathcal{L}(t-s)} Q(\hchi, \hchi)(s) ~\textrm{d}s \\
	+ \int_{0}^{t} e^{\mathcal{L}(t-s)} (1 - \pi)\mathcal{R}(h)(s) ~\textrm{d}s, \label{semigroup_nullspace_identity}
	\end{multline}
	since $Q(\hchi, \hchi)$ is already conservative and so $(1 - \pi) Q(\hchi, \hchi) = Q(\hchi, \hchi)$.  Also, by properties of the semigroup, which are detailed in \cite{Carrapatoso_hard}, 
	\begin{equation*}
	\left\|e^{\mathcal{L}(t)} \right\|_{L^2(\OmegaLv)} \leq e^{-\lambda_0 t},
	\end{equation*}
	where $\lambda_0$ is the spectral gap that was first estimated for the Fokker-Planck-Landau type equation associated to hard potentials in \cite{Degond&Lemou}.  
	
	So, by applying the $L^2$-norm and using the triangle inequality, as well as this bound on the semigroup $e^{\mathcal{L}(t)}$,
	\begin{align}
	&\left\|(1 - \pi)\hchi(t) \right\|_{L^2(\OmegaLv)} \nonumber \\
	\leq& \left\|(1 - \pi)\hchi_0 \right\|_{L^2(\OmegaLv)} + \int_{0}^{t} e^{-\lambda_0 (t-s)} \left\|Q(\hchi, \hchi)(s) \right\|_{L^2(\OmegaLv)} ~\textrm{d}s \nonumber \\
	&+ \int_{0}^{t} e^{-\lambda_0 (t-s)} \left\|(1 - \pi)\mathcal{R}(h)(s) \right\|_{L^2(\OmegaLv)} ~\textrm{d}s \nonumber \\
	\leq& \left\|(1 - \pi)\hchi_0 \right\|_{L^2(\OmegaLv)} + \frac{1}{\lambda_0} \left(O_{\frac{d}{2} + 2} + \frac{\mathcal{O} \left({\Lv}^{2\lambda + 2} \right)}{N^{\frac{d - 1}{2}}} \right) \nonumber \\
	&+ \int_{0}^{t} e^{-\lambda_0 (t-s)} C_H C_{\chi} \eta(h_0) \left(\widetilde{\zeta}(h_0) + \|h(s)\|_{L^2(\OmegaLv)} \right) ~\textrm{d}s, \label{projected_h_bound1}
	\end{align}
	where the bounds (\ref{Q_chih_estimate}) and (\ref{Remainder_bound}) have been used and the anti-derivative of the exponential term used to introduce the reciprocal of the spectral gap.  Furthermore, since the conservation routine ensures that $\pi h(t) = 0$ for all $t \geq 0$,
	\begin{equation}
	\left\|\pi \hchi(t) \right\|_{L^2(\OmegaLv)} = \left\|\pi (1 - \Chi) h(t) \right\|_{L^2(\Rd)} \leq C^4_d \left\|(1 - \Chi) h(t) \right\|_{L^1_2(\OmegaLv)} = \mathcal{O}(\delchi), \label{pi_chi_h_bound}
	\end{equation}
	by using the bound (\ref{pi_bound}) and then noting that this final norm is an $\mathcal{O}(\delchi)$ term by definition of the extension operator in (\ref{chi_def}).
	
	As a result of the bound (\ref{pi_chi_h_bound}),
	\begin{equation}
	\left\|(1 - \pi) \Chi h(t) \right\|_{L^2(\OmegaLv)} = \left\|\hchi(t) \right\|_{L^2(\OmegaLv)} + \mathcal{O}(\delchi) = \left\|h(t) \right\|_{L^2(\OmegaLv)} + \mathcal{O}(\delchi). \label{projected_h_bound2}
	\end{equation}
	So, rearranging equation (\ref{projected_h_bound2}) and then using the bound (\ref{projected_h_bound1}) gives
	\begin{equation}
	\left\|h(t) \right\|_{L^2(\OmegaLv)} = \left\|(1 - \pi) \Chi h(t) \right\|_{L^2(\OmegaLv)} + \mathcal{O}(\delchi) \leq W(t) \label{h_W_bound}
	\end{equation}
	\begin{flalign}
	\textrm{for } && W(t) :=& \left\|h_0 \right\|_{L^2(\OmegaLv)} + \mathcal{O}(\delchi) + \frac{1}{\lambda_0} \left(O_{\frac{d}{2} + 2} + \frac{\mathcal{O} \left({\Lv}^{2\lambda + 2} \right)}{N^{\frac{d - 1}{2}}} \right) && \nonumber \\
	&& &~+ C_H C_{\chi} \eta(h_0) \int_{0}^{t} e^{-\lambda_0 (t-s)} \left(\widetilde{\zeta}(h_0) + \|h(s)\|_{L^2(\OmegaLv)} \right) ~\textrm{d}s, && \label{W_def}
	\end{flalign}
	where expression (\ref{projected_h_bound2}) has again been used on the $h_0$ term.
	
	Here, by using the product rule,
	\begin{multline}
	W'(t) = C_H C_{\chi} \eta(h_0) \biggl(\widetilde{\zeta}(h_0) + \|h(t)\|_{L^2(\OmegaLv)} \\
	- \lambda_0 \int_{0}^{t} e^{-\lambda_0 (t-s)} \left(\widetilde{\zeta}(h_0) + \|h(s)\|_{L^2(\OmegaLv)} \right) ~\textrm{d}s \biggr)
	\end{multline}
	and so
	\begin{align*}
	W'(t) + \lambda_0 W(t) =& ~C_H C_{\chi} \eta(h_0) \|h(t)\|_{L^2(\OmegaLv)} + C_H C_{\chi} \eta(h_0) \widetilde{\zeta}(h_0) \\
	&~+ \lambda_0 \left\|h_0 \right\|_{L^2(\OmegaLv)} + \lambda_0 \mathcal{O}(\delchi) + O_{\frac{d}{2} + 2} + \frac{\mathcal{O} \left({\Lv}^{2\lambda + 2} \right)}{N^{\frac{d - 1}{2}}} \\
	=& ~C_H C_{\chi} \eta(h_0) \|h(t)\|_{L^2(\OmegaLv)} + C_H C_{\chi} \eta(h_0) \widetilde{\zeta}(h_0) + W(0),
	\end{align*}
	by noting that the definition of $W$ in (\ref{W_def}) gives
	\begin{equation}
	W(0) = \left\|h_0 \right\|_{L^2(\OmegaLv)} + \mathcal{O}(\delchi) + \frac{1}{\lambda_0} \left(O_{\frac{d}{2} + 2} + \frac{\mathcal{O} \left({\Lv}^{2\lambda + 2} \right)}{N^{\frac{d - 1}{2}}} \right). \label{W0_def}
	\end{equation}
	
	This means, since $\|h(t)\|_{L^2(\OmegaLv)} \leq Y(t)$ from inequality (\ref{h_W_bound}) and then rearranging, this gives the ODI
	\begin{align}
	W'(t) \leq& \left(C_H C_{\chi} \eta(h_0) - \lambda_0 \right)W(t)  + C_H C_{\chi} \eta(h_0) \widetilde{\zeta}(h_0) + W(0). \label{W_ODI1}
	\end{align}
	At this stage, for the ODI (\ref{W_ODI1}) to be useful, the linear term must be negative to control the derivative.  By recalling the definition of $\eta$ in (\ref{eta_def}), this linear term can be made negative by noting that $\eta(h_0)$ is controlled by $\|h_0\|_{H^2_{\lambda + 2}(\OmegaLv)}$.  More precisely, by assumption (\ref{long_time_assumption1}), the initial perturbation $h_0$ has weighted Sobolev norm $\|h_0\|_{H^2_{\lambda + 2}(\OmegaLv)} \leq \delta_{\eta}$, for $\delta_{\eta} > 0$ small enough to cause
	\begin{equation}
	C_H C_{\chi} \eta(h_0) \leq \frac{1}{2} \lambda_0. \label{spectral_gap_bound}
	\end{equation}
	In this case, the ODI (\ref{W_ODI1}) can be written as
	\begin{align}
	W'(t) \leq& - \frac{1}{2}\lambda_0 W(t) + \frac{1}{2}\lambda_0 \widetilde{\zeta}(h_0) + \lambda_0 W(0). \label{W_ODI2}
	\end{align}
	\begin{rem}
		In the definition (\ref{eta_def}) of $\eta$, only the unweighted Sobolev norm $\|h_0\|_{H^2}$ appears.  Since this bound is actually being used to control the weighted Sobolev norms, as per Remark \ref{Hsk_remark}, it is conjectured that the definition of $\eta$ for this purpose would also have to include the weighted norm $\|h_0\|_{H^2_{\lambda + 2}}$.  
	\end{rem}
	
	Then, by applying the Gronwall inequality for linear ODIs to (\ref{W_ODI2}) and using the inequality (\ref{h_W_bound}),
	\begin{align}
	\|h(t)\|_{L^2(\OmegaLv)} \leq W(t) \leq& ~e^{- \frac{1}{2} \lambda_0 t} \Biggl(W(0) - \Biggl(\widetilde{\zeta}(h_0) + 2 W(0) \Biggr)\Biggr) + \widetilde{\zeta}(h_0) + 2 W(0) \nonumber \\
	=& ~\widetilde{\zeta}(h_0) + 2 W(0) - \left(\widetilde{\zeta}(h_0) + W(0) \right)e^{- \frac{1}{2} \lambda_0 t} \nonumber \\
	\leq& ~C_{\zeta} \|h_0\|_{L^2(\OmegaLv)} + 2W(0), \label{h_bound}
	\end{align}
	because $\widetilde{\zeta}(h_0) \leq C_{\zeta} \|h_0\|_{L^2(\OmegaLv)}$ from (\ref{C_zeta_def}) and $\left(\widetilde{\zeta}(h_0) + W(0) \right)e^{- \frac{1}{2} \lambda_0 t} \geq 0$, where $W(0)$ is given by (\ref{W0_def}).
	
	Finally, given any $\delta > 0$, when the initial data satisfies assumptions (\ref{long_time_assumption1}) and (\ref{long_time_assumption2}), so that
	\begin{equation*}
	\|h_0\|_{L^2(\OmegaLv)} \leq \|h_0\|_{H^2_{\lambda + 2}(\OmegaLv)} \leq \min \left(\delta_{\eta}, \frac{1}{2 C_{\zeta}} \delta, \frac{1}{8}\delta \right),
	\end{equation*}
	and the domain size $\Lv$, number of Fourier modes $N$ and extension function $\Chi$ are chosen such that
	\begin{equation*}
	\mathcal{O}(\delchi) + \frac{1}{\lambda_0} \left(O_{\frac{d}{2} + 2} + \frac{\mathcal{O} \left({\Lv}^{2\lambda + 2} \right)}{N^{\frac{d - 1}{2}}} \right) \leq \frac{1}{8} \delta,
	\end{equation*}
	by definition of $W(0)$ in (\ref{W0_def}),
	\begin{equation*}
	W(0) \leq \frac{1}{8} \delta + \frac{1}{8} \delta = \frac{1}{4} \delta.
	\end{equation*}
	Therefore, by the bound (\ref{h_bound}), the perturbation $h = g - \Meq$ satisfies
	\begin{align*}
	\|h(t)\|_{L^2(\OmegaLv)} \leq ~C_{\zeta} \left(\frac{1}{2 C_{\zeta}} \delta \right) + 2 \left(\frac{1}{4} \delta \right) \leq \delta,
	\end{align*}
	for all time $t > 0$, which is the required result stated in Lemma \ref{long_time_lemma}.
\end{proof}

\section{Conclusion}
A mathematical analysis was presented to prove an estimate in $L^2$ spaces on the error between the approximate solution obtained by the conservative spectral method for solving space-homogeneous FPL type equations associated to hard potentials and the corresponding true solution.  This is the first time an error estimate has been derived for a direct method to FPL type equations, with any range of potentials.  

In order to obtain such a result, it was also shown that the there is always a unique solution to the conservative spectral method.  For a large enough cut-off velocity domain and a sufficient number of Fourier modes, the moments and $L^2$-norm of this solution also remain bounded.  Furthermore, regularity of the solution was also shown to propagate by proving that its derivatives remain bounded in $L^2$ spaces, up to any order.

\begin{appendices}
\section{Results for the Landau Equation}
\index{Appendix!Results for the Landau Equation @\emph{Results for the Landau Equation}}%

\subsection{The Weak Form}
\begin{lemma} \label{Weak_Landau_Lemma}
	If $f$ is a function such that $f(\boldsymbol{v}) \to 0$ as $|\boldsymbol{v}| \to 0$, the weak form of the operator $Q^a$ is
	\begin{equation}
	\int_{\mathbb{R}^d} Q(f,g) \phi ~\textrm{d}\boldsymbol{v} = \int_{\mathbb{R}^d} f \left(\bar{a}_{i,j} \frac{\partial^2 \phi}{\partial v_i \partial v_j} + 2 \bar{b_i} \frac{\partial \phi}{\partial v_i} \right) ~\textrm{d}\boldsymbol{v}. \label{Q_weak_form}
	\end{equation}
\end{lemma}

\begin{proof}
	To make the notation a little easier, define the vector $\bar{\boldsymbol{a}}_i$ as the $i$th row of the matrix $\bar{a}$ so that $\bar{\boldsymbol{a}}_i = (\bar{a}_{i,1}, \ldots, \bar{a}_{i,d})$, for $i = 1, \ldots, d$.  Then,
	\begin{align*}
	\int_{\mathbb{R}^d} \bar{a}_{i,j} \frac{\partial^2 f}{\partial v_i \partial v_j} \phi ~\textrm{d}\boldsymbol{v} &= \sum_{i = 1}^{d} \int_{\mathbb{R}^d} \nabla\left(\frac{\partial f}{\partial v_i}\right) \cdot \phi \bar{\boldsymbol{a}}_i ~\textrm{d}\boldsymbol{v} \\
	&= \sum_{i = 1}^{d} \lim\limits_{R \to \infty} \int_{\Omega_R} \nabla\left(\frac{\partial f}{\partial v_i}\right) \cdot \phi \bar{\boldsymbol{a}}_i ~\textrm{d}\boldsymbol{v} \\
	&= \sum_{i = 1}^{d} \lim\limits_{R \to \infty} \left(\int_{\partial \Omega_R} \frac{\partial f}{\partial v_i} \phi \bar{\boldsymbol{a}}_i \cdot \boldsymbol{n} ~\textrm{d}s + \int_{\Omega_R} \frac{\partial f}{\partial v_i} \nabla \cdot (\phi \bar{\boldsymbol{a}}_i) ~\textrm{d}\boldsymbol{v} \right) \\
	&= \sum_{i = 1}^{d} \lim\limits_{R \to \infty}\int_{\Omega_R} \frac{\partial f}{\partial v_i} \left(\nabla \phi \cdot \bar{\boldsymbol{a}}_i + \phi \nabla \cdot \bar{\boldsymbol{a}}_i \right) ~\textrm{d}\boldsymbol{v},
	\end{align*}
	where the boundary integral evaluates to zero as the limit as $R \to \infty$ is taken as it must also be that $\frac{\partial f}{\partial v_i} \to 0$ as $|\boldsymbol{v}| \to \infty$.  
	
	Now, note that $\nabla \cdot \bar{\boldsymbol{a}}_i = \frac{\partial}{\partial v_j}(a_{i,j}) = b_i$ and 
	\begin{equation*}
	\frac{\partial f}{\partial v_i} \nabla \phi \cdot \bar{\boldsymbol{a}}_i = \frac{\partial f}{\partial v_i} \frac{\partial \phi}{\partial v_j} \bar{a}_{i,j} = \frac{\partial f}{\partial v_i} \left(\bar{a} \nabla \phi\right)_i = \nabla f \cdot \bar{a} \nabla \phi.
	\end{equation*}
	So,
	\begin{align*}
	\int_{\mathbb{R}^d} \bar{a}_{i,j} \frac{\partial^2 f}{\partial v_i \partial v_j} \phi ~\textrm{d}\boldsymbol{v} &= \lim\limits_{R \to \infty}\int_{\Omega_R} \nabla f \cdot \left(\phi \bar{\boldsymbol{b}} + \bar{a} \nabla \phi \right) ~\textrm{d}\boldsymbol{v} \\
	& = \lim\limits_{R \to \infty} \left(\int_{\partial \Omega_R} f \left(\phi \bar{\boldsymbol{b}} + \bar{a} \nabla \phi \right) \cdot \boldsymbol{n} ~\textrm{d}s \right. \\
	&~~~~~~~~~~~~ + \int_{\Omega_R} f \biggl(\nabla \phi \cdot \bar{\boldsymbol{b}} + \phi \nabla \cdot \bar{\boldsymbol{b}} + \frac{\partial \bar{a}_{i,j}}{\partial v_i} \frac{\partial \phi}{\partial v_j} \\
	&~~~~~~~~~~~~~~~~~~~~~~~~~~~~~~~~~~~~~~~~~~~~~~ + \left. \bar{a}_{i,j} \frac{\partial^2 \phi}{\partial v_i \partial v_j} \biggr) ~\textrm{d}\boldsymbol{v} \right) \\
	& = \int_{\mathbb{R}^d} f \left(\bar{a}_{i,j} \frac{\partial^2 \phi}{\partial v_i \partial v_j} + 2 \bar{b_i} \frac{\partial \phi}{\partial v_i} + \bar{c} \phi \right) ~\textrm{d}\boldsymbol{v},
	\end{align*}
	since $\nabla \phi \cdot \bar{\boldsymbol{b}} = \bar{b_i} \frac{\partial \phi}{\partial v_i}$, $\nabla \cdot \bar{\boldsymbol{b}} = \bar{c}$ and the boundary term disappears again as the limit $R \to \infty$ is taken.
	
	Therefore, 
	\begin{flalign*}
	&& \int_{\mathbb{R}^d} Q(f,g) \phi ~\textrm{d}\boldsymbol{v} &= \int_{\mathbb{R}^d} \bar{a}_{i,j} \frac{\partial^2 f}{\partial v_i \partial v_j} \phi ~\textrm{d}\boldsymbol{v} - \int_{\mathbb{R}^d} \bar{c} \phi ~\textrm{d}\boldsymbol{v} && \\
	&& &= \int_{\mathbb{R}^d} f \left(\bar{a}_{i,j} \frac{\partial^2 \phi}{\partial v_i \partial v_j} + 2 \bar{b_i} \frac{\partial \phi}{\partial v_i} \right) ~\textrm{d}\boldsymbol{v}, && \textrm{as required.}
	\end{flalign*}
\end{proof}


\subsection{Collision Operator Decay Estimate}
\begin{lemma} \label{Q_decay}
	If $f$ is a function such that $f(\boldsymbol{v}) \to 0$ as $|\boldsymbol{v}| \to \infty$ and $f(\boldsymbol{v}) = 0$ when $\boldsymbol{v} \in \Omega_{L_{\boldsymbol{v}}}$ then, for any $\lambda \in (0, 1]$ and $k' \geq 2$,
	\begin{align}
	\left| \int_{\mathbb{R}^d \backslash \Omega_{L_{\boldsymbol{v}}}} Q(f,g) ~\textrm{d}\boldsymbol{v} \right| &\leq O_{k'} \Bigl(m_{0}(g) m_{k' + \lambda}(f) + m_{\lambda}(g) m_{k'}(f) \nonumber \\
	&~~~~~~~~~~~ + m_{2}(g) m_{k' + \lambda - 2}(f) \nonumber \\
	&~~~~~~~~~~~~~~+ m_{\lambda + 2}(g) m_{k' - 2}(f)\Bigr) \label{Q_bound}
	\end{align}
\end{lemma}

\begin{proof}
	First, for any $k' \geq 0$
	\begin{align}
	\left| \int_{\mathbb{R}^d \backslash \Omega_{L_{\boldsymbol{v}}}} Q(f,g) ~\textrm{d}\boldsymbol{v} \right| &= \left| \int_{\mathbb{R}^d \backslash \Omega_{L_{\boldsymbol{v}}}} Q(f,g) |\boldsymbol{v}|^{-k'} |\boldsymbol{v}|^{k'} ~\textrm{d}\boldsymbol{v} \right| \nonumber \\
	&\leq (\sqrt{d} L_{\boldsymbol{v}})^{-k'} \left| \int_{\mathbb{R}^d \backslash \Omega_{L_{\boldsymbol{v}}}} Q(f,g) |\boldsymbol{v}|^{k'} ~\textrm{d}\boldsymbol{v} \right| \label{Q_a/c_bound}
	\end{align}
	Now, by using identity (\ref{Q^a_weak}) for the weak form of $Q$,
	\begin{align*}
	&\left| \int_{\mathbb{R}^d \backslash \Omega_{L_{\boldsymbol{v}}}} Q(f,g) |\boldsymbol{v}|^{k'} ~\textrm{d}\boldsymbol{v} \right| \\
	=~ & \left| \int_{\mathbb{R}^d \backslash \Omega_{L_{\boldsymbol{v}}}} f \left(\bar{a}_{i,j} \frac{\partial^2}{\partial v_i \partial v_j}\left(|\boldsymbol{v}|^{k'} \right) + 2 \bar{b_i} \frac{\partial}{\partial v_i}\left(|\boldsymbol{v}|^{k'} \right) \right) ~\textrm{d}\boldsymbol{v} \right|. \\
	=~ & \left| k' \int_{\mathbb{R}^d \backslash \Omega_{L_{\boldsymbol{v}}}} \int_{\mathbb{R}^d} f(\boldsymbol{v}) g(\boldsymbol{v}_*) |\boldsymbol{v} - \boldsymbol{v}_*|^{\lambda} |\boldsymbol{v}|^{k' - 2} \biggl(-2|\boldsymbol{v}|^{2} + 2|\boldsymbol{v}_*|^{2} \nonumber \right. \\
	&~~~~~~~~~~~~~~~~~~~~~~~~~~~~~~~~~~~~~ \left. + (k' - 2) \frac{|\boldsymbol{v}|^{2} |\boldsymbol{v}_*|^{2} - (\boldsymbol{v} \cdot \boldsymbol{v}_*)^2}{|\boldsymbol{v}|^{2}} \biggr) ~\textrm{d}\boldsymbol{v}_* ~\textrm{d}\boldsymbol{v} \right|,
	\end{align*}
	which used an identity similar to one derived in Desvillette and Villani's proof of moment propagation in Theorem $3(i)$ in \cite{D&V_1}, where they use $\phi(\boldsymbol{v}) = (1 + |\boldsymbol{v}|^2)^{\frac{s}{2}}$ for $s > 2$ instead, so the details can be found there (and it can also be checked that their proof up to this result does still work for $s = 2$).  Then, by bringing the absolute values inside the integral and expanding the domain of integration for $\boldsymbol{v}$ to all of $\mathbb{R}^d$,
	\begin{align}
	&\left| \int_{\mathbb{R}^d \backslash \Omega_{L_{\boldsymbol{v}}}} Q(f,g) |\boldsymbol{v}|^{k'} ~\textrm{d}\boldsymbol{v} \right| \nonumber \\
	\leq~ & k' \int_{\mathbb{R}^d} \int_{\mathbb{R}^d} |f(\boldsymbol{v})\|g(\boldsymbol{v}_*)\|\boldsymbol{v} - \boldsymbol{v}_*|^{\lambda} |\boldsymbol{v}|^{k' - 2} \biggl(2|\boldsymbol{v}|^{2} + 2|\boldsymbol{v}_*|^{2} \nonumber \\
	&~~~~~~~~~~~~~~~~~~~~~~~~~~~~~~~~~~~~~~~~~~ + (k' - 2) \frac{|\boldsymbol{v}|^{2} |\boldsymbol{v}_*|^{2} - (\boldsymbol{v} \cdot \boldsymbol{v}_*)^2}{|\boldsymbol{v}|^{2}} \biggr) ~\textrm{d}\boldsymbol{v}_* \textrm{d}\boldsymbol{v}, \label{D&V_identity}
	\end{align}
	where it has also been used that $k' \geq 2$ and that $|\boldsymbol{v}|^{2} |\boldsymbol{v}_*|^{2} - (\boldsymbol{v} \cdot \boldsymbol{v}_*)^2 \geq 0$ by the Cauchy-Schwarz inequality.  Furthermore, since $(\boldsymbol{v} \cdot \boldsymbol{v}_*)^2 \geq 0$,
	\begin{equation}
	\frac{|\boldsymbol{v}|^{2} |\boldsymbol{v}_*|^{2} - (\boldsymbol{v} \cdot \boldsymbol{v}_*)^2}{|\boldsymbol{v}|^{2}} \leq \frac{|\boldsymbol{v}|^{2} |\boldsymbol{v}_*|^{2}}{|\boldsymbol{v}|^{2}} = |\boldsymbol{v}_*|^2 \label{Q^a_quotient_bound},
	\end{equation}
	and using this in (\ref{D&V_identity}) along with $|\boldsymbol{v} - \boldsymbol{v}_*|^{\lambda} \leq |\boldsymbol{v}|^{\lambda} + |\boldsymbol{v}_*|^{\lambda}$ gives
	\begin{align*}
	&\left| \int_{\mathbb{R}^d \backslash \Omega_{L_{\boldsymbol{v}}}} Q(f,g) |\boldsymbol{v}|^{k'} ~\textrm{d}\boldsymbol{v} \right| \\
	\leq~ & k' \int_{\mathbb{R}^d} \int_{\mathbb{R}^d} |f(\boldsymbol{v})\|g(\boldsymbol{v}_*)| \left(|\boldsymbol{v}|^{\lambda} + |\boldsymbol{v}_*|^{\lambda} \right) |\boldsymbol{v}|^{k' - 2} \left(2|\boldsymbol{v}|^{2} + k' |\boldsymbol{v}_*|^{2} \right) ~\textrm{d}\boldsymbol{v}_* \textrm{d}\boldsymbol{v}, \\
	=~ & k' \int_{\mathbb{R}^d} \int_{\mathbb{R}^d} |f(\boldsymbol{v})\|g(\boldsymbol{v}_*)| \Bigl(2|\boldsymbol{v}|^{k' + \lambda} + 2|\boldsymbol{v}|^{k'}|\boldsymbol{v}_*|^{\lambda} + k' |\boldsymbol{v}|^{k' + \lambda - 2}|\boldsymbol{v}_*|^{2} \\
	&~~~~~~~~~~~~~~~~~~~~~~~~~~~~~~~~~~~~~~~~~~~~~~~~~~~+ k' |\boldsymbol{v}|^{k' - 2}|\boldsymbol{v}_*|^{\lambda + 2} \Bigr) ~\textrm{d}\boldsymbol{v}_* \textrm{d}\boldsymbol{v}, \\
	=~ & 2 k' \Bigl(m_{0}(g) m_{k' + \lambda}(f) + m_{\lambda}(g) m_{k'}(f) \Bigr) \\
	&~~ + k'^2 \Bigl(m_{2}(g) m_{k' + \lambda - 2}(f) + m_{\lambda + 2}(g) m_{k' - 2}(f)\Bigr).
	\end{align*}
	Therefore, using this last expression in (\ref{Q_a/c_bound}) gives 
	\begin{align*}
	\left| \int_{\mathbb{R}^d \backslash \Omega_{L_{\boldsymbol{v}}}} Q(f,g) ~\textrm{d}\boldsymbol{v} \right| &= \left| \int_{\mathbb{R}^d \backslash \Omega_{L_{\boldsymbol{v}}}} Q(f,g) |\boldsymbol{v}|^{-k'} |\boldsymbol{v}|^{k'} ~\textrm{d}\boldsymbol{v} \right| \nonumber \\
	&\leq k' (\sqrt{d} L_{\boldsymbol{v}})^{-k'} \biggl( 2 \Bigl(m_{0}(g) m_{k' + \lambda}(f) + m_{\lambda}(g) m_{k'}(f) \Bigr) \\
	&~~~~~~~~~~~~~~~ + k' \Bigl(m_{2}(g) m_{k' + \lambda - 2}(f) + m_{\lambda + 2}(g) m_{k' - 2}(f)\Bigr) \biggr) 
	\end{align*}
	which is the result (\ref{Q_bound}), as required.
\end{proof}

\subsection{Lower Bound on Moments of Cut-off}
\begin{lemma} \label{moment_cutoff_lemma}
	Given some $\varepsilon_{\chi} \in (0,1)$, any solution $g(t, \boldsymbol{v})$ to the semi-discrete problem (\ref{semi-discrete2}), for which there exist constants $C \geq 1$ and $r \in (0, 1]$ such that
	\begin{equation}
	|g(t, \boldsymbol{v})| \leq \frac{C \rho_0}{(2\pi T_{0})^{\frac{d}{2}}} e^{-\frac{r |\boldsymbol{v}|^2}{2 T_{0}}} \label{moment_bound_assumption}
	\end{equation}
	\begin{flalign*}
	\textrm{where} && \rho_0 = \int_{\mathbb{R}^d} g_0(\boldsymbol{v}) ~\textrm{d}\boldsymbol{v} ~~~\textrm{and}~~~ T_0 = \frac{1}{3\rho_0} \int_{\mathbb{R}^d} g_0(\boldsymbol{v}) |\boldsymbol{v}|^2 ~\textrm{d}\boldsymbol{v}, &&
	\end{flalign*}
	satisfies
	\begin{flalign}
	(i) ~~~m_0(\mbox{\Large$\chi$}g) &\geq (1 - \varepsilon_{\chi}) m_0(g_0), && \label{moment_cutoff_bound_0} \\
	(ii) ~~~m_k(\mbox{\Large$\chi$}g) &\geq (1 - \varepsilon_{\chi}) m_k(g), \textrm{for any } k \geq 0. \label{moment_cutoff_bound_k} &&
	\end{flalign}
\end{lemma}

\begin{proof}
	By adding and subtracting the $k$-th moment of $g$ to that of \mbox{\Large$\chi$}$g$
	\begin{equation*}
	m_k(\mbox{\Large$\chi$}g) = m_k(g) - \int_{\mathbb{R}^d} (1 - \mbox{\Large$\chi$}) |g(\boldsymbol{v})| \langle \boldsymbol{v} \rangle^{k} ~\textrm{d}\boldsymbol{v}.
	\end{equation*}
	Then, by noting that \mbox{\Large$\chi$}$(\boldsymbol{v}) = 1$ when $\boldsymbol{v} \in \Omega_{(1-\delta_{\chi})L_{\boldsymbol{v}}}$ and $0 \leq$ \mbox{\Large$\chi$}$(\boldsymbol{v}) < 1$ otherwise,
	\begin{align*}
	\int_{\mathbb{R}^d} (1 - \mbox{\Large$\chi$}) |g(\boldsymbol{v})| \langle \boldsymbol{v} \rangle^{k} ~\textrm{d}\boldsymbol{v} & = \int_{\mathbb{R}^d \backslash \Omega_{(1-\delta_{\chi})L_{\boldsymbol{v}}}} (1 - \mbox{\Large$\chi$}) |g(\boldsymbol{v})| \langle \boldsymbol{v} \rangle^{k} ~\textrm{d}\boldsymbol{v} \\
	& \leq \int_{\mathbb{R}^d \backslash \Omega_{(1-\delta_{\chi})L_{\boldsymbol{v}}}} |g(\boldsymbol{v})| \langle \boldsymbol{v} \rangle^{k} ~\textrm{d}\boldsymbol{v} \\
	& \leq \int_{\mathbb{R}^d \backslash \Omega_{(1-\delta_{\chi})L_{\boldsymbol{v}}}} \frac{C \rho_0}{(2\pi T_{0})^{\frac{d}{2}}} e^{-\frac{r |\boldsymbol{v}|^2}{2 T_{0}}} \langle \boldsymbol{v} \rangle^{k} ~\textrm{d}\boldsymbol{v} \\
	& = \tau_k \rho_0,
	\end{align*}
	\begin{flalign*}
	\textrm{for } && \tau_k = \int_{\mathbb{R}^d \backslash \Omega_{(1-\delta_{\chi})L_{\boldsymbol{v}}}} \frac{C}{(2\pi T_{0})^{\frac{d}{2}}} e^{-\frac{r |\boldsymbol{v}|^2}{2 T_{0}}} \langle \boldsymbol{v} \rangle^{k} ~\textrm{d}\boldsymbol{v} &&
	\end{flalign*}
	where the domain $\Omega_{(1-\delta_{\chi})L_{\boldsymbol{v}}}$ has $L_{\boldsymbol{v}}$ chosen large enough and $\delta_{\chi}$ small enough so that $\tau_k \leq \varepsilon_{\chi}$ for any required value of $k \geq 0$, due to the Gaussian decay being so much faster than any polynomial can increase.
	
	\begin{flalign*}
	\textrm{Also, } && \rho_0 = \int_{\mathbb{R}^d} g_0(\boldsymbol{v}) ~\textrm{d}\boldsymbol{v} \leq \int_{\mathbb{R}^d} |g_0(\boldsymbol{v})| ~\textrm{d}\boldsymbol{v} = m_0(g_0), &&
	\end{flalign*}
	which means that $\int_{\mathbb{R}^d} (1 - \mbox{\Large$\chi$}) |g(\boldsymbol{v})| \langle \boldsymbol{v} \rangle^{k} ~\textrm{d}\boldsymbol{v} \leq \varepsilon_{\chi} m_0(g_0)$ and so, for $k = 0$,
	\begin{equation*}
	m_0(\mbox{\Large$\chi$}g) \geq m_0(g) - \varepsilon_{\chi} m_0(g) = (1 - \varepsilon_{\chi}) m_0(g_0),
	\end{equation*}
	which is result (\ref{moment_cutoff_bound_0}).
	
	Furthermore, since $g$ is a solution of equation (\ref{semi-discrete2}) which conserves mass,
	\begin{equation*}
	\rho_0 = \int_{\mathbb{R}^d} g_0(\boldsymbol{v}) ~\textrm{d}\boldsymbol{v} = \int_{\mathbb{R}^d} g(\boldsymbol{v}) ~\textrm{d}\boldsymbol{v} \leq \int_{\mathbb{R}^d} |g(\boldsymbol{v})| ~\textrm{d}\boldsymbol{v} \leq \int_{\mathbb{R}^d} |g(\boldsymbol{v})| \langle \boldsymbol{v} \rangle^{k} ~\textrm{d}\boldsymbol{v} = m_k(g).
	\end{equation*}
	This means that $\int_{\mathbb{R}^d} (1 - \mbox{\Large$\chi$}) |g(\boldsymbol{v})| \langle \boldsymbol{v} \rangle^{k} ~\textrm{d}\boldsymbol{v} \leq \varepsilon_{\chi} m_k(g)$ and therefore
	\begin{equation*}
	m_k(\mbox{\Large$\chi$}g) \geq m_k(g) - \varepsilon_{\chi} m_k(g) = (1 - \varepsilon_{\chi}) m_k(g),
	\end{equation*}
	which is result (\ref{moment_cutoff_bound_k}).
	
\end{proof}

\section{Additional Mathematical Results}
\index{Appendix!Additional Mathematical Results @\emph{Additional Mathematical Results}}%

\subsection{A Sobolev Embedding Result} \label{Sobolev_result}
\begin{lemma}
	For $\Omega \subset \Rd$ and any $1 \leq p < d$, given a function $h \in \dot{H}^1(\Omega)$, there exists some constant $C^S_{p, d}$ that depends on $p$ and the dimension $d$ but is independent of $\Omega$ such that 
	\begin{equation*}
	\| h \|_{L^{\frac{dp}{d - p}}(\Omega)} \leq C^S_{p, d} \| h \|_{\dot{H}^1(\Omega)}.
	\end{equation*}
\end{lemma}

\subsection{A Fourier Approximation Estimate}
\begin{lemma} \label{Fourier_approx_lemma}
	Given $s \in \mathbb{N}$ and $\OmegaLv = [-\Lv, \Lv]^d$, for $\Lv > 0$ and $d \in \mathbb{N}$, if $f \in H^s(\OmegaLv)$ then the remainder of a partial Fourier series of $f$ satisfies
	\begin{equation*}
	\left\|\left(1 - \Pi^N_{\Lv} \right) f \right\|_{L^2(\OmegaLv)} \leq \frac{1}{\left(2\pi \right)^{\frac{d}{2}}} \left(\frac{\Lv}{2\pi N} \right)^s \left\|f \right\|_{H^s(\OmegaLv)}.
	\end{equation*}
\end{lemma}

\subsection{A Nonlinear Gronwall Inequality}
\begin{lemma} \label{Gronwall_lemma}
	Given an ODI for $u(t)$, defined for $0 < t < T$, with constant coefficients $A, B \in \mathbb{R}$ of the form
	\begin{flalign*}
	&& ~~~~~~~~~~~~~~~\frac{d u}{d t} \leq A u(t) + B (u(t))^{\alpha}, && \textrm{with } 0 \leq \alpha < 1,
	\end{flalign*}
	if $\frac{B}{A} \geq 0$ then the solution satisfies
	\begin{equation*}
	u(t) \leq e^{At} \Biggl(\bigl(u(0) \bigr)^{1-\alpha} + \frac{B}{A} \Biggr)^{\frac{1}{1-\alpha}}.
	\end{equation*}
\end{lemma}

\begin{proof}
	This results follows directly from Theorem 21 in \cite{Gronwall} and then using the fact that the coefficients are constant.  First, by that result, the solution satisfies
	\begin{equation*}
	u(t) \leq \Biggl(\bigl(u(0) \bigr)^{1-\alpha} e^{\left((1-\alpha) \int_{0}^{t} A ~\textrm{d}s \right)} + (1-\alpha) \int_{0}^{t} B e^{\left((1-\alpha) \int_{s}^{t} A ~\textrm{d}r \right)} ~\textrm{d}s\Biggr)^{\frac{1}{1-\alpha}}.
	\end{equation*}
	So, by evaluating the integrals in the case of constant coefficients $A$ and $B$,
	\begin{align*}
	u(t) &\leq \Biggl(\bigl(u(0) \bigr)^{1-\alpha} e^{\left((1-\alpha) At \right)} + (1-\alpha) B e^{\left((1-\alpha) At \right)} \int_{0}^{t} e^{-\left((1-\alpha) As \right)} ~\textrm{d}s \Biggr)^{\frac{1}{1-\alpha}} \\
	&= \Biggl(\bigl(u(0) \bigr)^{1-\alpha} e^{\left((1-\alpha) At \right)} + \frac{B}{A} e^{\left((1-\alpha) At \right)} \left(1 - e^{-\left((1-\alpha) At \right)} \right) \Biggr)^{\frac{1}{1-\alpha}} \\
	&\leq e^{At} \Biggl(\bigl(u(0) \bigr)^{1-\alpha} + \frac{B}{A} \Biggr)^{\frac{1}{1-\alpha}}, ~~~~~~~~~~~~~~~~~~~~~~~~~~~~\textrm{if } \frac{A}{B} \geq 0, ~~\textrm{as required.}
	\end{align*}
\end{proof}

\end{appendices}

\section{ACKNOWLEDGMENTS}
The authors would like to thank Chenglong Zhang for help in understanding the code used to implement the conservative spectral method and being available to give advice on any developments.  Karl Schulz has also been extremely helpful in understanding modern high performance computing techniques used to improve the code structure.  The Institute for Computational Engineering Sciences has been incredibly supportive as well.  Both authors have been  partially funded by grants DMS-RNMS-1107291 (Ki-Net),  NSF DMS1715515 and  DOE DE-SC0016283 project \emph{Simulation Center for Runaway Electron Avoidance and Mitigation}.  This manuscript is based on a section of Clark Pennie Ph.D.'s Thesis Dissertation \cite{clark_pennie_phd_thesis}, written at The University of Texas at Austin, under the advising of the second author.

\section*{REFERENCES}


\bibliographystyle{acm}
\bibliography{Refs}
\end{document}